\documentclass[letterpaper, 11pt,  reqno]{amsart}

\usepackage{amsmath,amssymb,amscd,amsthm,amsxtra, esint, bbm, enumerate, relsize, xcolor}
\usepackage{mathtools} 
\usepackage{skak, bbm} 
\usepackage{accents, cancel}
\usepackage{ulem}
\usepackage[implicit=true]{hyperref}

\usepackage{stmaryrd}

\usepackage{mparhack}

\allowdisplaybreaks[2]

\newtheorem{theorem}{Theorem} [section]

\newtheorem{lemma}[theorem]{Lemma}
\newtheorem{proposition}[theorem]{Proposition}

\newtheorem{corollary}[theorem]{Corollary}

\theoremstyle{definition}
\newtheorem{remark}[theorem]{Remark}

\DeclareMathOperator*{\intt}{\int}

\DeclareMathOperator{\Law}{Law}

\newcommand{\I}{\hspace{0.5mm}\text{I}\hspace{0.5mm}}

\newcommand{\noi}{\noindent}
\newcommand{\Z}{\mathbb{Z}}
\newcommand{\R}{\mathbb{R}}
\newcommand{\C}{\mathbb{C}}
\newcommand{\T}{\mathbb{T}}

\newcommand{\wick}[1]{{:}\,#1\mspace{2mu}{:}}

\let\P= \undefined
\newcommand{\P}{\mathbb{P}}

\newcommand{\E}{\mathbb{E}}

\newcommand{\Res}{\mathcal{R}}

\newcommand{\al}{\alpha}

\newcommand{\dl}{\delta}

\newcommand{\FL}{\mathcal{F}L} 

\newcommand{\ep}{\varepsilon}

\newcommand{\g}{\gamma}

\newcommand{\ld}{\lambda}

\newcommand{\s}{\sigma}

\newcommand{\ft}{\widehat}

\newcommand{\wt}{\widetilde}
\newcommand{\cj}{\overline}
\newcommand{\dx}{\partial_x}
\newcommand{\dt}{\partial_t}
\newcommand{\dd}{\partial}

\newcommand{\ta}{\theta}

\renewcommand{\l}{\ell}
\renewcommand{\o}{\omega}
\renewcommand{\O}{\Omega}

\newcommand{\les}{\lesssim}
\newcommand{\ges}{\gtrsim}

\newcommand{\jb}[1]
{\langle #1 \rangle}

\newcommand{\ind}{\mathbbm 1}

\newcommand{\N}{\mathbb{N}}

\newcommand{\NN}{\mathcal{N}}
\newcommand{\RR}{\mathcal{R}}
\newcommand{\QQ}{\mathcal{Q}}

\newcommand{\VV}{\mathcal{V}}

\renewcommand{\I}{\mathcal{I}}

\newtheorem*{ackno}{Acknowledgements}

\newcommand{\eps}{\ep}

\renewcommand{\ft}{\widehat}

\numberwithin{equation}{section}
\numberwithin{theorem}{section}

\begin{document}
\baselineskip = 14pt

\title[A.S. GWP and quasi-invariance for FNLS]
{
Quasi-invariance of Gaussian measures of negative regularity for fractional nonlinear Schr\"{o}dinger equations}

\author[J.~Forlano and L.~Tolomeo]
{Justin Forlano and Leonardo Tolomeo}

\address{
Justin Forlano, Department of Mathematics\\
University of California\\
Los Angeles\\
CA 90095\\
USA}

\email{forlano@math.ucla.edu}

 \address{
 Leonardo Tolomeo\\
 Mathematical Institute\\
 Hausdorff Center for Mathematics\\
 Universit\"{a}t Bonn\\
 Bonn\\
 Germany
}

\email{tolomeo@math.uni-bonn.de}

\maketitle

\begin{abstract}
We consider the Cauchy problem for the fractional nonlinear Schr\"{o}dinger equation (FNLS) on the one-dimensional torus with cubic nonlinearity and high dispersion parameter $\alpha > 1$, subject to a Gaussian random initial data of negative Sobolev regularity $\s<s-\tfrac{1}{2}$, for $s \le \frac 12$. We show that
for all $s_{\ast}(\al) <s\leq \tfrac{1}{2}$,
the equation is almost surely globally well-posed. Moreover, the associated Gaussian measure supported on $H^{s}(\T)$ is quasi-invariant under the flow of the equation.
For $\alpha < \frac{1}{20}(17 + 3\sqrt{21}) \approx 1.537$, the regularity of the initial data is lower than the one provided by the deterministic well-posedness theory. 

We obtain this result by following the approach of DiPerna-Lions (1989); first showing global-in-time bounds for the solution of the infinite-dimensional Liouville equation for the transport of the Gaussian measure, and then transferring these bounds to the solution of the equation by adapting Bourgain's invariant measure argument to the quasi-invariance setting. This allows us to bootstrap almost sure global bounds for the solution of (FNLS) from its probabilistic local well-posedness theory. 
\end{abstract}
\tableofcontents

\section{Introduction}
\subsection{Background}
In this paper, we consider the Cauchy problem for the (renormalized) fractional nonlinear Schr\"odinger equation (FNLS):
\begin{equation}\label{alphaNLS}
\begin{cases}
i \partial_t u + (-\dx^2)^{\al} u \pm \big(|u|^2  - 2 \big(\int_{\T} |u|^2 dx \big)\big)u = 0,\\
u(0) = u_0.
\end{cases}
\end{equation}
with high dispersion $\al>1$, posed on the one-dimensional torus $\T=\R/2\pi\Z$. Our goal is to study the long-time behaviour and statistical properties of the flow of this equation when subjected to a Gaussian random initial data $u_0$ of low regularity, so that $u_0 \not\in L^2(\T)$. 

The study of Hamiltonian PDEs with random initial data was initiated by McKean and Vaninsky for the cubic wave equation, and by Bourgain for the mass-critical Schr\"odinger equation, both posed on $\T$, in \cite{McVa94,Bo94}. These works were concerned with showing the invariance of the Gibbs measure associated to the Hamiltonian; the measures having been rigorously studied earlier by Lebowitz, Rose and Speer~\cite{LRS}.

In particular, Bourgain in \cite{Bo94} introduced the so-called Bourgain's invariant measure argument. The idea is to use the formal invariance of the Gibbs measure, as a replacement for a conservation law, to obtain almost sure global well-posedness on the support of the Gibbs measure.
 This scheme was applied in \cite{Bo94} to the one-dimensional NLS 
 \begin{align}
i \partial_t u -\dx^2 u +|u|^{p-1}u=0, \label{NLS}
\end{align}
where $p\in 2\mathbb{N}+1$,
  for almost every initial data (distributed according to the Gibbs measure) belonging to the $L^2$-based Sobolev space $H^{\frac 12 - \epsilon}(\T)$, where $\eps>0$. The local-in-time dynamics had been constructed earlier in another seminal paper by Bourgain~\cite{Bou}.
  We recall that the Gibbs measure for \eqref{NLS} is formally described by
  \begin{align}
d\rho(u) = Z^{-1} \exp(-H(u)) du = Z^{-1} \exp\Big(-\frac 12 \int_{\T} |\dx u|^2 - \frac 1p \int_{\T} |u|^p\Big) du,  \label{Gibbs}
\end{align}
where $H(u)=\frac 12 \int_{\T} |\dx u|^2 + \frac 1p \int_{\T} |u|^p dx$ is the Hamiltonian (energy) for \eqref{NLS}, and $Z$ is a normalisation constant. We heuristically expect $\rho$ to be conserved under the flow of \eqref{NLS} because (i) the Hamiltonian $H(u)$ is conserved, and (ii) the (non-existent) Lebesgue measure $``du"$ is formally conserved by Liouville's theorem.

Bourgain's result was especially remarkable because on the torus, one can use conservation of mass and energy to obtain a priori estimates for solutions to \eqref{NLS} with $p>3$ only in $H^1 (\T)$. Bourgain also obtained similar results in the focusing case ($+$ sign in \eqref{NLS}) when $p\leq 6$, although the unboundedness from below of the Hamiltonian requires an additional (conserved) $L^2(\T)$ cut-off on the measure \eqref{Gibbs}. We will not discuss such issues further here and refer the interested reader to \cite{OSTol}.

These techniques were further developed by Bourgain himself in his later result \cite{Bo96}, where he considered the defocusing cubic NLS posed on $\T^2$. He showed a similar almost sure global well-posedness result, this time in $H^{-\eps}(\T)$. 
At this regularity, not only does one lack nice a-priori bounds, but the equation is scaling super-critical, and it is ill-posed in a strong sense; see \cite{Ki14, Ki18} for more details. 

These results by Bourgain generated a lot of interest in the study of dispersive/Hamiltonian PDEs with random initial data, especially in the situation where it is possible to prove invariance of the Gibbs measure. This is typically a very difficult problem as in many cases of interest, such as in \cite{Bo96}, the Gibbs measure is supported on a space of functions which are too rough for deterministic well-posedness theory to apply (if it even exists at such regularities). This issue necessitates a \textit{probabilistic} (local) well-posedness theory, that goes beyond deterministic results by exploiting cancellations due to random oscillations.
The literature on the study of Gibbs measures for nonlinear Hamiltonian PDEs is by now quite long, so we point the reader to the papers \cite{Zhid, Tz06, TV1, TV2, OhKdV, NORBS, Deng, DTV, DNY2, ST1, OTzW, ORSW2, KMV, ST2, CK, DNY4, bdny}
and the references contained therein. 

In a parallel direction, a natural question that arises is to provide some description of the flow of Hamiltonian PDEs in situations where the initial data is \textit{not} distributed according to an invariant measure. This pursuit was initiated by Tzvetkov in \cite{Tz}, where the following question was studied: given a Gaussian measure $\mu$ and the flow $\Phi_t$ of a Hamiltonian PDE, when can we say that 
\begin{align}
(\Phi_t)_\# \mu \ll \mu\,? \label{Question}
\end{align}
Here, $(\Phi_t)_\#\mu$ denotes the push-forward of the measure $\mu$ under the map $\Phi_t$. If the answer to this question is positive, we say that the measure $\mu$ is {\it quasi-invariant} under the flow $\Phi_t$. 

Before we discuss this further, we give a rigorous definition for the Gaussian measures supported on Sobolev spaces.
Equipping $\T $ with normalised Lebesgue measure $(2\pi)^{-1}dx$, we let $u_0$ denote the random initial data given by 
\begin{equation}
u_0 ( x; \o) = \sum_{n \in \Z} \frac{g_n (\o)}{\jb{n}^s} e^{i n  x},  \label{u0}
\end{equation}
where $\{g_n\}_{n\in \Z}$ are independent complex-valued Gaussian normal random variables (namely, $\text{Re}(g_n)$ and $\text{Im}(g_n)$ are independent $\mathcal N(0,\frac12)$ Gaussian random variables on $\R$) on a probability space $(\O,\mathcal{F},\mathbb{P}),$ and $\jb{n}:=(1+n^2)^{1/2}$. It is well known that $u_0 \in H^\sigma (\T) \setminus H^{s-\frac{1}{2}}(\T)$ almost surely if and only if $\sigma < s - \frac 12$.
We then define the Gaussian measure $\mu_s$ as
\begin{equation}\label{musdef}
\mu_s(u_0) := \Law(u_0).
\end{equation}
At least formally,
\begin{equation*}
d\mu_s(u_0) = Z_{s}^{-1}\exp\big(- \tfrac12 \| u_0 \|_{H^s}^2\big)du_0,
\end{equation*}
since $\mu_s$ is a Gaussian measure with inverse covariance operator $(1-\dd^{2}_{x})^{s}$. 
The measure $\mu_s$ is then supported on $H^{\s}(\T)$, for $\s<s-\tfrac{1}{2}$, and the triplet $(H^s, H^\s, \mu_s)$ forms an abstract Wiener space; see~\cite{GROSS, Kuo2}.

The restriction to Gaussian measures in \eqref{Question} is natural in the context of Hamiltonian PDEs, as such equations often have conservation laws of the form:
$$ E(u) = \tfrac12 \| u \|_{H^s}^2 + \text{Rem}(u),$$
where $\text{Rem}(u)$ is a remainder term due to the nonlinearity in the equation,
and typically the associated invariant measure $ \exp(-E(u)) du$ is absolutely continuous with respect to the Gaussian measure $\mu_{s}$.\footnote{The main example in which this absolute continuity does not hold but the measure can still be proven to be invariant is the $\Phi^4_3$ measure of quantum field theory. This was very recently shown to be invariant for the flow of the nonlinear wave equation on $\T^3$ in the impressive work \cite{bdny}.} 
Indeed, a strategy to rigorously interpret the Gibbs measure is to view it as a weighted Gaussian measure 
\begin{align*}
d\rho =  \exp(-\text{Rem}(u))d\mu_s,
\end{align*}
and show that $e^{-\text{Rem}(u)}$ is integrable on the support of $\mu_{s}$.


We stress that the question of quasi-invariance of Gaussian measures under the flow of Hamiltonian PDEs is a non-trivial problem, even in high-regularity situations. In probability theory, the seminal work of Cameron and Martin~\cite{Cameron} identified necessary and sufficient conditions for the quasi-invariance of Gaussian measures under shifts of the form $x \mapsto x + y$. Ramer~\cite{Ramer} extended this study to general nonlinear transformations. In order to apply Ramer's result to prove the quasi-invariance under the flow of nonlinear Hamiltonian PDEs, one essentially needs a $(d+\eps)$-degree of smoothing on the nonlinear part of the flow, where $d\geq 1$ is the dimension of the torus $\T^{d}$. See a discussion in \cite[Section 1.4]{Tz} for more details. We also mention the works by Cruzeiro~\cite{Cru1, Cru2} on quasi-invariance for general evolution equations.

For a typical dispersive PDE, such an amount of nonlinear smoothing seems to be out of reach, and even when available, it requires careful analysis to unlock. 
For example, for \eqref{alphaNLS}, there seems to be $(2\al-1)$-degrees of nonlinear smoothing, provided that $s>1$. Thus, Ramer's result would imply quasi-invariance of $\mu_s$ under the flow of \eqref{alphaNLS} only for $s>1$. 

In order to go beyond Ramer's result, Tzvetkov \cite{Tz} introduced a general method to obtain quasi-invariance for Hamiltonian PDEs, which has been iterated and expanded upon in various directions over the years.
Indeed, many results have appeared regarding the quasi-invariance of Gaussian measures  
for various different dispersive PDEs. In particular, there are results for quasi-invariance of the BBM and Benjamin-Ono equations \cite{Tz, GLT1,GLT2}, KdV type equations~\cite{PTV2}, wave equations \cite{OTz2, GOTW, STX}, and Schr\"odinger equations \cite{OTz, OST, OTT, PTV, FT, DT2020, OS, FS, GLT1}. The key underlying observation in this study is that the quasi-invariance of Gaussian measures is intimately tied to the dispersive character of the equation; see \cite{OST, STX} for negative results for some dispersionless ODEs.

We point out that, contrary to the situation in which the initial data is distributed according to an invariant measure, these quasi-invariance results all rely on the underlying equation to at least be deterministically globally well-posed at the typical regularity of a sample of the Gaussian measure. In particular, so far, there has been no quasi-invariance result which is built upon a probabilistic local well-posedness result. 

%

\subsection{Main result}

The main goal of this paper is to bridge this gap between invariance and quasi-invariance results. In particular, we exploit the quasi-invariance of certain Gaussian measures in order to show almost sure global well-posedness \textit{below} the known local well-posedness threshold, in the spirit of the original works by Bourgain \cite{Bo94,Bo96}.

Our main result is the following:
\begin{theorem}\label{thm:qiandgwp}
Let $\alpha > 1$ and $s \le \frac 12$ be such that  
\begin{equation}\label{conditionalphas}
s > s_*(\alpha):= \begin{cases}
\frac14 \big(\sqrt{68\alpha^2-52\alpha + 9} - 10 \alpha + 7 \big) & \text{ if } \quad  1 < \alpha \le \frac{1}{32} \big(35 + \sqrt{105}\big),  \\
\frac12 \big(\sqrt{4\alpha^2-2\alpha+1} - 2\alpha + 1\big) & \text{ if } \quad \alpha >  \frac{1}{32} \big(35 + \sqrt{105}\big).
\end{cases}
\end{equation}
Then the equation \eqref{alphaNLS} is almost surely globally well-posed with respect to the Gaussian measure $\mu_{s}$. Moreover, the measure $\mu_s$ is quasi-invariant with respect to the flow and, if $\sigma < s - \frac 12$, for $\mu_s$ a.e.\ initial data $u_0$, there is an exponent $A = A(\alpha,s) > 0$ such that
\begin{equation}\label{sobolevgrowth}
\| u(t) \|_{H^\sigma} \les_{u_0,\sigma} \jb{t}^{A},
\end{equation}
\end{theorem}

Together with the result of the first author (and K.~Seong) \cite{FS}, that shows quasi invariance for \eqref{alphaNLS} for every $s>\frac12$, Theorem \ref{thm:qiandgwp} extends quasi-invariance of the measure $\mu_s$ to every $s > s_*(\alpha)$.

The condition \eqref{conditionalphas} is quite complicated, and we do not advance any claims about its optimality.\footnote{However, we point out that in the case $\alpha = 2$, Theorem \ref{thm:qiandgwp} provides $s_*=\frac12(\sqrt{13}-3)\approx 0.3028$, while the result in \cite{OS} shows quasi-invariance for $s > \frac 3{10}=0.3$, so $s_*(\alpha)$ is not optimal in the case $\alpha =2$.} However, we can compare it with the deterministic theory for \eqref{alphaNLS}. The equation \eqref{alphaNLS} is known to be globally well-posed in $H^\sigma (\T)$ for 
\begin{align}
 \sigma \ge \frac{1-\alpha}3,  \label{dethold}
\end{align}
thanks to the works \cite{OW, MT, Kwak, LSZ}. For more details, we refer to the discussion in Section~\ref{SEC:truncaux}.
In view of the optimal regularity for an initial data $u_0$ distributed according to $\mu_s$, Theorem \ref{thm:qiandgwp} provides a result stronger than the one derived by the deterministic theory if $s_*(\alpha) - \frac12 < \frac{1-\alpha}3$, which corresponds to 
$$ 1 < \alpha < \tfrac{1}{20}(17 + 3\sqrt{21}) \approx 1.537.$$

\subsection{Overview of the proof}\label{SEC:truncaux}
We summarise here the high level structure of the proof, explaining how it relates with the existing results in the literature, and how it allows us to go beyond the deterministic global (and local) well-posedness threshold. 

In view of Bourgain's invariant measure argument (or more precisely, a variant of the argument for the quasi-invariant case, see Theorem \ref{bima}), it is enough to show good $L^p$ estimates for the density of the transported measure. Therefore, we shift our attention to the Liouville equation associated with the equation \eqref{alphaNLS}, in the spirit of the celebrated works of DiPerna-Lions \cite{DL} and Ambrosio \cite{Amb}. We first show that, if $(\Phi_t)_\# \mu_s = f_t \mu_s$, then 
\begin{equation}
 f_t(\Phi_t(u_0)) = \exp\Big( \int_0^t \QQ(\Phi_{t'}(u_0))dt'\Big), \label{densityformulaintro}
\end{equation}
where $\QQ$ is an explicit function; see Lemma \ref{LEM:density}. Such a formula for the density has been known in the community for quite some time, however, it has been used in the context of quasi-invariance for dispersive PDEs only recently, initially by Debussche and Tsutsumi \cite{DT2020}, and subsequently in \cite{BTh,FS,GLT1}. 

Afterwards, by a bootstrap/Gronwall argument, we can derive an $L^p(\mu_s)$ estimate for $f_t$ from an $L^p(\mu_s)$ estimate for 
\begin{equation}
\exp\Big(\frac t {\tau} \int_0^\tau \QQ(\Phi_{t'}(u_0)) dt'\Big), \label{toLp}
\end{equation}
where $\tau\le \min(t,1)$ is an appropriately chosen stopping time. This argument is performed in Lemma \ref{LEM:ftNSM}. The bootstrap argument can be seen as a discrete-in-time version of the argument in \cite[Section 6]{AmbFi}, which has the term $ \QQ(u_0)$ in the exponential instead of its time average. However, we point out that in general, the term $\QQ(u_0)$ is ill-defined (see Remark \ref{QQdivergence}), and only its integral in time is well defined. 
As a consequence, the solution to the Liouville equation will not be Lipschitz in time (nor it will have a weak derivative in time in any sense), and this lack of temporal regularity extends to the transported density function $f_t \circ \Phi_t$.\footnote{Formally, this transported density solves an equation of the form 
$$ \partial_t (f_t \circ \Phi_t) = (\QQ\circ \Phi_t) (f_t \circ \Phi_t), $$
so the lack of regularity in time is somewhat unexpected.} 

In order to estimate the $L^p$ norm of \eqref{toLp}, we exploit the Bou\'e-Dupuis variational formula \eqref{P3}, popularised by Barashkov and Gubinelli~\cite{BG} the context of the construction of the $\Phi^4_3$ quantum field theory on $\T^3$. This formula essentially states that 
\begin{multline}\label{toestimatevialwp}
\log \Big(\int \exp\Big(\frac t {\tau} \int_0^\tau \QQ(\Phi_{t'}(u_0)) dt'\Big) d \mu_s(u_0)\Big)\\
 \le \int \sup_{v_0 \in H^s} \Big(\frac t {\tau} \int_0^\tau \QQ(\Phi_{t'}(u_0+v_0)) dt' - \frac 12 \| v_0 \|_{H^s}^2\Big) d \mu_s(u_0).
\end{multline}
Finally, we obtain the required estimate on $\frac t {\tau} \int_0^\tau \QQ(\Phi_{t'}(u_0+v_0)) dt'$ by developing a local well-posedness theory for \eqref{alphaNLS}, where $u_0$ is a random initial data distributed according to $\mu_s$, and $v_0 \in H^s$. As we need to exploit the nonlinear smoothing of the linear propagator of \eqref{alphaNLS} in order to be able to define $\int_0^\tau \QQ(\Phi_{t'}(u_0+v_0)) dt'$, it is important to develop the local well-posedness theory using the appropriate Fourier restriction norm spaces (introduced by Bourgain  in \cite{Bou} and Klainerman-Machedon in \cite{KlM}). 

In order to get slightly sharper estimates, which allow us to improve the value of $s_{*}(\al)$ and go below the deterministic threshold \eqref{dethold}, we need to work with the somewhat unusual spaces $X^{s,b}_{p,q}$. 
These keep track of the $L^q$-integrability in the time Fourier variable, rather than the more standard case of $q=2$. We point to Section~\ref{SEC:spaces} for precise definitions and also Remark~\ref{RMK:reasonforq} for a technical explanation for this choice of function space. The relevant local well-posedness estimates are contained in Proposition \ref{PROP:LWP}, while the estimate for $\QQ$ is shown in Lemma \ref{LEM:densityontau}.

In order to perform the steps described above, we need to introduce some auxiliary objects. The main reasons are that:
\begin{enumerate}[(I)]
\item the theory for the Liouville equation admits a rigorous justification only after a finite dimensional truncation of the equation \eqref{alphaNLS}, and 
\item the exponential estimates such as \eqref{toestimatevialwp} do not hold directly for the Gaussian measure $\mu_s$, but only for an auxiliary measure $\rho_{s,\gamma} \ll \mu_s$.
\end{enumerate}
We start by introducing the truncation for \eqref{alphaNLS}. In the following, we will exclusively consider the positive sign $+$ in \eqref{alphaNLS}. As we never perform any estimate that depends on the positivity of this sign, the same results will also hold for the negative sign $-$. 
We define  the nonlinearity in \eqref{alphaNLS} to be
\begin{equation}
|u|^2 u - 2 \Big(\int_{\T} |u|^2 dx\Big)u = \NN(u,u,u) + \Res(u,u,u), \label{nonlin}
\end{equation}
where the trilinear operators $\NN$ and $\Res$ are defined as
\begin{align}
\NN(u_1,u_2,u_3) &= \sum_{n_1-n_2+n_3-n_4 = 0, \atop n_1 \neq n_2,n_4} \ft{u_1}(n_1) \overline{\ft{u_2}(n_2)} \ft{u_3}(n_3) e^{in_4  x}, \label{nonres}  \\
\Res(u_1,u_2,u_3) &= -\sum_{n\in \Z}e^{inx} \ft{u_1}(n) \cj{\ft{u_2}(n)}\ft{u_3}(n), \label{res}
\end{align}
and $\ft u(n)$ denotes the Fourier transform in space of $u$. 

We point out that the subtraction of the term $2\int_{\T} |u|^2 dx$ in \eqref{nonlin} is necessary for studying \eqref{alphaNLS} below $L^2(\T)$, and is thus necessary for Theorem~\ref{thm:qiandgwp} to hold. Let us explain further. The equation \eqref{alphaNLS} without this term, namely,
\begin{align}
i\dt u +(-\dx^2)^{\al}u + |u|^2 u=0, \label{FNLS}
\end{align}
is the usual cubic fractional NLS. The equations \eqref{FNLS} and \eqref{alphaNLS} are equivalent in $L^2(\T)$ in the sense that if $u\in C([0,T];L^2(\T))$ solves \eqref{FNLS}, then  
\begin{align*}
\mathcal{G}(u) = e^{-it\int_{\T} |u|^2 dx} u
\end{align*}
solves \eqref{alphaNLS}. Moreover, this gauge transform $\mathcal{G}$ is also invertible. This relies on the fact that the mass 
\begin{align}
 M(u) := \int_\T |u|^2 dx \label{mass}
\end{align}
is a conserved quantity for the flow of both \eqref{FNLS} and \eqref{alphaNLS}. Whilst \eqref{FNLS} and \eqref{alphaNLS} are both globally-well posed in $L^2(\T)$ using the Fourier restriction norm method \cite{Bou, MT, OTz, FT}, the subtraction of $2M(u)$ in \eqref{nonlin} removes a bad resonant part, so that \eqref{alphaNLS} behaves better outside of $L^2(\T)$ than \eqref{FNLS}. Indeed, the arguments in \cite{GO, OW} can be applied to \eqref{FNLS} to show non-existence of solutions below $L^2(\T)$, whilst the flow of \eqref{alphaNLS} merely fails to be locally uniformly continuous with respect to the initial data; see, for example, \cite[Appendix A.2]{OTz}. 

The best known deterministic well-posedness results for \eqref{alphaNLS} that we mentioned earlier with the regularity restriction \eqref{dethold}, are based off a second gauge transformation which further weakens the resonant nonlinearity in \eqref{res}. We point out that, whilst results in \cite{MT, OW} imply the existence of global solutions to \eqref{alphaNLS} for $\s>-\tfrac{1}{4}$, if $\al=\tfrac{3}{2}$, and $\s>-\tfrac{9}{20}$, if $\al=4$, respectively, there is no uniqueness and thus no well-defined flow (so the question of quasi-invariance is ill-defined).

To continue with our discussion of $\textup{(I)}$ above, with a slight abuse of notation, we denote 
\begin{equation*}
\NN(u_1,u_2,u_3,u_4) := \jb{\NN(u_1,u_2,u_3),u_4}_{L^2} = \sum_{n_1-n_2+n_3-n_4 = 0, \atop n_1 \neq n_2,n_4} \ft{u_1}(n_1) \overline{\ft{u_2}(n_2)} \ft{u_3}(n_3) \overline{\ft{u_4}(n_4)},
\end{equation*}
and similarly
\begin{equation*}
\RR(u_1,u_2,u_3,u_4) := \jb{\RR(u_1,u_2,u_3),u_4}_{L^2}  = -\sum_{n\in \Z}\ft{u_1}(n) \cj{\ft{u_2}(n)}\ft{u_3}(n) \overline{\ft{u_4}(n_4)}.
\end{equation*}
Then, for $N\in \mathbb{N}$, we define the truncated system to be
\begin{equation} \label{4NLStrunc}
\begin{cases}
i \partial_t u + (-\dx^2)^{\al} u + P_N\big(|P_Nu|^2 P_Nu - 2 \big(\int |P_Nu|^2\big)P_Nu\big) = 0, \\
u(0) = u_0,
\end{cases}
\end{equation}
where $P_N$ is the sharp projection on Fourier frequencies with $|n| \le N$, i.e.\
\begin{equation*}
P_N u (x) = \sum_{|n| \le N} \ft{u}(n) e^{inx}.
\end{equation*}

Similar to the analogous result for \eqref{alphaNLS}, one can check that the mass \eqref{mass}
is preserved by the flow of \eqref{4NLStrunc}. An easy, but important, consequence of this mass conservation, together with the fact that the linear equation 
$$ i \partial_t u + (-\dx^2)^{\al} u = 0 $$
preserves the $H^\sigma (\T)$ norm for any $\sigma \in \R$, is the following:
\begin{proposition} \label{4NLStruncGWP}
The equation \ref{4NLStrunc} is globally well posed in $H^\sigma$ for every $\sigma \in \R$.
\end{proposition}
We will denote the nonlinearity in \eqref{4NLStrunc} as 
$$ P_N\big(|P_Nu|^2 P_Nu - 2 \big(\int |P_Nu|^2\big)P_Nu\big) = \NN_N(u,u,u) + \Res_N(u,u,u),$$
with $\NN_N$ and $\Res_N$ having analogous definitions to $\NN$ and $\Res$ respectively. Notice that we have exactly 
\begin{equation}
\begin{split}
\NN_N(u_1,u_2,u_3) &= P_N\NN(P_Nu_1,P_Nu_2,P_Nu_3), \\
 \Res_N(u_1,u_2,u_3) &= P_N\Res(P_Nu_1,P_Nu_2,P_Nu_3).
\end{split}\label{nonlineartruncationcomp}
\end{equation}
By an abuse of notation again, if all arguments are the same, we simply write $\NN_{N}(u)=\NN_{N}(u,u,u)$ and $\Res_{N}(u)=\Res_{N}(u,u,u)$.


We now move on to introducing the auxiliary measure $\rho_{s,\gamma}$ from point $\textup{(II)}$ earlier. The main obstacle to showing the relevant estimate for \eqref{toestimatevialwp} is the fact that $\QQ$ is a (unsigned) polynomial of degree $4$ in the variable $v_0$, while clearly $\|v_0\|_{H^s}^2$ is a polynomial of degree 2. In order to get around this issue, we would like to exploit the mass conservation for the equation \eqref{alphaNLS}. 
However, $\int_{\T} |u_0|^2 dx = \infty$ almost surely, for $u_0$ in the support of the Gaussian measures $\mu_s$ when $s\leq \tfrac{1}{2}$. We instead need to use a renormalized version of the mass, which is formally given by 
$$ \wick{M(u)} = \int_\T \wick{|u|^2}\, dx = \int_\T |u|^2 dx - \int \int_{\T} |u_0|^2dx d \mu_s(u_0), $$
and we will make rigorous sense of this object in Section~\ref{SEC:rhomeas}. We notice that, since formally $\wick{M(u)} = M(u) + \text{constant}$, then $\wick{M(u)}$ is also preserved by the flow of \eqref{alphaNLS}, and so we can multiply $\mu_s$ by any function of $\wick{M(u)}$, and this will not affect the formula for the density $f_t$ given in \eqref{densityformulaintro}.
Therefore, for $s > \frac 14$ and $\gamma < \frac{1}{1-2s}$, we consider the measure
\begin{equation}
d \rho_{s,\gamma}(u_0) := \exp\Big(- \Big| \int_{\T} \wick{|u_0|^2}\, dx \Big|^\gamma\Big) d\mu_s(u_0). \label{rhosg}
\end{equation}
In view of the inequality 
\begin{align*}
\ind_{\{ |\cdot | \leq K\}} \leq \exp(-|x|^{\g})\exp(K^\g),
\end{align*}
the choice of $\rho_{s,\g}$ is reminiscent of the $L^2$-cut-offs, commonly employed in the construction of Gibbs measures and in some quasi-invariance results above $L^2(\T)$.

\subsection{Further remarks}
\begin{remark}\label{rk:gammacondition}
The condition $s > \frac 14$ in the definition of $\rho_{s,\gamma}$ comes from the estimate \eqref{UnifRN}, and it is necessary in order for $\wick{M(u)}$ to be well-defined (almost surely); see also \eqref{s14}. However, the restriction $\gamma < \frac{1}{1-2s}$ at first glance might seem insubstantial, since one could easily define the measure $\rho_{s,\gamma}$ for any $\gamma > 0$. 
This restriction on $\gamma$ will play an important role in the numerology of Theorem \ref{thm:qiandgwp} as discussed in Remark \ref{rk:numerology}, so one might wonder where this restriction comes from.

The main reason why we need to impose it is so that we can use the estimate \eqref{l2integrability} while applying the Bou\'e-Dupuis variational formula. While we will not perform this analysis here, one can show that when taking $\gamma_1 \ge \frac{1}{1-2s}$ in the definition of the measure $\rho_{s,\gamma_1}$, the right hand side of \eqref{l2integrability} will at best have a term of the form $\eps_0 \|V_0\|_{L^2}^{\frac 2{1-2s}}$ (for some $\eps_0 > 0$). 
This is due to the fact that it is possible to choose $V_0$ so that $\big|\wick{M(Y+V_0)}\big|\les 1$ and $\|V_0\|_{H^s}^2 \les \|V_0\|_{L^2}^{\frac{2}{1-2s}}$. This has essentially been shown in the papers \cite{OOT1,OOT2} by Oh, Okamoto and the second author, where similar measures have been considered. 

As shown in those, this interplay between the bounds on $\wick{M(Y+V_0)}$ and the respective bounds on the $L^2$ norm of $V_0$ leads to unexpected results, such as phase transitions in the constructions of certain quantum fields theories that happen far from mass criticality. For these reasons, while we do expect the numerology of Theorem \ref{thm:qiandgwp} to be in principle possible to improve, we believe that the choice of the weight $\exp\big(- \big| \int \wick{u_0^2} \big|^\gamma\big)$ is optimal in this context.
\end{remark}
\begin{remark}
Despite being partially inspired by the approach of DiPerna-Lions \cite{DL}, and the work by Ambrosio \cite{Amb} that introduced the concept of regular Lagrangian flows for ODEs with drifts that admit only weak derivatives, we do not develop explicitly the  theory for the Liouville equation associated to \eqref{alphaNLS}. 
We instead simply show the $L^p$ estimate \eqref{lpdensity} for the solution to the Liouville equation with initial data $\rho_{s,\gamma}$. 

If one were to be interested instead to the well-posedness theory for the Liouville equation 
\begin{equation}\label{eqn:inftylio}
 \nu_t = (\Phi_t)_\# \nu_0, 
\end{equation}
one obtains uniqueness of the solution in the class $\{ \nu_t \ll \mu_s $ for every $t\}$ as a direct consequence of the local well-posedness theory for \eqref{alphaNLS}. Similarly, global existence for $\nu_0 \ll \mu_s$ follows by imposing 
$$ \nu_t = f_t \Big(\frac{d \nu_0}{d \rho_{s,\gamma}}\circ{\Phi_{-t}}\Big) \rho_{s,\gamma}. $$    
Moreover, the $L^p$ estimate \eqref{lpdensity} shows that for every $\nu_0 \in L^\infty(\rho_{s,\gamma})$, then $\nu_t \in L^p(\rho_{s,\gamma})$ for every $t \in \R$ and every $p < \infty$. By keeping track of the dependence on $p$ of the estimate \eqref{lpdensity}, it should be possible to show that the equation \eqref{eqn:inftylio} is globally well posed in the Orlicz space $\exp(\log(L)^\beta)(\rho_{s,\gamma})$, for some appropriate $\beta=\beta(\alpha,s, \gamma)>1$. However, this goes beyond the scope of this paper, and we chose not to perform this analysis.
\end{remark}
\begin{remark}
A vast amount of work has been performed in the context of studying local and global well-posedness of dispersive PDEs with a random Gaussian initial data, independently of the properties of the transported measure. This process has been started by Burq and Tzvetkov in \cite{BuTz1, BuTz2}, who showed that the cubic wave equation posed on $\T^3$,
\begin{equation}\label{NLW}
u_{tt} - \Delta u + u^3 = 0,
\end{equation}
is globally well-posed as soon as the random initial data belongs to the space $H^{\eps}(\T^3)$ for some $\eps > 0$. This result is especially remarkable, since the equation \eqref{NLW} is deterministically ill-posed in $H^\sigma (\T^3)$ for $\sigma < \frac 12$. This result was achieved by writing the solution $u$ to \eqref{NLW} as 
$$ u = \psi + v, $$
where $\psi$ is the rough solution to the linear wave equation, $\psi \in H^\eps (\T^3)$, and the remainder $v$ belongs to the energy space $H^1$. 
In this way, they showed almost sure global well-posedness for \eqref{NLW} by performing an energy estimate for the remainder $v$. Similar results, always in the context of wave equations with random initial data, have been obtained in \cite{OP1,OP2,P}. 

However, this globalization technique breaks down very quickly when the regularity of $\psi$ is below $L^2 (\T^3)$. 
By exploiting the $I$-method, it should possible to show global well-posedness for \eqref{NLW} when $\psi \in \bigcap_{\eps > 0} H^{-\eps}(\T^3)$,  extending the local-in-time solutions in \cite{OPT},
but the case $\psi \not\in H^{-\eps_0}(\T^3)$ is beyond the current technology. In the context of \eqref{alphaNLS}, this corresponds exactly to the case $s = \frac 12$ (with the extra difficulty that the remainder $v$ in general does not belong to the energy space $H^{\alpha}$). 
For more details about this approach, see the papers of Gubinelli, Koch, Oh and the second author \cite{GKOT,TowaveR2}, where this analysis was performed in the setting of the stochastic wave equation on $\T^2$ and $\R^2$, respectively. See also the paper by the first author~\cite{Fo}, where a similar analysis was performed for the BBM equation. 

Other global well-posedness results that rely on energy estimates for the remainder have been obtained in \cite{CO,ADL,OOP}, in the context of nonlinear Schr\"odinger equations. Of particular relevance to this work is the result of \cite{CO}, 
where the authors exploited Bourgain's high-low method to show global well-posedness for \eqref{alphaNLS} with $\alpha =1$ and initial data distributed according to $\mu_s$, for $s> \frac{5}{12}$. However, we point out that due to the complete integrability of \eqref{alphaNLS} when $\alpha = 1$, and the existence of conserved quantities of the form 
$$ E_s(u) = \tfrac 12 \| u \|_{H^s}^2 + \mathrm{Rem}(u) $$
for every $s > - \frac12$ \cite{KT, KVZ}, we expect this particular case to belong to the more general setting of Bourgain's invariant measure argument (even if the invariance of such measures on $\T$  for $s < 1$ is still an open problem).
\end{remark}

\begin{remark}
In their works \cite{DNY2,DNY3}, Deng, Nahmod and Yue introduced the concepts of {\it random operators} and {\it random tensors}, with the explicit goal of building a local-in-time solution theory for nonlinear dispersive equations with Gaussian random initial data of low regularity. After a suitable renormalisation of the nonlinearity, this theory allowed them to solve the nonlinear Schr\"odinger equation posed on the torus $\T^d$:
\begin{equation} \label{NLS}
\begin{cases}
 i\dt u- \Delta u + |u|^{p-1}u = 0 , \\
u(0) = u_0,
\end{cases}
\end{equation}
with Gaussian initial data $u_0$ of the form \eqref{u0}, for $s > \frac d2- \frac{1}{p-1}$. In principle, it would be possible to apply the theory of random tensors also to the study of \eqref{alphaNLS}. 
This should lead to improvements in the bilinear estimates of Proposition \ref{PROP:bilinearlwp} and of Lemma \ref{lem:bilineargwp}. 
However, we point out that improving these estimates does not lead to any improvement in the statement of Theorem \ref{thm:qiandgwp}. Indeed, the current restriction \eqref{conditionalphas} on $s$ and $\alpha$ depends only on the following estimates:
\begin{enumerate}
\item The power of $\|v_0\|_{L^2}$ in the estimates \eqref{tau0bound} and \eqref{tausigmabound}, which in turn depends only on the \textit{deterministic} estimate of Proposition \ref{PROP:detquad},
\item The restriction \eqref{LWPqbd} on $q$, which follows again from the estimates of Proposition \ref{PROP:detquad}.
In this case, this restriction comes from the terms of $\NN(u)$ that depend linearly on the random initial data $u_0$. Hence, in principle we cannot exclude that exploiting further the randomness of $u_0$ it could be possible to improve the restriction on $q$. However, it is unclear to the authors how to obtain such an estimate. Moreover, this would lead to a lower bound for $s_*(\alpha)$ in Theorem \ref{thm:qiandgwp} only in the case $1 < \alpha \le \frac{1}{32} \big(35 + \sqrt{105}\big)$. 
\item The \textit{deterministic} estimates of Lemma \ref{LEM:detenergy}.
\end{enumerate}
Since introducing the notions of random operators or random tensors would make the paper significantly more complex, without leading to an improvement of the main result, we decided to rely on a less technological approach for the estimates of Proposition \ref{PROP:bilinearlwp} and Lemma \ref{lem:bilineargwp}. 

However, we do not exclude that the theory of random tensors, combined with the techniques developed in this paper, could lead to globalization results when considering equations with lower dispersion, higher order nonlinearities, or posed in higher dimension.
\end{remark}

\subsection{Structure of the paper}
This paper is organised as follows. In Section 2, we go over some preliminaries: First, we lay out some notation. Then, we show certain estimates for the phase function 
\begin{align}
\Phi_{\al}(\cj n):= \Phi_{\al}(n_1,n_2,n_3,n_4)= |n_1|^{2\al}-|n_2|^{2\al}+|n_3|^{2\al}-|n_4|^{2\al}, \label{p1}
\end{align}
and the symbol of $2s$ symmetrized  derivatives 
\begin{align}
\Psi_{s}(\cj{n}):= \Psi_{s}(n_1,n_2,n_3,n_4)=\jb{n_1}^{2s}-\jb{n_2}^{2s}+\jb{n_3}^{2s}-\jb{n_4}^{2s}. \label{symderiv}
\end{align}
The phase function $\Phi_\alpha$ will enter naturally in the analysis due to the structure of the equation \eqref{alphaNLS}, while the symbol of the symmetrized  derivatives will enter in the formula for the density $f_t$ of $(\Phi_t)_\#\mu_s$ with respect to $\mu_s$. 
Finally, we introduce the main probabilistic tools that we will use throughout the paper, and discuss the measures $\rho_{s,\g}$ at a rigorous level.

In Section 3, we introduce the Fourier restriction spaces that we are going to use for the local well-posedness theory, and show the multilinear estimates required to prove the main local well-posedness result of Proposition \ref{PROP:LWP}.

In Section 4, we define the stopping time $\tau$ that enters in \eqref{toLp}, and show an exponential integrability estimate for $\tau$.

In Section 5, we perform the main steps of the proof described earlier in this section for the truncated equation \eqref{4NLStrunc}. Most of the section is dedicated to showing the relevant multilinear estimates for the quadrilinear operator $\QQ$. The main result of the section is the  $L^p$ estimate for the density $f_t$, uniform in the parameter $N$, contained in Proposition \ref{prop:lpbound}. 

Finally, in Section 6, we perform the ``soft" part of the analysis, showing Bourgain's quasi-invariant measure argument and taking a limit for $N \to \infty$. This way, we extend the $L^p$ estimate for the density to the limit, and finalise the proof of Theorem \ref{thm:qiandgwp}.

\begin{ackno}
L.T.~was supported by the Deutsche
Forschungsgemeinschaft (DFG, German Research Foundation) under Germany's Excellence
Strategy-EXC-2047/1-390685813, through the Collaborative Research Centre (CRC) 1060. 
\end{ackno}

\section{Preliminaries}

\subsection{Notation}

For $1 \le p \le \infty$, we denote by $p'$ its H\"older conjugate, i.e.\ the only number such that 
$$ \frac 1 p + \frac 1 {p'} = 1. $$

Given $s\in \R$ and $1\leq p\leq \infty$, we define the Fourier-Lebesgue spaces $\FL^{s,p}(\T)$ via the norm
\begin{align*}
\| u\|_{\FL^{s,p}(\T)} = \| \jb{n}^{s} \ft u (n)\|_{\l^{p}_{n}(\Z)}. 
\end{align*}
When $p=2$, the spaces $\FL^{s,2}(\T)=H^{s}(\T)$. 

Given an integral operator $H:\l^2 (\Z) \to \l^2 (\Z)$ with kernel $K(n,m)$ so that 
\begin{align*}
(Hv)(n)= \sum_{m\in \Z} K(n,m) v(m), \quad v\in \l^2(\Z),
\end{align*}
we define the Hilbert-Schmidt norm $\| H\|_{HS(\l^2,\l^2)} =\| K(n,m)\|_{\l_{n}^2 \l_{m}^2}$.

We use $a+$ (and $a-$) to denote $a+\eps$ (and $a-\eps$, respectively) for arbitrarily small $\eps\ll 1$,
where an implicit constant is allowed to depend on $\eps > 0$ (and it usually diverges as $\eps\to 0$).

We use $c$ and $C$ to denote various constants that may change from line to line.
For two quantities $A$ and $B$ , we write
$A\les B$ if there is a uniform constant $C>0$ such that $A\leq CB$. We write
$A\sim B$ if $A\les B$ and $B\les A$. We also underscore the dependency of $C$ on any additional parameter $a$ by writing $A\les_{a} B$. We also write $A\vee B=\max(A,B)$ and $A\wedge B=\min(A,B)$.

Given dyadic numbers $N_{j} \geq 1$, $j=1,\ldots, k$, we define $N_{\max}=\max_{j=1,\ldots, k} N_j$ and $N_{\min}=\min_{j=1,\ldots, k} N_j$. For $N_1,N_2,N_3,N_4$ dyadic, we let $N^1 \ge N^2 \ge N^3 \ge N^4$ be a permutation of the $\{N_j\}_{j=1}^{4}$. 

\subsection{The phase function and symmetrized derivatives}


For particular values of $\al$, the phase function $\Phi_{\al}(\cj{n})$ in \eqref{p1} enjoys explicit factorisations on the hyperplane $n_4=n_1-n_2+n_3$. Indeed, if $\al=1$, it is well known that 
\begin{align*}
\Phi_{1}(\cj n)=n_1^{2}-n_2^2+n_3^2-n_4^2=-2(n_4-n_1)(n_4-n_3)  \qquad \text{when }\, n_4=n_1-n_2+n_3.
\end{align*}
This corresponds to the case of the one-dimensional cubic NLS.
A similar explicit factorisation holds for the 4NLS ($\al=2$ in \eqref{alphaNLS}); see \cite[Lemma 3.1]{OTz}.
For general $\al>\tfrac{1}{2}$, there is no such factorisation, so instead we rely heavily on the following lower bound, which we will use without explicit reference: 
for $\al>\tfrac 12$, we have
\begin{align}
\left|\Phi_{\al}(\overline{n})\right|
& \gtrsim |n_4-n_1||n_4-n_3|n_{\max}^{2\alpha-2} \qquad \text{when} \qquad n_4=n_1-n_2+n_3. \label{Phase}
\end{align}
and where $n_{\max}:=\max( |n_1|,|n_2|,|n_3|,|n_4|)+1$. We refer to \cite{demirbas2013existence,FT} for proofs. In particular, $|\Phi_{\al}(\cj{n})| \ges 1$ on $n_4=n_1-n_2+n_3$ and when $\{n_1,n_3\}\neq \{n_2,n_4\}$.

We also require estimates on the symmetrized derivatives $\Psi_{s}(\cj{n})$ in \eqref{symderiv}
 when $s\leq \tfrac{1}{2}$. 
The following two lemmas provide the estimates that we need.

\begin{lemma}
Let $0\leq s \le \frac12$. For $a,b \in \R$, we have that 
\begin{equation}\label{jbconcavity}
|\jb{a}^{2s} - \jb{b}^{2s}| \les \jb{a-b}^{2s}.
\end{equation}
In particular, if $n_4=n_1-n_2+n_3$, then 
\begin{equation}
|\Psi_{s}(\cj n)| \les \min( \jb{n_1-n_2}^{2s}, \jb{n_1-n_4}^{2s}) \les \jb{n_1-n_2}^{s}\jb{n_1-n_4}^{s}.
\label{psisbd}
\end{equation}
\end{lemma}
\begin{proof}
Since $s \le \frac 12$, the function $x \mapsto x^{2s}$ is concave on $[0, + \infty)$. Therefore, we have that 
\begin{align*}
||a|^{2s} - |b|^{2s}| \le | |a| - |b| |^{2s}  \le |a-b|^{2s}.
\end{align*}
Therefore, since $ |x|^{2s} - 1 \le \jb{x}^{2s} \le |x|^{2s} + 1$, we have
\begin{align*}
\jb{a}^{2s} - \jb{b}^{2s} \le 2 + |a|^{2s} - |b|^{2s} \le 2 +  |a-b|^{2s} \les \jb{a-b}^{2s}.
\end{align*}
Now, \eqref{psisbd} follows from $n_1 - n_2 = -(n_3 - n_4)$, $n_1 - n_4 = -(n_3 - n_2)$ and \eqref{jbconcavity}.
\end{proof}

As a consequence of the elementary inequality: 
\begin{align*}
\big||n|^{2s} - |m|^{2s}\big| \les \tfrac{|n-m|}{\jb{n}^{1-2s} \vee \jb{m}^{1-2s}} \quad \text{for} \quad \tfrac{1}{4}<s\leq \tfrac{1}{2},
\end{align*} we have the following estimate for $\Psi_s(\cj n)$.
\begin{lemma}\label{LEM:psi}
Fix $\tfrac 14 < s \le \tfrac{1}{2}$ and let $n_1,n_2,n_3,n_4\in \Z$ be such that $n_4=n_1-n_2+n_3$. Then, we have
	\begin{align*}
	\left|\Psi_s(\cj{n})\right|\les \frac{\jb{n_1-n_2}}{(\jb{n_1}^{1-2s}\vee\jb{n_2}^{1-2s})\wedge(\jb{n_3}^{1-2s}\vee\jb{n_4}^{1-2s})},
	\end{align*}
	and similarly
	\begin{align*}
	\left|\Psi_s(\cj{n})\right|\les \frac{\jb{n_1-n_4}}{(\jb{n_1}^{1-2s}\vee\jb{n_4}^{1-2s})\wedge(\jb{n_2}^{1-2s}\vee\jb{n_3}^{1-2s})},
	\end{align*}	
	where the implicit constant depends only on $s$.
%
%
%
\end{lemma}

%

\subsection{Probabilistic tools}

The main probabilistic tools we use in this paper are the Wiener chaos estimate and the Bou\'e-Dupuis variational formula.

Recall that $\{ g_{n}\}_{n\in \Z}$ denotes a family of i.i.d standard complex-valued normal random variables on a probability space $(\O, \mathcal{F}, \mathbb{P})$.

\begin{lemma}[Wiener Chaos estimate]  \label{LEM:WCE}
Given $k\in \mathbb{N}$, let $c:\Z^{k} \to \C$ and 
\begin{align*}
S(\o)= \sum_{(n_1,\ldots, n_k)\in \Z^{k}} c(n_1,\ldots,n_k)g_{n_1}(\o)\cdots g_{n_k}(\o).
\end{align*}
Suppose that $S\in L^2(\O)$. Then, there exists $C_{k}>0$ such that for every $2\leq p<\infty$, 
\begin{align*}
\|S\|_{L^p(\O)} \leq C_{k} p^{\frac{k}{2}} \|S\|_{L^2(\O)}.
\end{align*}
\end{lemma}

In order to state the Bou\'e-Dupuis variational formula,
as in \cite{BG,TW, OOT1,OOT2, OST20}, we first need to introduce some notations.
Let $W(t)$ be  a cylindrical Brownian motion in $L^2(\T)$, which means that
\begin{align*}
W(t) = \sum_{n \in \Z} B_n(t) e^{inx},
\end{align*}

\noi
where  
$\{B_n\}_{n \in \Z}$ is a sequence of mutually independent complex-valued\footnote{By convention, we normalize $B_n$ so that $\text{Var}(B_n(t)) = t$.} Brownian motions on a probability space $(\O, \mathcal{F}, \mathbb{P})$.
We then define a centered Gaussian process $Y(t)$
by 
\begin{align}
Y(t)
=  \jb{\nabla}^{-s}W(t).
\label{P2}
\end{align}

\noi
Note that 
we have $\text{Law}_{\P}(Y(1)) = \mu_s$, 
where $\mu_s$ is the Gaussian measure in \eqref{musdef}.
By setting  $Y_N = P_NY $, 
we have   $\text{Law}_{\P}(Y_N(1)) = (P_N)_*\mu_s$, 
i.e.~the push-forward of $\mu_s$ under $P_N$.

Next, let $\mathbb{H}_{a}$ denote the space of drifts, 
which are progressively measurable processes 
belonging to 
$L^2([0,1]; L^2(\T))$, $\P$-almost surely. 
We now state the variational formula~\cite{BD, Ust};
in particular, see Theorem 7 in \cite{Ust}.

\begin{lemma}[Bou\'e-Dupuis variational formula]\label{LEM:var3}
	Let $Y$ be as in \eqref{P2} and fix $N \in \mathbb{N}$.
	Suppose that  $F:C^\infty(\T^d) \to \R$
	is measurable, and satisfies $\E\big[|F(P_NY(1))|^p\big] < \infty$
	and $\E\big[|e^{-F(P_NY(1))}|^{p'} \big] < \infty$, for some $1 < p < \infty$.
	Then, we have
	\begin{align}
	- \log \E\Big[e^{-F(P_N Y(1))}\Big]
	= \inf_{\theta \in \mathbb H_a}
	\E\bigg[ F(P_N Y(1) + P_N I(\theta)(1)) + \frac{1}{2} \int_0^1 \| \theta(t) \|_{L^2_x}^2 dt \bigg], 
	\label{P3}
	\end{align}	
\noi
where  $I(\theta)$ is  defined by 
\begin{align*}
I(\theta)(t) = \int_0^t \jb{\nabla}^{-s} \theta(t') dt'
\end{align*}

\noi
and the expectation $\E = \E_\P$
is with respect to the underlying probability measure~$\P$. 
\end{lemma}

 We recall some pathwise regularity bounds  on 
$Y(1)$ and $I(\theta)(1)$. The proofs are well-known and can be found, for instance, in \cite[Lemma 4.7]{GOTW}.

\begin{lemma}  \label{LEM:Dr}
For any finite $p \ge 1$ and $\s<s-\tfrac{1}{2}$, 
\begin{align*}
\sup_{N\in \mathbb{N}}\E \Big[  \| Y_N(1)\|_{W^{\s,\infty} }^p \Big] \leq C_p <\infty,
\end{align*}
and moreover, for any $\theta \in \mathbb{H}_{a}$, 
\begin{align}
\| I(\theta)(1) \|_{H^{s}}^2 \leq \int_0^1 \| \theta(t) \|_{L^2}^2dt.
\label{CM}
\end{align}
\end{lemma}

As a corollary of \eqref{P3} and \eqref{CM}, we obtain the following simplified version of the Bou\'e-Dupuis variational formula.

\begin{lemma}\label{LEM:BD2}
Let $Y$ be a random variable distributed according to $\mu_s$.
	Fix $N \in \mathbb{N}$.
	Suppose that  $F:C^\infty(\T^d) \to \R$
	is measurable such that $\E\big[|F_{-}(P_NY)|^p\big] < \infty$ for some $p > 1$, where $F_{-}$ denotes the negative part of $F$. Then 
	\begin{align}
	\log \E\Big[e^{F(P_N Y)}\Big]
	\le 
	\E\bigg[ \sup_{V \in H^s} \Big\{  F(P_N Y + P_N V) -  \frac{1}{2} \|V\|_{H^s}^2 \Big\} \bigg],
	\label{P4}
	\end{align}
where the expectation is taken with respect to the underlying probability measure.
\end{lemma}
\begin{proof}
We start by proving \eqref{P4} in the case where $F$ is bounded from above. In this case, abusing of notation, we pick a process $Y(t)$ as in \eqref{P2}, so that $\Law(Y(1)) = \mu_s$. We can then apply Lemma \ref{LEM:var3} (to $-F$) and \eqref{CM}, and obtain that 
\begin{align*}
\log \E\Big[e^{F(P_N Y)}\Big]&=\log \E\Big[e^{F(P_N Y(1))}\Big] \\
&=\sup_{\theta \in \mathbb H_a}
	\E\bigg[ F(P_N Y(1) + P_N I(\theta)(1)) - \frac{1}{2} \int_0^1 \| \theta(t) \|_{L^2_x}^2 dt \bigg] \\
&\le \sup_{\theta \in \mathbb H_a}
	\E\bigg[ F(P_N Y(1) + P_N I(\theta)(1)) - \frac{1}{2} \|I(\theta)(1)\|_{H^s}^2 \bigg] \\
&\le \sup_{\theta \in \mathbb H_a}
	\E\bigg[ \sup_{V \in H^s} \Big\{ F(P_N Y(1) + P_N V) - \frac{1}{2} \|V\|_{H^s}^2 \Big\} \bigg] \\
&= \E\bigg[ \sup_{V \in H^s} \Big\{ F(P_N Y + P_N V) - \frac{1}{2} \|V\|_{H^s}^2 \Big\} \bigg].
\end{align*}
We now move to the situation in which $F$ is unbounded from above. For $n \in \N$, define $F_n := \min(F,n)$. Since $F_n$ is bounded from above, we obtain that 
\begin{align*}
\log \E\Big[e^{F_n(P_N Y)}\Big] &\le \E\bigg[ \sup_{V \in H^s} \Big\{ F_n(P_N Y + P_N V) - \frac{1}{2} \|V\|_{H^s}^2 \Big\} \bigg]\\
&\le \E\bigg[ \sup_{V \in H^s} \Big\{ F(P_N Y + P_N V) - \frac{1}{2} \|V\|_{H^s}^2 \Big\} \bigg].
\end{align*}
By monotone convergence, we have that 
$$ \lim_{n \to \infty} \E\Big[e^{F_n(P_N Y)}\Big] =   \E\Big[e^{F(P_N Y)}\Big], $$
so by taking limits as $n \to \infty$ we obtain \eqref{P4}.
\end{proof}

\subsection{On the measures $\rho_{s,\gamma}$} \label{SEC:rhomeas}

In the following section, we give a rigorous meaning to the auxiliary measures $\rho_{s,\g}$ defined in \eqref{rhosg}.

For $s\leq \tfrac{1}{2}$ and $u_0$ distributed according to $\mu_s$, the quantity $\|u_0\|_{L^2}^{2}$ is almost surely infinite. This divergence is due to the unboundedness of the variance of $|u_0(x)|^2$ for any fixed $x\in \T$. Indeed, for any $N\in \mathbb{N}$,
\begin{align*}
\s_N :=\E[ |P_{N}u_0 (x)|^{2}] = \sum_{|n| \le N} \frac{1}{\jb{n}^{2s}}  \sim 
\begin{cases}
\log N & \quad \text{if} \quad s=\tfrac{1}{2}, \\
N^{1-2s} & \quad \text{if} \quad s<\tfrac{1}{2}.
\end{cases}
\end{align*}
We instead consider the Wick ordered product:
\begin{align*}
\wick{|P_{N} u_0|^{2}}\,= |P_{N} u_0|^{2} -\s_{N},
\end{align*}
and thus the Wick ordered $L^2$-norm:
\begin{align*}
\int_{\T}  \wick{|P_{N}u_0|^2 (x)} \, dx  =\int_{\T} |P_{N}u_0(x)|^2 dx -\s_{N}.
\end{align*}
Since $\s_N$ is independent of $x\in \T$, we also have $\s_N = \E[ \|P_{M}u_0\|_{L^2}^{2}]$.
Then, by Lemma~\ref{LEM:WCE}, for any $1\leq p<\infty$, 
\begin{align}
\E \bigg[ \Big \vert \int_{\T}  \wick{|P_{N}u_0|^2 (x)} \, dx \Big\vert^{p} \bigg]^{\frac{1}{p}} \les p \E \Big[ \Big| \sum_{|n|\leq N} \tfrac{|g_n|^2-1}{ \jb{n}^{2s}}\Big\vert^{2} \Big]^{\frac{1}{2}} \les p \bigg( \sum_{|n|\leq N} \frac{1}{\jb{n}^{4s}} \bigg)^{\frac{1}{2}} \les_{p} 1, \label{s14}
\end{align}
uniformly in $N\in \mathbb{N}$ if $s>\tfrac{1}{4}$. Moreover, by trivially modifying the above computation,
we obtain
\begin{align}
\bigg\|  \int_{\T}  \wick{|P_{N}u_0|^2 (x)} \, dx -  \int_{\T}  \wick{|P_{M}u_0|^2 (x)}\, dx \bigg\|_{L^p(\mu_s)} \leq Cp N^{-1+4s},
\label{WickCauchy}
\end{align}
for any $1\leq p<\infty$ and $M\geq N \geq 1$.
Now, \eqref{WickCauchy} implies that the sequence $\{ \int_{\T} \wick{(P_{N}u_0)^2} dx\}_{N\in \mathbb{N}}$ is Cauchy in $L^p (\mu_{s})$ and we may define 
\begin{align*}
\int_{\T} \wick{ |u_0|^{2}} \, dx = \lim_{N\to \infty} \int_{\T} \wick{ |P_{N} u_0|^{2}}\, dx \in L^{p}(d\mu_{s})
\end{align*}
for any $1\leq p<\infty$.
Given fixed $N\in \mathbb{N}$ and $\g\geq 0$, we define
\begin{align*}
R_{N}(u_0) = \exp \bigg( - \bigg\vert \int_{\T} \wick{ |P_{N}u_0|^{2}} \,dx \bigg\vert^{\g} \bigg). 
\end{align*}
Then, the following proposition ensures that the measures $\rho_{s,\g}$ in \eqref{rhosg} are well-defined.

\begin{proposition}\label{PROP:measconst}
Let $\tfrac{1}{4}<s \leq \tfrac{1}{2}$ and $\g<\tfrac{1}{1-2s}$. Then, for any finite $p\geq 1$,
\begin{align}
\sup_{N\in \mathbb{N}}  \|R_{N}(u_0)\|_{L^p(\mu_s)} <\infty. \label{UnifRN}
\end{align} Moreover, for any $1\leq p<\infty$, $R_{N}(u_0)$ converges to $\exp \big( - \big\vert \int_{\T} \wick{ |u_0|^{2}}dx \big\vert^{\g} \big)$
in $L^p(d\mu_s)$.
 \end{proposition}

The proof of Proposition~\ref{PROP:measconst} follows by standard arguments so we will be brief. Clearly, the bound \eqref{UnifRN} follows from the point-wise bound $R_N \le 1$. Then, the convergence of $R_{N}$ in $L^p(\mu_s)$ follows from the uniform bound \eqref{UnifRN} and the weaker convergence in measure (which in turn rests upon \eqref{WickCauchy}). 

Notice that, as observed in Remark \ref{rk:gammacondition}, the restriction $\gamma < \tfrac{1}{1-2s}$ is completely inessential in the construction of the measure $\rho_{s,\gamma}$. However, we decided to keep it in the statement of Proposition \ref{PROP:measconst} in view of the following estimate, which will be crucial when developing the arguments of Section 4 and 5.
\begin{lemma}\label{LEM:var5}
For $\frac 14 < s \leq  \frac 12$ and $\gamma < \frac{1}{1-2s}$, there exist $C_\gamma = C_\gamma(s) > 0$, $\kappa = \kappa(\gamma,s)$, $\delta > 0$, such that 
\begin{equation} \label{l2integrability}
\Big|\int_{\T} |V_0|^2 + 2   \textup{Re}( \cj{V_0} Y) dx+ Z  \Big|^\gamma \ge \frac 12\|V_0\|_{L^2}^{2\g} - \frac 1{10} \| V_0 \|_{H^s}^2 - C_\gamma \|Y\|_{H^{s-\frac12 - \delta}}^\kappa - C_\gamma |Z|^\gamma.
\end{equation}
\end{lemma}
\begin{proof}
By the triangle inequality, there exists a constant $C$ such that 
\begin{align}\label{l2integrability1}
\Big|\int_{\T} |V_0|^2 + 2   \textup{Re}( \cj{V_0} Y)  dx+ Z \Big|^\gamma \ge \frac 34 \Big(\int_{\T} |V_0|^2 dx\Big)^\gamma - C \Big|\int_{\T} \cj{V_0} Y dx\Big|^\gamma - C |Z|^\gamma.
\end{align}
Moreover, for $\delta$ small enough, by interpolation and Young's inequality, for every $\eps_0,\eps_1 > 0$, we have that
\begin{align}
\Big|\int_{\T} \cj{V_0} Y dx\Big|^\gamma & \les \|V_0\|_{H^{\frac12+\delta-s}}^\gamma \|Y\|_{H^{s-\frac12 - \delta}}^\gamma \notag \\
&\les \|V_0\|_{L^2}^{\gamma \frac{4s-1-2\delta}{2s}} \|V_0\|_{H^s}^{\gamma \frac{1+2\delta-2s}{2s}} \|Y\|_{H^{s-\frac12 - \delta}}^\gamma \notag\\
&\les \eps_0 \|V_0\|_{L^2}^{2\gamma} + \eps_1 \|V_0\|_{H^s}^2 + C(\gamma,s,\eps_0,\eps_1) \|Y\|_{H^{s-\frac12 - \delta}}^{( 1 -\frac{4s-1-2\delta}{4 s}-  \frac{\gamma(1+2\delta - 2s)}{4s})^{-1}} \label{l2integrability2}.
\end{align}
Notice that, recalling that $\gamma < \frac{1}{1-2s}$, for $\delta$ small enough,
\begin{align*}
\tfrac{4s-1-2\delta}{4 s} +  \tfrac{\gamma(1+2\delta - 2s)}{4s} &= \tfrac{4s-1}{4 s} +  \tfrac{\gamma(1- 2s)}{4s} + O(\delta) < \tfrac{4s-1}{4 s} +  \tfrac{1}{4s} =1. 
\end{align*}
Therefore, from \eqref{l2integrability1} and \eqref{l2integrability2}, we obtain \eqref{l2integrability}.
\end{proof}

\section{Local well-posedness} \label{SEC:LWP}

\subsection{Function spaces}\label{SEC:spaces}

For the local well-posedness analysis of \eqref{alphaNLS} and \eqref{4NLStrunc}, we will use the Fourier restriction norm spaces $X^{s,b}_{p,q}$, which are adapted to the linear flow $S(t)=e^{it(-\dx^2)^{\al}}$, $\al>1$, and the Fourier-Lebesgue spaces $\FL^{s,q}(\T)$. These were first introduced in \cite{GH}. 
Given $s,b\in \R$ and $1\leq p,q \leq \infty$, we define the space $X_{p,q}^{s,b}=X^{s,b}_{p,q}(\R\times \T)$ via the norm
\begin{align*}
\|u\|_{X^{s,b}_{p,q}(\R\times \T)}=\big\|\| \jb{n}^{s}\jb{\tau -|n|^{2\al} }^{b}\, \ft{u}(\tau, n)\|_{L^{q}_{\tau}(\R)}\big\|_{l^{p}_{n}(\Z)}, 
\end{align*}
where $\ft{u}(\tau, n)$ denotes the space-time Fourier transform of $u(t,x)$. 
When $p=q=2$, the $X^{s,b}_{2,2}=X^{s,b}$ spaces reduce to the standard $L^2$-based $X^{s,b}$-spaces~\cite{Bou, KlM}. We will make heavy use of the case when $p=2$ and $q\geq 2$, and as such we will define $X^{s,b}_{2,q}=X^{s,b}_{q}$.
These spaces enjoy the following embedding properties:
for every $1 \le p \le \wt p \le \infty$,
\begin{equation}\label{Xsbpembedding}
\|u\|_{X^{s,b}_{\wt p,q}} \le \|u\|_{X^{s,b}_{p,q}},
\end{equation}
which follows from the elementary estimate for sequences $\|f\|_{l^{\wt p}_{n}} \le \| f\|_{l^p_{n}}$.

Given $T>0$, we define the local-in-time version $X^{s,b}_{p,q}([0,T]\times \T)$ of $X^{s,b}_{p,q}(\R\times \T)$ as 
\begin{align}
X^{s,b}_q([0,T]\times \T)=\inf\{ \|v\|_{X^{s,b}_{p,q}(\mathbb{R}\times \T)} \, : \, v|_{[0,T]}=u \}. \label{timeloc0}
\end{align}
We will denote by $X^{s,b}_{p,q}(T)$ the spaces $X^{s,b}_{p,q}([0,T]\times \T)$. 
We have the following important embedding:
for any $s\in \R$ and $b>\frac{1}{q'}$, we have
\begin{align}
X^{s,b}_{p,q}(T) \hookrightarrow C([0,T];\FL^{s,p}(\T)). \label{ctsembed}
\end{align}
Given any function $F$ on $[0, T]\times \T$, we denote by $\tilde{F}$ any extension of $F$ onto $\R\times \T$.
Let $\eta\in C_{c}^{\infty}(\R)$ be a smooth non-negative cutoff function supported on $[-2,2]$ with $\eta\equiv 1$ on $[-1,1]$, and we set $\eta_{T}(t):=\eta(T^{-1}t)$ for $T>0$. 
For $0<T\leq 1$, we have the point-wise bounds
\begin{align}
|\ft \eta_{T}(\tau)|+|\ft \chi_{T}(\tau)| \les \frac{T^{\ta}}{\jb{\tau}^{1-\ta}}, \quad 0\leq \ta \leq 1, \label{etabd}
\end{align}
where $\chi (t)=\chi_{[-1,1]}(t)$ is a sharp cut-off to the interval $[-1,1]$ and, given $T>0$, we define $\chi_{T}(t)=\chi_{[-1,1]}(T^{-1}t)=\chi_{[-T,T]}(t)$.

We recall the following linear estimates related to the $X^{s,b}_{p,q}(T)$ spaces.

\begin{lemma}
The following are true:
\begin{enumerate}[\normalfont (i)]
\setlength\itemsep{0.3em}
\item \textup{ [Homogeneous linear estimate]} Given $s,b\in \R$ and $1\leq p,q\leq \infty$, we have 
\begin{align}
\| S(t)u_0\|_{X^{s,b}_{p,q}(T)} \les \|u_0\|_{\FL^{s,p}}, \label{linest}
\end{align}
for any $0<T\leq 1$.
\item \textup{ [Nonhomogeneous linear estimate]} Let $s\in \R$, $1\leq p,q\leq \infty$ and $b>\tfrac{1}{q'}$. Then we have
 \begin{align}
\bigg\| \int_{0}^{t} S(t-t')F(t')dt' \bigg\|_{X^{s,b}_{p,q}(T)} \les  \|F\|_{X^{s,b-1}_{p,q}(T)}, \label{duhamelest}
\end{align}
 for any $0<T\leq 1$. 
\end{enumerate}
  
\end{lemma}

\begin{proof}
The first estimate \eqref{linest} can be found in \cite[Lemma 2.2 (i)]{ofw20}. For the second estimate \eqref{duhamelest}, we note that for any extension $\wt F$ of $F$ onto $\R$, 
\begin{align*}
\I \wt{F}(t,x):= \eta(t) \int_{0}^{t} S(t-t')\eta(t') \wt{F}(t',x)dt' 
\end{align*}
is an extension of $\int_{0}^{t} S(t-t')F(t',x)dt'$, and so it suffices to prove 
\begin{align}
\| \I \wt{F} \|_{X^{0,b}_{p,q}} \les \| \wt F\|_{X^{0,b-1}_{p,q}}. \label{inhomog2}
\end{align}
Taking a space-time Fourier transform, we find
\begin{align*}
\mathcal{F}_{t,x} \{  \I \wt{F} \}(\tau+|n|^{2\al},n) = \int_{\R} K(\tau, \ld)\wt F(\ld+|n|^{2\al},n)d\ld,
\end{align*}
where the kernel $K(\tau,\ld)$ satisfies
\begin{align}
|K(\tau, \ld)| \les \tfrac{1}{\jb{\ld}} \big(  \tfrac{1}{\jb{\tau-\ld}^{M}} +\tfrac{1}{\jb{\tau}^{M}}\big), \label{kernel}
\end{align}
for any $M \geq 1$.
See, for example, \cite[Lemma 3.1]{DNY1}. Using \eqref{kernel}, we see that 
\begin{align*}
 \jb{\tau}^{b} \jb{\ld}^{1-b} |K(\tau,\ld)| \les \frac{1}{ ( \jb{\tau}\wedge \jb{\ld})^{b} (\jb{\tau}\vee \jb{\ld})^{M}} + \frac{1}{ \jb{\tau-\ld}^{M}} \ind_{ \jb{\tau}\sim \jb{\ld}}.   
\end{align*}
By H\"{o}lder's inequality (which requires $b>\tfrac{1}{q'}$) and Young's inequality, we obtain
\begin{align*}
\Big\|   \int_{\R}  \jb{\tau}^{b} \jb{\ld}^{1-b} K(\tau,\ld) G(\ld) d\ld \Big\|_{L^q_{\tau}} \les \|G\|_{L^q_{\tau}},
\end{align*}
which implies \eqref{inhomog2}.
\end{proof}

We also have the following properties for the $X^{s,b}_{p,q}$ spaces related to localisation in time.

\begin{lemma}
 For any $s\in \R$, $1\leq p,q\leq \infty$ and $-\tfrac{1}{q}<b'\leq b<\tfrac{1}{q'}$, we have 
 \begin{align}
\|u\|_{X^{s,b'}_{p,q}(T)} &\les T^{b-b'}\|u\|_{X^{s,b}_{p,q}(T)}, \label{timeloc}
\end{align}
for any $0<T\leq 1$. In particular, \eqref{timeloc} also holds for $-\frac 1q < b' \le b < 1+\tfrac{1}{q'}$, under the additional condition that $u(0) = 0$.
Moreover, for any $s\in \R$, $1\leq p \leq \infty$, $q\geq 2$ and $0\leq b<\tfrac{1}{q'}$,
\begin{align}
\|\eta_{T}u\|_{X^{s,b}_{p,q}} &\les T^{\frac{1}{q'}-b-\dl}\|u\|_{X^{s,\frac{1}{q'}-\dl}_{p,q}}, \label{timeloc2} \\
\|\chi_{T} u\|_{X^{s,b}_{p,q}} & \les T^{\frac{1}{q'}-b-2\dl}\|u\|_{X^{s,\frac{1}{q'}-\dl}_{p,q}(T)},\label{timeloc4}
\end{align}
where $0<T\leq 1$ and $\dl>0$ is sufficiently small. Finally, for any $s\in \R$, $b\geq 0$, $1\leq p\leq \infty$, $1\leq q\leq r\leq \infty$, and $0<T\leq 1$, we have 
\begin{align}
\|u\|_{X^{s,b}_{p,r}(T)} \les \|u\|_{X^{s,b}_{p,q}(T)}
\label{q'+}
\end{align}
\end{lemma}
\begin{proof}
The first estimate \eqref{timeloc} follows a similar proof as the case when $q=2$, which can be found in \cite[Lemma 2.11]{TAO}, and \eqref{timeloc2} is just a special case of \eqref{timeloc}. When $u(0)=0$, the additional claim regarding \eqref{timeloc} can be proven by slight modification of \cite[Lemma 3.2]{DNY1}.

We now discuss the estimate \eqref{timeloc4}. This follows from the inequality:
\begin{align}
\|\chi_{T} u\|_{X^{s,b}_{p,q}} & \les T^{\frac{1}{q'}-b-2\dl}\|u\|_{X^{s,\frac{1}{q'}-\dl}_{p,q}}. \label{timeloc3}
\end{align}
This estimate holds without the extra $\dl$ loss in the case $q=2$ leveraging the Plancherel theorem~\cite{DBD}. 
When $q>2$, we can settle for the $\dl$-loss, and \eqref{timeloc3} follows from interpolating between 
\begin{align*}
\| \chi_{T} u\|_{X^{s,0}_{p,q}} \les T^{\frac{1}{q'}-2\dl} \|u\|_{X^{s,\frac{1}{q'}-\dl}_{p,q}}, \quad \text{and} \quad \| \chi_{T} u\|_{X^{s,b}_{p,q}} \les \|u\|_{X^{0,b+\dl}_{p,q}},
\end{align*}
for any $s\in \R$, $1\leq p\leq \infty$, $q> 2$, $b<\tfrac{1}{q'}$ and $0<T\leq 1$. Now, both of these inequalities follow readily by using \eqref{etabd}. 
As for \eqref{timeloc4}, using \eqref{timeloc3} and the definition \eqref{timeloc0}, if $w$ is any extension of $u$ onto $[0,T]$, then
\begin{align*}
\|\chi_{T} w\|_{X^{s,b}_{p,q}} = \|\chi_T w\|_{X^{s,b}_{p,q}} \les T^{\frac{1}{q'}-b-2\dl}\|w\|_{X^{s,\frac{1}{q'}-\dl}_{p,q}}.
\end{align*}
Now we take an infimum over all such extensions $w$ to obtain \eqref{timeloc4}.

We now move onto \eqref{q'+}. It suffices to show that 
\begin{align*}
\| \eta u\|_{X^{s,b}_{p,r}} \les \|u\|_{X^{s,b}_{p,q}}.
\end{align*}
 which in turn follows from the following inequality: for any $1\leq q<r\leq \infty$, and $b\geq 0$  that 
\begin{align}
 \| \jb{\tau -\mu}^{b}  ( \ft{\eta} \ast \ft f)\|_{L^r_{\tau}} \les \|\jb{\tau-\mu}^{b}f\|_{L^q_{\tau}}, \label{q'+1}
\end{align}
with implicit constant independent of $\mu\in \R$. If $b=0$, \eqref{q'+1} follows from Young's inequality and \eqref{etabd}. So suppose that $b>0$ and fix $\mu\in \R$. We write 
\begin{align*}
\| \jb{\tau -\mu}^{b}  ( \ft{\eta} \ast \ft f)\|_{L^r_{\tau}} = \bigg\| \intt_{\tau_1+\tau_2=\tau}  \jb{\tau-\mu}^{b} \ft \eta (\tau_1) \ft f(\tau_2)d\tau_1 \bigg\|_{L^r_{\tau}}.
\end{align*}
If $\jb{\tau_2-\mu} \ges \jb{\tau-\mu}$, then we proceed as in the case $b=0$ using Young's inequality. Otherwise, if $\jb{\tau_2-\mu} \ll \jb{\tau-\mu}$, then $\tau_1= \tau-\tau_2=(\tau-\mu)-(\tau_2-\mu)$ implies that $\jb{\tau_1}\sim \jb{\tau-\mu}$. Therefore,
\begin{align*}
\tfrac{\jb{\tau-\mu}^{b} |\eta(\tau_1)|}{\jb{\tau_2-\mu}^{b}} \les \tfrac{1}{\jb{\tau-\mu}^{M} \jb{\tau_2-\mu}^{M}},
\end{align*}
for any $M>0$, at which point, \eqref{q'+1} follows from H\"{o}lder's inequality.
\end{proof}

\subsection{The local well-posedness result}

The remainder of this section is to show the following local well-posedness result. See also Proposition~\ref{PROP:LWPN} for a similar statement with regards to a truncated version of \eqref{FNLSv}.

\begin{proposition}\label{PROP:LWP}
Let $\al>1$ and $\tfrac{1}{4}<s\leq \tfrac{1}{2}$ be such that $\al+s>\tfrac{3}{2}$ and \begin{align}
 2 \le q <\tfrac{4\al}{3-2s}. \label{LWPqbd}
\end{align}
Given $u_0$ distributed according to $\mu_s$, we define
$$ \Xi(u_0) := \big(u_0, J_{12}, J_{13}, X^{(3)} \big), $$
where $J_{12}, J_{13}$ are defined in \eqref{J12def} and \eqref{J13def}, $X^{(3)}$ is defined in \eqref{X3def}, and we define a norm
\begin{align}
\begin{split}
\|\Xi(u_0)\|_{\mathcal X^{\sigma,q}}:=& \|u_0\|_{H^{s-\frac 12 -}} + \|u_0\|_{\FL^{\frac{\s}{3}, 6}} + \|X^{(3)}\|_{X^{\sigma,\frac 1{q'} +}_q(1)}\\
& \hphantom{\|u_0\|_{H^{s-\frac 12 -}}}+ \|J_{12}\|_{b_J+, q, \sigma,-\sigma} +  \|J_{13}\|_{b_J+,q, \sigma, -\sigma} ,
\end{split}\label{normxidef}
\end{align}
where $b_J$ is defined in \eqref{bilinearlwpcondition} and $\s\in \R$.
Let $v_0\in H^{s}(\T)$. Then, the equation \begin{equation}\label{FNLSv}
\begin{cases}
i\partial_t v+(-\dx^2)^\alpha v+\NN( S(t)(u_0+v_0)+v)+\RR(S(t)(u_0+v_0)+v)=0 ,\\
v|_{t=0}=0,
\end{cases}
\end{equation}
is locally-well posed in $L^2(\T)$. Moreover, if $\tau_0$ is defined by 
\begin{align} \label{tau0def}
\tau_0(u_0,v_0)=\max\{ t\geq 0\,:\, \|v\|_{X^{0,\frac{1}{q'}+}_{q}(T)} \les 1+\|v_0\|_{L^2}\},
\end{align}
then there exists $C_{\eps}>0$ such that 
\begin{align}
\tau_{0}(u_0,v_0)^{-1} \les 1 + \|v_0\|_{L^2}^{\frac{4\al}{2\al-1}+} +\|\Xi(u_0)\|_{\mathcal X^{0,q}}^{C_{\eps}}. \label{tau0bound}
\end{align}
Furthermore, for $0 \le \sigma \le s$, we define 
\begin{align*} 
\tau_\sigma(u_0,v_0)=\max\{ 0 \le t \le \tau_{0}(u_0,v_0)\,:\, \|v\|_{X^{\sigma,\frac{1}{q'}+}_{q}(T)} \les 1+\|v_0\|_{H^\sigma}\},
\end{align*}
and there exists $C_{\eps,\sigma}>0$ such that 
\begin{align}
\tau_{\sigma}(u_0,v_0)^{-1} \les 1+\|v_0\|_{L^2}^{\frac{4\al}{2\al-1}+} +\|\Xi(u_0)\|_{\mathcal X^{0,q}}^{C_{\eps}}+\|\Xi(u_0)\|_{\mathcal X^{\sigma,q}}^{C_{\eps,\sigma}}. \label{tausigmabound}
\end{align}
\end{proposition}

Let us make a few remarks about this local well-posedness result. 
Heuristically, if $v$ is a solution to \eqref{FNLSv}, then 
\begin{align*}
u = S(t)u_0 +S(t)v_0+v(t)
\end{align*}
is a solution to \eqref{alphaNLS} with initial data given by the sum $u_0+v_0$, where $u_0 \sim \mu_s$ and $v_0$ belongs to the Cameron-Martin space of $\mu_s$; namely, $H^{s}(\T)$. The presence of $v_0$ here and the tracking of the dependence of the local existence times \eqref{tau0bound} and \eqref{tausigmabound} is a crucial part of our globalisation argument; indeed, see Lemma~\ref{lem:tauintegrability} and Proposition~\ref{LEM:densityontau}. Moreover, for $\al$ close to $1$, we need to exploit the fact that solutions to \eqref{FNLSv} belong to  $H^{\s}(\T)$ in space for some $\s > 0$, and not just $L^2(\T)$. The maximum possible upper bound on this additional regularity is dictated by the perturbation $v_0$. 

It is also crucial that we propagate this additional regularity of $v$ by a persistence of regularity argument, so that the smaller time of existence $\tau_{\s}(u_0,v_0)$ only depends on the $L^2$-norm of $v_0$ and not on its $H^{\s}$-norm. 

Finally, in order to obtain the numerology in Theorem~\ref{thm:qiandgwp}, it is essential that we aim to prove nonlinear estimates with the best, or optimal, exponents of the local-existence time $T$ that we can. See for instance \eqref{TconditionlistNN}. These exponents then trickle through to the exponent on the norm $\|v_0\|_{L^2}$ in \eqref{tau0bound} and \eqref{tausigmabound}.

\subsection{Nonlinear estimates}

In order to prove Proposition~\ref{PROP:LWP}, we need estimates on the nonlinear terms in \eqref{FNLSv}, which were defined in \eqref{nonres} and \eqref{res}.
As we wish to obtain results below the known deterministic well-posedness threshold, we first need to establish a probabilistic local well-posedness result for \eqref{alphaNLS}.
Namely, we seek to exploit probabilistic cancellations. Writing $w(t)=S(t)v_0 +v(t)$, the multi-linearity of $\mathcal{N}$ implies that $\mathcal{N}(S(t)u_0+w(t))$ splits into four kinds of pieces: (i) trilinear in $S(t)u_0$, (ii) bilinear in $S(t)u_0$ and linear in $w(t)$, (iii) linear in $S(t)u_0$ and bilinear in $w(t)$, and (iv) trilinear in $w(t)$. We establish estimates for the contributions (i) and (ii) probabilistically in Lemma~\ref{randomtrilwp} and Proposition~\ref{PROP:bilinearlwp}. Proposition~\ref{PROP:detquad} states deterministic estimates which are suitable to handle (iii) and (iv).

We begin with the simpler estimates for the resonant term $\Res$. We also need to exploit multi-linearity of $\Res$ to place $u_0$ in an appropriate Fourier-Lebesgue space.

\begin{proposition}
For any $0 \le \sigma \le s$, we have the following estimates. 
\begin{gather}
\int_\R \Res(u_1(t),u_2(t),u_3(t),u_4(t)) dt \les \| u_1 \|_{X^{0,\frac1{q'}+}_{\infty, q}}\| u_2 \|_{X^{0,\frac1{q'}+}_{\infty, q}} \|u_3\|_{X^{\sigma,0}_q} \|u_4\|_{X^{-\sigma,0}_{q'}} \label{res12},\\
\int_\R \Res(u_1(t),u_2(t),u_3(t),u_4(t)) dt \les \| u_1 \|_{X^{0,\frac1{q'}+}_{\infty, q}}\|u_2\|_{X^{\sigma,\frac1{q'}+}_q} \| u_3 \|_{X^{0,0}_{\infty, q}}  \|u_4\|_{X^{-\sigma,0}_{q'}} \label{res13},\\
\int_\R \Res(u_1(t),u_2(t),u_3(t),u_4(t)) dt \les \| u_1 \|_{X^{\frac \sigma 3,\frac1{q'}+}_{6,q}}\| u_2 \|_{X^{\frac \sigma 3,\frac1{q'}+}_{6,q}} \|u_3\|_{X^{\frac \sigma 3,0}_{6,q}} \|u_4\|_{X^{-\sigma,0}_{q'}} \label{resrandom}.
\end{gather}
\end{proposition}
\begin{proof}
We have that 
\begin{align*}
&\int_\R \Res(u_1(t),u_2(t),u_3(t),u_4(t)) dt \\
=& \sum_{n \in \Z} \int_{\tau_1-\tau_2+\tau_3-\tau_4=0} \ft{u_1}(n,\tau_1) \cj{\ft{u_2}(n,\tau_2)}\ft{u_3}(n,\tau_3)\cj{\ft{u_4}(n,\tau_4) }d \tau_1d \tau_2d \tau_3d \tau_4.
\end{align*}
From the inequality 
$$\int_{t_1-t_2+t_3-t_4=a} \frac{f_1(t_1)}{\jb{t_1}^{\frac1{q'}+}}\frac{f_2(t_2)}{\jb{t_2}^{\frac1{q'}+}}{f_3(t_3)}{f_4(t_4)} d t_1 d t_2 d t_3 d t_4 \les \prod_{j=1}^3 \| f_j \|_{L^q}\|f_4\|_{L^{q'}},$$
where the implicit constant is independent of $a\in \R$, we obtain that 
\begin{align*}
&\int_\R \Res(u_1(t),u_2(t),u_3(t),u_4(t)) dt \\
\les& \sum_{n \in \Z} \| \jb{\tau-|n|^{2\al}}^{\frac1{q'}+}\, \ft{u_1}(n,\tau)\|_{L^q_\tau}\|\jb{\tau-|n|^{2\al}}^{\frac1{q'}+}\, \ft{u_2}(n,\tau)\|_{L^q_\tau}\| \ft{u_3}(n,\tau)\|_{L^q_\tau}\|\ft{u_4}(n,\tau)\|_{L^{q'}_\tau}.
\end{align*}
Therefore, \eqref{res12}, \eqref{res13}, \eqref{resrandom} follow respectively by the inequalities 
\begin{align*}
\sum_n f_1(n)f_2(n){f_3(n)}f_4(n) &\les \|f_1\|_{l_{n}^\infty}\|f_2\|_{l_{n}^\infty}\| \jb{n}^{\s} f_3\|_{l_{n}^2}\| \jb{n}^{-\s}f_4\|_{l_{n}^2}, \\
\sum_n f_1(n)f_2(n){f_3(n)}f_4(n)  &\les \|f_1\|_{l_{n}^\infty}\|\jb{n}^{\s}f_2\|_{l_{n}^2}\|f_3\|_{l_{n}^\infty}\|\jb{n}^{-\s}f_4\|_{l_{n}^2}, \\
\sum_n f_1(n)f_2(n){f_3(n)}f_4(n)  &\les \prod_{j=1}^{3} \| \jb{n}^{\frac{\s}{3}}f_j\|_{l_{n}^6} \|\jb{n}^{-\s}f_4\|_{l_{n}^2}, 
\end{align*}
which can be all proven by using H\"older.
\end{proof}

We now control the non-resonant part $\mathcal{N}$ of the nonlinearity and begin by proving an estimate on the contribution which is trilinear in $S(t)u_0$.

\begin{proposition}[Random trilinear estimate] \label{randomtrilwp}
Let 
\begin{equation}
 X^{(3)}(t) = \int_0^t S(t-t') \NN(S(t') u_0) dt'. \label{X3def}
\end{equation}
Then, for every $\alpha > 1$, $\frac 14 < s \leq  \frac 12$, $0 \le \sigma \le s$, and 
\begin{align}
2\leq q< \tfrac{4\al}{3+2\s-6s} \wedge \tfrac{2(2\al-1)}{1+2\s-2s}, \label{qbdX3}
\end{align} we have that 
\begin{equation}
\E \Big[ \|X^{(3)} \|_{X^{\sigma, \frac{1}{q'}+ }_q(1)}^p  \Big] < \infty,\label{X3integrability}
\end{equation} for every $1\leq p<\infty$.
Moreover, if 
\begin{equation}
 X_N^{(3)}(t) := \int_0^t S(t-t') \NN_N(S(t') u_0) dt', \label{X3Ndef}
\end{equation}
then under the same conditions on $\alpha, s, \sigma$ and $q$, there exists $\delta = \delta(\alpha, s, \sigma,q) > 0$ such that for every $1\leq p < \infty$,
\begin{equation}
 \E \bigg[  \|X_N^{(3)} - X^{(3)} \|_{X^{\sigma,\frac 1{q'}+ }_q(1)}^p \bigg] \les N^{-\delta p}. \label{X3convergence}
\end{equation}
\end{proposition}

\begin{proof}
In order to prove \eqref{X3integrability}, it suffices to show
$$  \E \Big[ \|\eta X^{(3)}\|_{X^{\sigma,b}_q}^p \Big] < \infty,$$
where $b=\tfrac{1}{q'}+$ and $p \ge q$.
By definition, we have 
\begin{equation}
 X^{(3)}(t) = \sum_{n \in \Z} \sum_{\substack{n_1-n_2+n_3 = n, \\ n_2 \neq n_1,n_3}} \frac{g_{n_1}}{\jb{n_1}^s}\frac{\overline{g_{n_2}}}{\jb{n_2}^s}\frac{g_{n_3}}{\jb{n_3}^s} \frac{e^{it\Phi_\alpha{(\cj n)}}-1}{i\Phi_\alpha{(\cj n)}} e^{it|n|^{2\alpha} + i n x}, \label{X3}
\end{equation}
from which we obtain that
\begin{equation*}
\widehat{\eta X^{(3)}}(\tau,n) =  \sum_{\substack{n_1-n_2+n_3 = n, \\ n_2 \neq n_1,n_3}} \frac{g_{n_1}}{\jb{n_1}^s}\frac{\overline{g_{n_2}}}{\jb{n_2}^s}\frac{g_{n_3}}{\jb{n_3}^s} \frac{\ft \eta(\tau -\Phi_\alpha{(\cj n)} - |n|^{2\alpha}) - \ft\eta(\tau - |n|^{2\alpha})}{i\Phi_\alpha{(\cj n)}}. 
\end{equation*}
Notice that $\widehat{\eta X^{(3)}}(\tau,n)$ is (formally) a polynomial of degree $3$ in $\{g_n\}$. Therefore, by Lemma~\ref{LEM:WCE},
\begin{align*}
\E \Big[  \big|\widehat{\eta X^{(3)}}(\tau,n)\big|^{p} \Big] \les_p \Big( \E  \Big[ \big|\widehat{\eta X^{(3)}}(n,\tau)\big|^{2} \Big]\Big)^\frac p2.
\end{align*}
Moreover, by independence of the $\{g_n\}$, together with the condition $n_2 \neq n_1,n_3$, we have that 
\begin{equation*}
\E\Big[ |\widehat{\eta X^{(3)}}(\tau,n)|^{2} \Big] = \sum_{\substack{n_1-n_2+n_3 = n, \\ n_2 \neq n_1,n_3}} \frac{|\ft \eta(\tau -\Phi_\alpha{(\cj n)} - |n|^{2\alpha}) - \ft\eta(\tau - |n|^{2\alpha})|^2}{\jb{n_1}^{2s}\jb{n_2}^{2s}\jb{n_3}^{2s} |\Phi_\alpha{(\cj n)}|^2}.
\end{equation*}
Therefore, by definition of the $X^{s,b}_q$-norm and Minkowski's integral inequality, for $p \ge q\ge2$ we have that 
\begin{align*}
 &\E \Big[ \|\eta X^{(3)}\|_{X^{\sigma,b}_q}^p  \Big]^\frac2p  \\
 &= \E \bigg[  \Big| \sum_{n \in \Z} \jb{n}^{2\sigma} \Big|\int_\R \Big|\widehat{\eta X^{(3)}}(\tau,n)\Big|^{q}\jb{\tau-|n|^{2\alpha}}^{qb}d\tau\Big|^{\frac 2q}\Big|^\frac p2  \bigg] ^\frac2p\\
 &\le  \sum_{n \in \Z}\jb{n}^{2\sigma}\E \bigg[\Big|\int_\R \Big|\widehat{\eta X^{(3)}}(\tau,n)\Big|^{q}\jb{\tau-|n|^{2\alpha}}^{qb}d\tau\Big|^{\frac {p}q} \Big| \bigg]^\frac2p \\
 &\le  \sum_{n \in \Z}\jb{n}^{2\sigma}\Big|\int_\R  \E \Big[  \Big|\widehat{\eta X^{(3)}}(\tau,n)\Big|^{p}\Big]^\frac qp   \jb{\tau-|n|^{2\alpha}}^{qb}d\tau\Big|^\frac2q\\
 &\les_p \sum_{n \in \Z}\jb{n}^{2\sigma}\Big|\int_\R \E\Big[ \Big|\widehat{\eta X^{(3)}}(\tau,n)\Big|^{2}\Big]^{\frac q2}\jb{\tau-|n|^{2\alpha}}^{qb} d\tau \Big|^\frac2q\\
 &= \sum_{n \in \Z}\jb{n}^{2\sigma}\bigg(\int_\R \bigg[ \sum_{\substack{n_1-n_2+n_3 = n, \\ n_2 \neq n_1,n_3}} \frac{|\ft \eta(\tau -\Phi_\alpha{(\cj n)} - |n|^{2\alpha}) - \ft\eta(\tau - |n|^{2\alpha})|^2}{\jb{n_1}^{2s}\jb{n_2}^{2s}\jb{n_3}^{2s} |\Phi_\alpha{(\cj n)}|^2}\bigg]^{\frac q2}\jb{\tau-|n|^{2\alpha}}^{qb} d\tau \bigg)^\frac2q\\
 &\les \sum_{n \in \Z}\jb{n}^{2\sigma}\sum_{\substack{n_1-n_2+n_3 = n, \\ n_2 \neq n_1,n_3}} \Big|\int_\R \Big|  \frac{|\ft \eta(\tau -\Phi_\alpha{(\cj n)} - |n|^{2\alpha}) - \ft\eta(\tau - |n|^{2\alpha})|^2}{\jb{n_1}^{2s}\jb{n_2}^{2s}\jb{n_3}^{2s} |\Phi_\alpha{(\cj n)}|^2}\Big|^{\frac q2}\jb{\tau-|n|^{2\alpha}}^{qb} d\tau \Big|^\frac2q\\
 &\les \sum_{\substack{n_1-n_2+n_3 - n_4 = 0, \\ n_2 \neq n_1,n_3}} \frac{\jb{n_4}^{2\sigma}}{\jb{\Phi_\alpha{(\cj n)}}^{2(1-b)}\jb{n_1}^{2s}\jb{n_2}^{2s}\jb{n_3}^{2s}}\\
 &\les \sum_{\substack{n_1-n_2+n_3 - n_4 = 0, \\ n_2 \neq n_1,n_3}} \frac{\jb{n_4}^{2\sigma}}{\jb{n_1-n_2}^{2(1-b)}\jb{n_2-n_3}^{2(1-b)}\jb{n_{\max}}^{4(\alpha-1)(1-b)}\jb{n_1}^{2s}\jb{n_2}^{2s}\jb{n_3}^{2s}} \\
 &\les \sum_{\substack{N_1,N_2,N_3 \\ \text{dyadic}}} (N_1N_2N_3)^{-2s}N_{\max}^{-4(\al-1)(1-b)+2\sigma}\sum_{\substack{n_1-n_2+n_3 - n_4 = 0, \\ n_2 \neq n_1,n_3, \\ \jb{n_j} \sim N_j}} \frac{1}{\jb{n_1-n_2}^{2(1-b)}\jb{n_2-n_3}^{2(1-b)}} \\
 &\les \sum_{\substack{N_1,N_2,N_3 \\ \text{dyadic}}} (N_1N_2N_3)^{-2s}N_{\max}^{-4(\al-1)(1-b)+2\sigma}N_1N_2N_3 (N_1 \vee N_2)^{-2(1-b)} (N_2 \vee N_3)^{-2(1-b)},
\end{align*}
which is finite
as long as 
\begin{gather*}
4(\al-1)(1-b)+2(1-b)+2s-1-2\s >0, \quad \text{and} \quad  4\al(1-b)+ 6s -3- 2\s > 0.
\end{gather*}
These conditions yield \eqref{qbdX3}.

In order to show \eqref{X3convergence}, we simply note that $X^{(3)}_{N}(t)$ obeys the same formula as $X^{(3)}(t)$ in \eqref{X3} but with the extra restrictions that $|n_j| \le N$ for $1 \le j \le 4$.
 Therefore, proceeding in the same way, we obtain that 
\begin{align*}
&\E\Big[ \|X^{(3)} - X_N^{(3)} \|_{X^{\sigma,b }_{q}(1)}^p\Big]^{\frac 2p} \\
&\les \sum_{\substack{N_1,N_2,N_3, \\ N_{\max} \gtrsim N}} (N_1N_2N_3)^{-2s}N_{\max}^{-4(\al-1)(1-b)+2\sigma}N_1N_2N_3 (N_1 \vee N_2)^{-2(1-b)} (N_2 \vee N_3)^{-2(1-b)} \\
&\les N^{-2\delta}
\end{align*}
for $\delta \ll 1$ such that 
\begin{gather*}
4(\al-1)(1-b)+2(1-b)+2s-1-2\s >2\dl, \quad \text{and} \quad  4\al(1-b)+ 6s -3- 2\s>2\dl.\qedhere
\end{gather*} 
\end{proof}

The following estimate controls the contributions in the nonresonant part of the nonlinearity which are bilinear in $S(t)u_0$ and linear in $S(t)v_0+v$.

\begin{proposition}[Random bilinear estimate]\label{PROP:bilinearlwp}
Define the random bilinear forms $J_{12}(t)$ and $J_{13}(t)$ by 
\begin{gather}
J_{12}(t)[v_1,v_2] := \NN(S(t)u_0,S(t)u_0,v_1,v_2),\label{J12def} \\
J_{13}(t)[v_1,v_2] := \NN(S(t)u_0,v_1,S(t)u_0,v_2),\label{J13def} 
\end{gather}
and let $\|J\|_{b,q,\sigma,-\sigma}$ be the best constant $C$ such that 
\begin{equation*}
\int_\R J(t)[v_1(t),v_2(t)] d t \le C\big( \|v_1\|_{X^{\sigma,0}_{q}} \|v_2\|_{X^{-\sigma,b}_{q'+}} + \|v_1\|_{X^{\sigma,b}_q} \|v_2\|_{X^{-\sigma,0}_{q'+}}  \big),
\end{equation*}
for test functions $v_1,v_2$. Then for every $1\leq p< \infty$, $0 \le \sigma < s$, $1< q < \infty$,
\begin{equation}\label{Jintegrability}
\E \big[ \|J_{12}\|_{b,q,\sigma,-\sigma}^p +\|J_{13}\|_{b,q,\sigma,-\sigma}^p \big] < \infty,
\end{equation}
as long as 
\begin{equation}\label{bilinearlwpcondition}
b_J:=\tfrac{1}{2\al} < b < \tfrac 12.
\end{equation}
Moreover, defining
\begin{gather}
J_{12,N}(t)[v_1,v_2] := \NN_N(  S(t)u_0, S(t)u_0,v_1,v_2),\label{J12Ndef} \\
J_{13,N}(t)[v_1,v_2] := \NN_N( S(t)u_0,v_1,S(t)u_0,v_2),\label{J13Ndef} 
\end{gather}
then, under the same hypotheses on $\sigma, b, q$, there exists $\delta = \delta(\sigma, b, q) > 0$ so that for every $p < \infty$, 
\begin{equation}\label{Jconvergence}
\E\big[ \|J_{12}-J_{12,N}\|_{b,q,\sigma,-\sigma}^p+\|J_{13}-J_{13,N}\|_{b,q,\sigma,-\sigma}^p \big] \les N^{-p\delta}.
\end{equation}
\end{proposition}
\begin{proof}
We start with showing the estimate for $J_{12}$. Since 
\begin{equation}\label{J12formula}
\begin{aligned}
 &\int_\R J_{12} (t)[v_1(t),v_2(t)]  \\
 &=  \sum_{\substack{n_1-n_2+n_3 - n_4 = 0 \\ n_2 \neq n_1,n_3}} \int_{\tau_1-\tau_2= |n_2|^{2\alpha}-|n_1|^{2\alpha}}\frac{g_{n_1}}{\jb{n_1}^s}\frac{\overline{g_{n_2}}}{\jb{n_2}^s} \ft{v_1}(n_3, \tau_1) \overline{\ft{v_2}}(n_4,\tau_2) d\tau_1d\tau_2,
\end{aligned}
\end{equation}
writing $\wt{\tau}_j = \tau_j - |n_{j+2}|^{2\alpha}$, the condition $\tau_1 - \tau_2 =  |n_2|^{2\alpha}-|n_1|^{2\alpha}$ becomes $\wt\tau_1 - \wt\tau_2 = - \Phi_\alpha{(\cj n)}$. Therefore, by symmetry between $\ft{v_1}$ and $\overline {\ft{v_2}}$, $n_1$ and $n_2$, and $q$ and $q'$, it is enough to show that 
\begin{align}
\begin{split}
 \int_{\R}   \sum_{\substack{n_1-n_2+n_3 - n_4 = 0, \\ n_2 \neq n_1,n_3}} \ind_{\{|\wt\tau_1|\gtrsim \jb{\Phi_\alpha{(\cj n)}}\}}  \frac{\jb{n_4}^\sigma}{\jb{n_3}^\sigma}\frac{g_{n_1}}{\jb{n_1}^s}\frac{\overline{g_{n_2}}}{\jb{n_2}^s} &\frac{f(n_3, \wt\tau_1)}{\jb{\wt\tau_1}^b}g(n_4,\wt\tau_1+\Phi_\alpha{(\cj n)})  d\wt\tau_1 \\
 &  \le C \|f\|_{\l^2_n L^q_{\wt \tau}} \|g\|_{l^{2}_nL^{q'}_{\wt \tau}}, 
 \end{split} \label{Bilin1}
\end{align}
with $\E[ |C|^p]< \infty$ for every $p < \infty$. Fix $b' < b$ such that \eqref{bilinearlwpcondition} holds for $b'$ as well.
For each fixed $\wt\tau_1\in \R$, we define the kernel 
\begin{equation*}
K_{\wt\tau_1}(n_3,n_4) := \sum_{\substack{n_1-n_2+n_3 - n_4 = 0, \\ n_1 \neq n_2,n_4}} \ind_{\{|\wt\tau_1|\gtrsim \jb{\Phi_\alpha{(\cj n)}}\}} \frac{\jb{n_4}^\sigma}{\jb{n_3}^\sigma}\frac{g_{n_1}\overline{g_{n_2}}}{\jb{n_1}^s\jb{n_2}^s\jb{\wt\tau_1}^{b'}}.
\end{equation*}
It is convenient to split this kernel into dyadic components. Henceforth, we define for $N_1, N_2$ dyadic,
\begin{equation*}
 K_{\wt\tau_1}^{N_1,N_2}(n_3,n_4) = \sum_{\substack{n_1-n_2+n_3 - n_4 = 0, \\ n_1 \neq n_2,n_4, \\ \jb{n_1} \sim N_1, \jb{n_2} \sim N_2}}  \ind_{\{|\wt\tau_1|\gtrsim \jb{\Phi_\alpha{(\cj n)}}\}}  \frac{\jb{n_4}^\sigma}{\jb{n_3}^\sigma}\frac{g_{n_1}\overline{g_{n_2}}}{\jb{n_1}^s\jb{n_2}^s\jb{\wt\tau_1}^{b'}}. 
\end{equation*}
We obtain that 
\begin{align*}
&\int\sum_{\substack{n_1-n_2+n_3 - n_4 = 0, \\ n_2 \neq n_1,n_3 \\ \jb{n_j}\sim N_j, \,\, j=1,2,3,4}}  \ind_{\{|\wt\tau_1|\gtrsim \jb{\Phi_\alpha{(\cj n)}}\}}\frac{\jb{n_4}^{\sigma}}{\jb{n_3}^\sigma}\frac{g_{n_1}}{\jb{n_1}^s}\frac{\overline{g_{n_2}}}{\jb{n_2}^s} \frac{f(n_3, \wt\tau_1)}{\jb{\wt\tau_1}^b}g(n_4,\wt\tau_1+\Phi_\alpha{(\cj n)})  d\wt\tau_1\\
&\les \sum_{\jb{n_j} \sim N_j, \, j=3,4} G^{N_1,N_2}(n_3,n_4) \|f(n_3,\cdot)\|_{L^q} \|g(n_4,\cdot)\|_{L^{q'+}}\\
&\les \|G^{N_1,N_2} \ind_{\jb{n_3}\sim N_3} \ind_{\jb{n_4}\sim N_4}\|_{\text{HS}(\l^2,\l^2)} \|f\|_{\l^2_n L^q_{\wt \tau}} \|g\|_{l^{2}_nL^{q'+}_{\wt \tau}}, \end{align*}
where we defined another kernel 
 \begin{align*}
G^{N_1,N_2}(n_3,n_4) := \Big\| \jb{\wt\tau_1}^{-(b-b')}  K^{N_1, N_2}_{\wt\tau_1}(n_3,n_4)\Big\|_{L^r_{\wt \tau_1}},
\end{align*}
for $r$ such that $r  > \frac{1}{b-b'}$ (depending on the implicit $\eps$ in $q'+ = q'+ \eps$).
Using that $\jb{\wt\tau_1}\gtrsim \jb{\Phi_\alpha{(\cj n)}}$, for $\jb{n_3} \sim N_3$, $\jb{n_4} \sim N_4$, we have that 
\begin{align}
\begin{split}
\E \Big[|K_{\wt\tau_1}^{N_1,N_2}(n_3,n_4)|^2 \Big] &\les \frac{N_4^{2\sigma}}{N_1^{2s}N_2^{2s}N_3^{2\sigma}} \bigg( \ind_{\{ N_1\leq N_2\}}\sum_{\substack{\jb{n_1} \sim N_1,\\ n_1 \neq n_2, n_4}}  \tfrac{1}{\jb{\Phi_\alpha{(\cj n)}}^{2b}} \\
& \hphantom{XXXXXXXXXX} +\ind_{\{ N_2\leq N_1\}}\sum_{\substack{\jb{n_2} \sim N_2,\\ n_2 \neq n_3, n_1}}  \tfrac{1}{\jb{\Phi_\alpha{(\cj n)}}^{2b}}  \bigg)
\end{split} \label{HKest1}
\end{align}
Fix $p > r$. Then, Minkowski's inequality, that $K^{N_1,N_2}_{\wt\tau_1}$ is a polynomial (of degree 2) in the Gaussian random variable $u_0$, and \eqref{HKest1} imply 
\begin{align}
 \E \bigg[   \|G^{N_1,N_2} &\ind_{\jb{n_3}\sim N_3} \ind_{\jb{n_4}\sim N_4}\|_{\text{HS}(\l^2,\l^2)}^{p}  \bigg]^{\frac{2}{p}} \notag \\
 & \leq  \sum_{\jb{n_j} \sim N_j, \, j=3,4} \bigg| \int \frac{1}{\jb{ \wt \tau_1}^{r(b-b')}}   \E[ |K_{\wt \tau_1}^{N_1,N_2}(n_3,n_4)|^{p}]^{\tfrac{r}{p}} d\wt \tau_1 \bigg|^{\frac{2}{r}} \notag \\ 
 & \les   \sum_{\jb{n_j} \sim N_j, \, j=3,4} \bigg| \int \frac{1}{\jb{ \wt \tau_1}^{r(b-b')}}   \E[ |K_{\wt \tau_1}^{N_1,N_2}(n_3,n_4)|^{2}]^{\tfrac{r}{2}} d\wt \tau_1 \bigg|^{\frac{2}{r}}  \notag \\
&\les \frac{N_4^{2\sigma}}{N_1^{2s}N_2^{2s}N_3^{2\sigma}} \bigg[ \ind_{\{ N_1\leq N_2\}}  \sum_{\substack{n_1 \sim N_1,n_3 \sim N_3, n_4 \sim N_4,\\ n_1 \neq n_2, n_4}} \frac{1}{\jb{\Phi_\alpha{(\cj n)}}^{2b'}}  \notag\\
& \hphantom{XXXXXXXXX} + \ind_{\{ N_2\leq N_1\}}  \sum_{\substack{n_2 \sim N_2,n_3 \sim N_3, n_4 \sim N_4,\\ n_1 \neq n_2, n_4}} \frac{1}{\jb{\Phi_\alpha{(\cj n)}}^{2b'}} \bigg] \notag\\
&\les \frac{N_4^{2\sigma}}{N_1^{2s}N_2^{2s}N_3^{2\sigma}N_{\max}^{4b'(\alpha-1)}} \bigg[   \ind_{\{ N_1\leq N_2\}}  \sum_{\substack{n_1 \sim N_1,n_3 \sim N_3, n_4 \sim N_4,\\ n_1 \neq n_2, n_4}}\tfrac{1}{\jb{n_1-n_2}^{2b'}}\tfrac{1}{\jb{n_1-n_4}^{2b'}}\notag   \\
& \hphantom{XXXXXXXX} +  \ind_{\{ N_2\leq N_1\}}  \sum_{\substack{n_1 \sim N_1,n_3 \sim N_3, n_4 \sim N_4,\\ n_1 \neq n_2, n_4}}\tfrac{1}{\jb{n_3-n_4}^{2b'}}\tfrac{1}{\jb{n_2-n_3}^{2b'}} \bigg] \notag\\
&\les N_4^{2\sigma}N_{\max}^{2-4b'} N_{\min}(N_1 N_2)^{-2s}N_3^{-2\s}N_{\max}^{-4b'(\al-1)},\label{GN1Nest}
\end{align}
Summing this over $N_1,N_2, N_3, N_4$, with the restriction that the largest two frequencies are similar, since $n_1-n_2+n_3=n_4$, and recalling that $\s \leq s$, we obtain \eqref{Bilin1} 
as long as
$$2b'(2\alpha - 2) > 2-4b', \quad \text{and} \quad 3-4b' < 4s + 2b'(2\alpha-2).$$
Therefore, since $s>\tfrac{1}{4}$, we need 
$b'> \max( \tfrac{1}{2\al}, \tfrac{3-4s}{4\al}) = \tfrac{1}{2\al},$
 which enforces \eqref{bilinearlwpcondition}.

We now move to the estimate for $J_{13}$. We proceed similarly, and define 

\begin{align*}
K_{\wt\tau_1}^{N_1,N_3}(n_2,n_4) &= \sum_{\substack{n_1-n_2+n_3 - n_4 = 0, \\ n_1 \neq n_2,n_4, \\ \jb{n_1} \sim N_1, \jb{n_3} \sim N_3}} \ind_{\{|\wt\tau_1|\gtrsim \jb{\Phi_\alpha{(\cj n)}}\}} \frac{\jb{n_4}^{\sigma}}{\jb{n_2}^\sigma}\frac{g_{n_1}{g_{n_3}}}{\jb{n_1}^s\jb{n_3}^s\jb{\wt\tau_1}^{b'}}, \\
G^{N_1,N_3}(n_2,n_4) &= \Big\| \jb{\wt\tau_1}^{-(b-b')}  K^{N_1, N_3}_{\wt\tau_1}(n_2,n_4)\Big\|_{L^r_{\wt \tau_1}}.
 \end{align*}

Then,
\begin{align*}
 \Big( \E \Big[ &\|G^{N_1,N_2} \ind_{\jb{n_3}\sim N_3} \ind_{\jb{n_4}\sim N_4}\|_{\text{HS}(\l^2,\l^2)}^\frac p2  \Big] \Big)^{\frac{2}{p}}  \\
 &\les \frac{N_4^{2\sigma}}{N_1^{2s}N_2^{2\s}N_3^{2s}} \bigg[ \ind_{\{ N_1\leq N_3\}}  \sum_{\substack{n_1 \sim N_1,n_3 \sim N_3, n_4 \sim N_4,\\ n_1 \neq n_2, n_4}} \frac{1}{\jb{\Phi_\alpha{(\cj n)}}^{2b'}}  \notag\\
& \hphantom{XXXXXXXXX} + \ind_{\{ N_3\leq N_1\}}  \sum_{\substack{n_3 \sim N_3,n_2 \sim N_2, n_4 \sim N_4,\\ n_1 \neq n_2, n_4}} \frac{1}{\jb{\Phi_\alpha{(\cj n)}}^{2b'}} \bigg] \notag\\
& \les  N_4^{2\sigma}N_{\max}^{2-4b'}N_{\min} (N_1 N_3)^{-2s}N_{2}^{-2\s} N_{\max}^{-4b'(\al-1)},
\end{align*}
so that under the condition \eqref{bilinearlwpcondition}, we can sum over $N_1,N_2, N_3, N_4$,  and obtain \eqref{Bilin1}, concluding the estimate for $J_{13}$ in the same way as we did for the estimate of $J_{12}$.

We now move to showing \eqref{Jconvergence}. We show the details only for $J_{12,N}$ as the estimates for $J_{13,N}$ are completely analogous. We have the following analogue of \eqref{J12formula} for $J_{12,N}$
\begin{equation*}
\begin{aligned}
 &\int_\R J_{12,N} (t)[v_1(t),v_2(t)]  \\
 &= \int_{\tau_1-\tau_2= |n_2|^{2\alpha}-|n_1|^{2\alpha}} \sum_{\substack{n_1-n_2+n_3 - n_4 = 0, \\ n_2 \neq n_1,n_3, \\ |n_j| \le N}} \frac{g_{n_1}}{\jb{n_1}^s}\frac{\overline{g_{n_2}}}{\jb{n_2}^s} \ft{v_1}(n_3, \tau_1) \overline{\ft{v_2}}(n_4,\tau_2) d\tau_1d\tau_2.
\end{aligned}
\end{equation*}
We define 
\begin{align*}
K_{\wt\tau_1,N}^{N_1,N_2}(n_3,n_4) &= \sum_{\substack{n_1-n_2+n_3 - n_4 = 0, \\ n_1 \neq n_2,n_4, \\ \jb{n_1} \sim N_1, \jb{n_2} \sim N_2} |n_j| \le N,}   \ind_{\{|\wt\tau_1|\gtrsim \jb{\Phi_\alpha{(\cj n)}}\}}\frac{\jb{n_4}^\sigma}{\jb{n_3}^\sigma}\frac{g_{n_1}\overline{g_{n_2}}}{\jb{n_1}^s\jb{n_2}^s\jb{\wt\tau_1}^{b'}}, \\
G^{N_1,N_2}_{N}(n_3,n_4) &= \Big\| \jb{\wt\tau_1}^{-(b-b')}  K^{N_1, N_2}_{\wt\tau_1}(n_3,n_4)\Big\|_{L^r_{\wt \tau_1}}.
\end{align*}
Proceeding exactly in the same way as before, we obtain that 
\begin{equation*}
\| J_{12} - J_{12,N} \|_{b,q,\sigma,-\sigma} \les \sum_{ \substack{N_1,N_2,N_3,N_4 \\ N_{\text{max}}\sim N_{\text{med}}}}   \big\| \big(G^{N_1,N_2} -G^{N_1,N_2}_{N} \big) \ind_{\jb{n_3}\sim N_3} \ind_{\jb{n_4} \sim N_4} \big\|_{\text{HS}(\l^2,\l^2)},
\end{equation*}
and proceeding as in the estimates for $J_{12}$, we obtain the analogous of \eqref{GN1Nest}, 
$$\E \Big[ \big\| \big(G^{N_1,N_2} -G^{N_1,N_2}_{N} \big) \ind_{\jb{n_3}\sim N_3} \ind_{\jb{n_4} \sim N_4} \big\|_{\text{HS}(\l^2,\l^2)}^{p} \Big]^{\frac{1}{p}} \les\frac{N_4^{\sigma}N_{\max}^{-2b'}N_{\min}^{\frac{1}{2}}}{N_1^{s}N_2^{\sigma}N_3^{s}N_{\max}^{b'(2\alpha-2)}}  \chi_{\{N_{\max} \gtrsim N\}}.$$
Therefore, in the same way as before, we obtain that 
$$ \E \big[\| J_{12} - J_{12,N} \|_{b,q,\sigma,-\sigma}^p \big] \les N^{\frac{p}{2}(3-4b - 4s - 2b(2\alpha-2)+)}.   $$
\end{proof}

The final estimates below controls the remaining contributions in a deterministic way.

\begin{proposition}[Deterministic quadrilinear estimate] \label{PROP:detquad}
Let $0 \le \sigma_1, \sigma_2, \sigma_3 < \frac 12$, and $0~<~b~<~\frac 12$ be such that 
\begin{equation} \label{quadlwpcond}
\quad b > \frac{1+\sigma_1+\sigma_2-\sigma_3}{2\alpha} 
\end{equation}
Then, for every $1 < q < \infty$, we have that 
\begin{equation}\label{quad_det_1}
\begin{aligned}
\int_{\R} \NN(u_1(t),u_2(t),u_3(t),u_4(t)) dt \les&\ \|u_1\|_{X^{-\sigma_1,b}_q}\|u_2\|_{X^{-\sigma_2,\frac1{q'}+}_q}\|u_3\|_{X^{\sigma_3,\frac 1 {q'}+}_q}\|u_4\|_{X^{0,0}_{q'}} \\
&+  \|u_1\|_{X^{-\sigma_1,\frac1{q'}+}_q}\|u_2\|_{X^{-\sigma_2,b}_q}\|u_3\|_{X^{\sigma_3,\frac 1 {q'}+}_q}\|u_4\|_{X^{0,0}_{q'}} \\
&+  \|u_1\|_{X^{-\sigma_1,\frac1{q'}+}_q}\|u_2\|_{X^{-\sigma_2,\frac1{q'}+}_q}\|u_3\|_{X_{q}^{\sigma_3,b}} \|u_4\|_{X^{0,0}_{q'}} \\
&+  \|u_1\|_{X^{-\sigma_1,0}_q}\|u_2\|_{X^{-\sigma_2,\frac1{q'}+}_q}\|u_3\|_{X^{\sigma_3,\frac 1 {q'}+}_q}\|u_4\|_{X^{0,b}_{q'}}. \\
\end{aligned}
\end{equation}
Similarly, we have that 
\begin{equation}\label{quad_det_2}
\begin{aligned}
\int_{\R} \NN(u_1(t),u_2(t),u_3(t),u_4(t)) dt \les&\ \|u_1\|_{X^{-\sigma_1,b}_q}\|u_2\|_{X^{\sigma_3,\frac1{q'}+}_q}\|u_3\|_{X^{-\sigma_2,\frac 1 {q'}+}_q}\|u_4\|_{X^{0,0}_{q'}} \\
&+  \|u_1\|_{X^{-\sigma_1,\frac1{q'}+}_q}\|u_2\|_{X^{\sigma_3,b}_q}\|u_3\|_{X^{-\sigma_2,\frac 1 {q'}+}_q}\|u_4\|_{X^{0,0}_{q'}} \\
&+  \|u_1\|_{X^{-\sigma_1,\frac1{q'}+}_q}\|u_2\|_{X^{\sigma_3,\frac1{q'}+}_q}\|u_3\|_{X^{-\sigma_2,b}_{q}}\|u_4\|_{X^{0,0}_{q'}} \\
&+  \|u_1\|_{X^{-\sigma_1,0}_q}\|u_2\|_{X^{\sigma_3,\frac1{q'}+}_q}\|u_3\|_{X^{-\sigma_2,\frac 1 {q'}+}_q}\|u_4\|_{X^{0,b}_{q'}}. \\
\end{aligned}
\end{equation}
\end{proposition}
\begin{proof}
We show only \eqref{quad_det_1}, as the proof of \eqref{quad_det_2} is completely analogous. 
Let 
$$ v_1 = \jb{\dx}^{-\sigma_1} u_1, \,\,\, v_2 = \jb{\dx}^{-\sigma_1} u_2, \,\,\, v_3 =\jb{\dx}^{\sigma_3} u_3, \,\,\, v_4=u_4,$$
so that 
$$  \|u_1\|_{X^{-\sigma_1,b_1}_q}\|u_2\|_{X^{-\sigma_2,b_2}_q}\|u_3\|_{X^{\sigma_3,b_3}_q}\|u_4\|_{X^{0,b_4}_{q'}} = \prod_{j=1}^3 \| v_j\|_{X^{0,b_j}_q} \|v_4\|_{X^{0,b_4}_{q'}}, $$
and
\begin{equation} \label{Nexpansion}
\begin{aligned}
&\int_\R \NN(u_1(t),u_2(t),u_3(t),u_4(t)) dt \\
=&\intt_{\tau_1-\tau_2+\tau_3-\tau_4 = 0} \sum_{\substack{n_1-n_2+n_3 - n_4 = 0, \\ n_2 \neq n_1,n_3}} \tfrac{\jb{n_1}^{\sigma_1}\jb{n_2}^{\sigma_2}}{\jb{n_3}^{\sigma_3}} \prod_{\substack{j=1 \\ j\neq 2}}^{3} \ft{v_j}(\tau_j,n_j) \overline{\ft{v_{j+1}}(\tau_{j+1},n_{j+1})} d\tau_1 \cdots d\tau_4.
\end{aligned}
\end{equation}
For convenience of notation, let $\wt{\tau_j} := \tau_j+|n_j|^{2\alpha}$. The condition $\tau_1-\tau_2+\tau_3-\tau_4 = 0$ then becomes 
$\wt{\tau_1}-\wt{\tau_2}+\wt{\tau_3}-\wt{\tau_4} = \Phi_\alpha{(\cj n)}.$ 
Therefore, for any $f_1, f_2, f_3 \in L^q(\R)$, $f_4 \in L^{q'}(\R)$, by H\"older we have that
\begin{align*}
&\int_{\substack{\wt{\tau}_{1}-\wt{\tau}_{2}+\wt{\tau}_{3}-\wt{\tau}_{4} = \Phi_\alpha{(\cj n)}}} \frac{\prod_{j=1}^{4} |f_j(\wt\tau_j)|}{\jb{\wt{\tau}_{1}}^{\frac1{q'}+}\jb{\wt{\tau}_{2}}^{\frac1{q'}+}\jb{\wt{\tau}_{\max}}^b} d\tau_1d\tau_2d\tau_3d\tau_4 \\
&\les \frac{1}{\jb{\Phi_\alpha{(\cj n)}}^b} \int_{\wt{\tau}_1, \wt{\tau}_2, \wt{\tau}_3} \frac{ |f_1(\wt{\tau_1})||f_2(\wt{\tau_2})||f_3(\wt{\tau_3})||f_4(\Phi_\alpha{(\cj n)} -\wt{\tau_1} - \wt{\tau_2} - \wt{\tau_3})|}{\jb{\wt{\tau}_{1}}^{\frac1{q'}+}\jb{\wt{\tau}_{2}}^{\frac1{q'}+}} d \wt{\tau_1}d \wt{\tau_2}d \wt{\tau_3} \\
&\les \frac{1}{\jb{\Phi_\alpha{(\cj n)}}^b} \prod_{j=1}^3 \|f_j\|_{L^q(\R)} \|f_4\|_{L^{q'}(\R)}
\end{align*}
By applying this inequality to \eqref{Nexpansion} over all possible configurations of $\wt{\tau_j}=\wt{\tau}_{\max}$, and permuting the roles of $f_1,f_2,f_3$ (so that $f_1$ and $f_2$ never correspond to $\wt\tau_{\max}$), we obtain 
\begin{align*}
&\int_{\R} \NN(u_1(t),u_2(t),u_3(t),u_4(t)) dt \\
\les& \sum_{\substack{n_1-n_2+n_3 - n_4 = 0, \\ n_2 \neq n_1,n_3}}  \frac{\jb{n_1}^{\sigma_1}\jb{n_2}^{\sigma_2}}{\jb{n_3}^{\sigma_3}\jb{\Phi_\alpha{(\cj n)}}^b} \\
&\phantom{ \sum} \times \big(\| \ft{v_1}(n_1,\tau_1) \jb{\tau_1}^b \|_{L^q_{\tau_1}} \| \ft{v_2}(n_2,\tau_2) \jb{\tau_2}^{\frac 1 {q'}+} \|_{L^q_{\tau_2}} \| \ft{v_3}(n_3,\tau_3) \jb{\tau_3}^{\frac 1 {q'}+} \|_{L^q_{\tau_3}}  \| \ft{v_4}(n_4,\tau_4) \|_{L^{q'}_{\tau_4}} \\
&\phantom{ \sum \times )} + \| \ft{v_1}(n_1,\tau_1) \jb{\tau_1}^{\frac1 {q'} +} \|_{L^q_{\tau_1}} \| \ft{v_2}(n_2,\tau_2) \jb{\tau_2}^{b} \|_{L^q_{\tau_2}} \| \ft{v_3}(n_3,\tau_3) \jb{\tau_3}^{\frac 1 {q'}+} \|_{L^q_{\tau_3}}  \| \ft{v_4}(n_4,\tau_4) \|_{L^{q'}_{\tau_4}} \\
&\phantom{ \sum \times ) } + \| \ft{v_1}(n_1,\tau_1) \jb{\tau_1}^{\frac1 {q'} +} \|_{L^q_{\tau_1}} \| \ft{v_2}(n_2,\tau_2) \jb{\tau_2}^{\frac1 {q'} +} \|_{L^q_{\tau_2}} \| \ft{v_3}(n_3,\tau_3) \jb{\tau_3}^{b} \|_{L^q_{\tau_3}}  \| \ft{v_4}(n_4,\tau_4) \|_{L^{q'}_{\tau_4}} \\
&\phantom{ \sum \times )} + \| \ft{v_1}(n_1,\tau_1) \|_{L^q_{\tau_1}} \| \ft{v_2}(n_2,\tau_2) \jb{\tau_2}^{\frac1 {q'} +} \|_{L^q_{\tau_2}} \| \ft{v_3}(n_3,\tau_3) \jb{\tau_3}^{\frac1 {q'} +} \|_{L^q_{\tau_3}} \| \ft{v_4}(n_4,\tau_4) \jb{\tau_4}^b \|_{L^{q'}_{\tau_4}}\big).
\end{align*}
Therefore, \eqref{quad_det_1} follows if we show that 
\begin{equation*}
\sum_{\substack{n_1-n_2+n_3 - n_4 = 0, \\ n_2 \neq n_1,n_3}} \frac{\jb{n_1}^{\sigma_1}\jb{n_2}^{\sigma_2}}{\jb{n_3}^{\sigma_3}\jb{\Phi_\alpha{(\cj n)}}^b} \prod_{j=1}^4 w_j(n_j) \les \prod_{j=1}^4 \| w_j \|_{\l^2_n}.
\end{equation*}
By \eqref{p1} and H\"older, we have that 
\begin{align*}
&\sum_{\substack{n_1-n_2+n_3 - n_4 = 0, \\ n_2 \neq n_1,n_3}} \frac{\jb{n_1}^{\sigma_1}\jb{n_2}^{\sigma_2}}{ \jb{n_3}^{\s_3}  \jb{\Phi_\alpha{(\cj n)}}^b} \prod_{j=1}^4 w_j(n_j)\\
&\les \sum_{n_1-n_2+n_3-n_4=0}  \frac{\jb{n_1}^{\sigma_1}\jb{n_2}^{\sigma_2}}{\jb{n_3-n_2}^{b}\jb{n_3-n_4}^b\jb{n_{\max}}^{2b(\alpha - 1)}} \prod_{j=1}^{4} | w_j (n_j)|\\
&\les \sum_{\substack{N_1,N_2,N_3,N_4  \\ N^{1}\sim N^2}} 
\frac{ N_{1}^{\s_1}N_{2}^{\s_2}}{ N_{3}^{\s_3} (N^1)^{2b(\al-1)}} \sum_{\substack{n_1-n_2+n_3-n_4 = 0, \\ n_1 \neq n_2,n_4,\\ \jb{n_j}\sim N_j}}
\frac{1}{\jb{n_3-n_2}^{b}\jb{n_3-n_4}^b}\prod_{j=1}^{4} | w_j (n_j)|. 
\end{align*}
We distinguish two cases.

\medskip
\noi
\textbf{Case 1:} $N_2 \vee N_3 \sim N_{\max}$. \\
We have that 
\begin{align*}
&\sum_{\substack{N_1,N_2,N_3,N_4  \\ N^{1}\sim N_2 \vee N_3 }} 
\frac{ N_{1}^{\s_1}N_{2}^{\s_2}}{ N_{3}^{\s_3} (N^1)^{2b(\al-1)}} \sum_{\substack{n_1-n_2+n_3-n_4 = 0, \\ n_1 \neq n_2,n_4,\\ \jb{n_j}\sim N_j}}
\frac{1}{\jb{n_3-n_2}^{b}\jb{n_3-n_4}^b}\prod_{j=1}^{4} | w_j (n_j)|. \\
&\les \sum_{\substack{N_1,N_2,N_3,N_4  \\ N^{1}\sim N_2 \vee N_3}} 
\frac{ N_{1}^{\s_1}N_{2}^{\s_2}}{ N_{3}^{\s_3} (N^1)^{2b(\al-1)}} \Big(\sum_{ \substack{\jb{n_j} \sim N_j \\ j=2,3,4}} \frac{|w_2(n_2)|^2 |w_4(n_4)|^2}{ \jb{n_3-n_2}^{2b}}\Big)^{\frac12} \Big(\sum_{ \substack{ \jb{n_j}\sim N_j \\ j=1,3,4}} \frac{|w_1(n_1)|^2 |w_3(n_3)|^2}{ \jb{n_3-n_4}^{2b}}\Big)^{\frac12} \\
&\les \sum_{\substack{N_1,N_2,N_3,N_4  \\ N^{1}\sim N_2 \vee N_3 }} N_1^{\sigma_1}N_2^{\sigma_2}N_3^{-\sigma_3}N_{\max}^{-(2\alpha-2)b} \big(N_3 (N_3 \vee N_2)^{-2b}\big)^{\frac12} \big(N_4 (N_3 \vee N_4)^{-2b}\big)^{\frac12} \prod_{j=1}^{4} \| w_j \|_{\l^2}  \\
&\les \sum_{\substack{N_1,N_2,N_3,N_4  \\ N^{1}\sim N_2 \vee N_3}}  N_1^{\sigma_1}N_2^{\sigma_2}N_3^{\frac12-\sigma_3}N_{\max}^{-2\alpha b+b} N_4^{\frac 12-b} \prod_{j=1}^{4} \| w_j \|_{\l^2} . 
\end{align*}
Since $\sigma_3, b \le \frac 12$, it easy to check that the last sum is summable if
$$ 2\alpha b  > 1 + \sigma_1 + \sigma_2 - \sigma_3,$$
i.e.\ exactly under the condition \eqref{quadlwpcond}.


\medskip
\noi
\textbf{Case 2:} $N_1 \sim N_{4}\gg N_2,N_3$. \\
We have that 
\begin{align*}
&\sum_{\substack{N_1,N_2,N_3,N_4  \\ N^{1}\sim N_1 \wedge N_4}} 
\frac{ N_{1}^{\s_1}N_{2}^{\s_2}}{ N_{3}^{\s_3} (N^1)^{2b(\al-1)}} \sum_{\substack{n_1-n_2+n_3-n_4 = 0, \\ n_1 \neq n_2,n_4,\\ \jb{n_j}\sim N_j}}
\frac{1}{\jb{n_3-n_2}^{b}\jb{n_3-n_4}^b}\prod_{j=1}^{4} | w_j (n_j)|. \\
&= \sum_{\substack{N_1,N_2,N_3,N_4  \\ N^{1}\sim N_1 \wedge N_4}} 
\frac{ N_{1}^{\s_1}N_{2}^{\s_2}}{ N_{3}^{\s_3} (N^1)^{2b(\al-1)}} \sum_{\substack{n_1-n_2+n_3-n_4 = 0, \\ n_1 \neq n_2,n_4,\\ \jb{n_j}\sim N_j}}
\frac{1}{\jb{n_3-n_4}^{b}\jb{n_1-n_4}^b}\prod_{j=1}^{4} | w_j (n_j)|. \\
&\les \sum_{\substack{N_1,N_2,N_3,N_4  \\ N^{1}\sim N_1 \wedge N_4}} 
\frac{ N_{1}^{\s_1}N_{2}^{\s_2}}{ N_{3}^{\s_3} (N^1)^{2b(\al-1)}} \Big(\sum_{ \substack{\jb{n_j} \sim N_j \\ j=2,3,4}} \frac{|w_2(n_2)|^2 |w_4(n_4)|^2}{ \jb{n_3-n_4}^{2b}}\Big)^{\frac12} \Big(\sum_{ \substack{ \jb{n_j}\sim N_j \\ j=1,3,4}} \frac{|w_1(n_1)|^2 |w_3(n_3)|^2}{ \jb{n_1-n_4}^{2b}}\Big)^{\frac12} \\
&\les \sum_{\substack{N_1,N_2,N_3,N_4  \\ N^{1}\sim N_1 \wedge N_4}} N_1^{\sigma_1}N_2^{\sigma_2}N_3^{-\sigma_3}N_{\max}^{-(2\alpha-2)b} \big(N_3 (N_3 \vee N_4)^{-2b}\big)^{\frac12} \big(N_4 (N_1 \vee N_4)^{-2b}\big)^{\frac12} \prod_{j=1}^{4} \| w_j \|_{\l^2}  \\
&\les \sum_{\substack{N_1,N_2,N_3,N_4  \\ N^{1}\sim N_1 \wedge N_4}}  N_1^{\sigma_1}N_2^{\sigma_2}N_3^{\frac12-\sigma_3}N_{\max}^{-2\alpha b + \frac 12} \prod_{j=1}^{4} \| w_j \|_{\l^2} . 
\end{align*}
Similarly to the previous case, since $\sigma_3 \le \frac 12$, the last sum is summable if
$$ 2\alpha b  > 1 + \sigma_1 + \sigma_2 - \sigma_3,$$
i.e.\ exactly under the condition \eqref{quadlwpcond}.
\end{proof}

\subsection{Proof of Proposition \ref{PROP:LWP}}

Fix $\al$, $s$ and $q$ as in the statement of Proposition~\ref{PROP:LWP}.
We start by showing local well-posedness in $L^2(\T)$, together with the estimate \eqref{tau0bound}. To this purpose, it is enough to show that the functional
$$\Gamma(v)(t) := \int_0^t \NN(S(t)(u_0+v_0)+v) + \RR(S(t)(u_0+v_0)+v) $$
is a contraction (for instance, with Lipschitz constant $1/2$) on the ball
$$B:= \Big\{ \|v\|_{X^{0,\frac1{q'}+}_q(T)}\les 1 + \|v_0\|_{L^2}\Big\},$$
for every $T$ such that 
\begin{equation}
 0\le T \ll \Big(1 + \|v_0\|_{L^2}^{\frac{4\al}{2\al-1}+\eps} +\|\Xi(u_0)\|_{\mathcal X^{0,q}}^{C_{\eps}}\Big)^{-1}. \label{Tconditionlwp}
\end{equation}
By \eqref{duhamelest} and duality, it is enough to check that for $v \in B$, for $T$ that satisfies \eqref{Tconditionlwp}, and for $z(t):=S(t)(u_0+v_0)$,
\begin{align}
\int_0^t \NN(z(t')+v,z(t')+v,z(t')+v,z(t')+v,& \chi_T(t') w(t')) d t'  \notag \\
 &\ll (1 + \|v_0\|_{L^2})\|w\|_{X^{0,\frac1q-}_{q'}} \label{imageconditionNNlwp}
\end{align}
and
\begin{align}
\int_0^t \RR(z(t')+v,z(t')+v,z(t')+v,z(t')+v, &\chi_T(t') w(t')) d t' \notag\\
 &\ll (1 + \|v_0\|_{L^2})\|w\|_{X^{0,\frac1q-}_{q'}},\label{imageconditionRRlwp}
\end{align}

and that for $v_1,v_2 \in B$,

\begin{align}
 &\int_0^t \NN(z(t')+v_1,z(t')+v_1,z(t')+v_1,z(t')+v_1, \chi_T(t') w(t')) d t'  \notag \\
 -&\int_0^t \NN(z(t')+v_2,z(t')+v_2,z(t')+v_2,z(t')+v_2, \chi_T(t') w(t')) d t'  \notag \\
 &\hphantom{\int_0^t \NN(z(t')+v_2,z(t')+v_2,z(t')+v_2,z(t')+v_2, }\ll \|v_1-v_2\|_{X_{q}^{0,\frac{1}{q'}+}(T)}\|w\|_{X^{0,\frac1q-}_{q'}}.\label{contractiononditionNN}
\end{align}
and 
\begin{align}
 &\int_0^t \RR(z(t')+v_1,z(t')+v_1,z(t')+v_1,z(t')+v_1, \chi_T(t') w(t')) d t' \notag \\
 -&\int_0^t \RR(z(t')+v_2,z(t')+v_2, z(t')+v_2,z(t')+v_2, \chi_T(t') w(t')) d t' \notag\\
 &\hphantom{\int_0^t \RR(z(t')+v_1,z(t')+v_1,z(t')+v_1,z(t')+v_1,}\ll \|v_1-v_2\|_{X_{q}^{0,\frac{1}{q'}+}(T)}\|w\|_{X^{0,\frac1q-}_{q'}}. \label{contractiononditionRR}
\end{align}

In the following estimates, we use that on the ball $B$, we have that 
$$ \|S(t) v_0 + v \|_{X^{0,\frac1{q'}+}_q(T)} \les 1 + \|v_0\|_{L^2}. $$
From the definitions \eqref{X3def}, \eqref{J12def} and \eqref{J13def} of $X^{(3)}$,\footnote{Notice that the condition \eqref{qbdX3} is automatically satisfied under the hypothesis \eqref{LWPqbd}.}  $J_{12}$ and $J_{13}$ respectively, the symmetry $\NN(u_1,u_2,u_3,u_4) = \NN(u_3,u_2,u_1,u_4)$, \eqref{quad_det_1}, \eqref{quad_det_2}, \eqref{timeloc}, \eqref{q'+} and H\"{o}lder's inequality,  we obtain that 
\begin{align*}
&  \int_0^t \NN(S(t)(u_0+v_0)+v,S(t)(u_0+v_0)+v,S(t)(u_0+v_0)+v,S(t)(u_0+v_0)+v, \chi_T(t) w(t)) d t \\
\les&\  \|X^{(3)}\|_{X^{0, \frac1{q'}+}_q}\|w\|_{X^{0,\frac1q-}_{q'}} \\
&+ \big(\|J_{12}\|_{b_J+,q,0,0} + \|J_{13}\|_{b_J+,q,0,0} \big)\\
&\phantom{+\ } \times \big(\|S(t) v_0 + v\|_{X^{0,0}_q(T)}\|\chi_T w  \|_{X^{0,b_J+}_{q'+}(T)} + \|\chi_T(t)(S(t) v_0 + v)\|_{X^{0,b_J+}_q}\|\chi_T w \|_{X^{0,0}_{q'+}}\big)\\
&+ \Big(\|S(t)u_0\|_{X^{s-\frac12-, \frac{3-2s-}{4\alpha}+}_q(T)} \|S(t) v_0 + v\|_{X^{0,\frac1{q'}+}_q(T)}^2 \|\chi_T w  \|_{X^{0,0}_{q'}} \\
&\phantom{\Big()}+ \|S(t)u_0\|_{X^{s-\frac12-, \frac1{q'}+}_q(T)} \|S(t) v_0 + v\|_{X^{0,\frac{3-2s-}{4\alpha}+}_q(T)}\|S(t) v_0 + v\|_{X^{0,\frac1{q'}+}_q} \|\chi_T w  \|_{X^{0,0}_{q'}}\\
&\phantom{\Big()}+ \|S(t)u_0\|_{X^{s-\frac12-,0}_q(T)} \|S(t) v_0 + v\|_{X^{0,\frac1{q'}+}_q(T)}^2 \|\chi_T w  \|_{X^{0,\frac{3-2s-}{4\alpha}+}_{q'}(T)}\Big)\\
&+\Big( \|S(t)v_0 + v\|_{X^{0, \frac{1}{2\alpha}+}_q(T)} \|S(t) v_0 + v\|_{X^{0,\frac1{q'}+}_q(T)}^2 \|\chi_T w  \|_{X^{0,0}_{q'}} \\
&\phantom{\Big()}+ \|S(t)v_0 + v\|_{X^{0, 0}_q(T)} \|S(t) v_0 + v\|_{X^{0,\frac1{q'}+}_q(T)}^2 \|\chi_T w  \|_{X^{0,\frac{1}{2\alpha}+}_{q'}}\Big)\\
\les&\ \|w\|_{X^{0,\frac1q-}_{q'}} \Big(T^{\eta-} \|X^{(3)}\|_{X^{0,\frac1{q'}+\eta+}_q(1)} + \big(\|J_{12}\|_{b_J+,q,0,0} + \|J_{13}\|_{b_J+,q,0,0} \big) T^{1 - b_J-} (1 + \|v_0\|_{L^2}) \\
&\phantom{\|w\|_{X^{0,\frac1{q'}+\eta+}_q(1)} \Big()}+ T^{1-\frac{3-2s-}{4\alpha}-} \|u_0\|_{H^{s-\frac12-}} (1 + \|v_0\|_{L^2})^2 + T^{1-\frac 1{2\alpha}}(1 + \|v_0\|_{L^2})^3\Big).
\end{align*}
Therefore, \eqref{imageconditionNNlwp} holds if we choose $T$ such that 
\begin{equation}
\begin{aligned}
T &\ll \|X^{(3)}\|_{X^{0,\frac1{q'}+\eta+}_q(1)}^{-\eta^{-1}-}, \\
T &\ll (\|J_{12}\|_{b_J+,q,0,0} + \|J_{13}\|_{b_J+,q,0,0})^{-\frac1{1-b_J}-},\\
T&\ll  \big[\|u_0\|_{H^{s-\frac12-}} (1 + \|v_0\|_{L^2})\big]^{-\frac1{1-\frac{3-2s}{4\alpha}}-},\\
T&\ll (1 + \|v_0\|_{L^2})^{-\frac{2}{1-\frac 1{2\alpha}}-}.
\end{aligned}\label{TconditionlistNN}
\end{equation}
By Young's inequality, these conditions are all satisfied under \eqref{Tconditionlwp}, after noticing that 
\begin{equation*}
\frac{2}{1-\frac 1{2\alpha}} > \frac1{1-\frac{3-2s}{4\alpha}} \iff \alpha + s > 1. 
\end{equation*}
Notice that applying Proposition~\ref{randomtrilwp} enforces the condition on $q$ in \eqref{LWPqbd}, which also forces $\al+s>\tfrac{3}{2}$. 

We proceed similarly for the difference estimate \eqref{contractiononditionRR}, using quadrilinearity of the functional $\NN$.
We obtain that
\begin{align*}
 &\int_0^t \NN(z(t')+v_1,z(t')+v_1,z(t')+v_1,z(t')+v_1, \chi_T(t') w(t')) d t' \\
 -&\int_0^t \NN(z(t')+v_2,z(t')+v_2,z(t')+v_2,z(t')+v_2, \chi_T(t') w(t')) d t' \\
\les&\ \|v_1-v_2\|_{X^{0,\frac1{q'}+}_q(T)}\|w\|_{X^{0,\frac1q-}_{q'}} \Big(\big(\|J_{12}\|_{b_J+,q,0,0} + \|J_{13}\|_{b_J+,q,0,0} \big) T^{1 - b_J-}  \\
&\phantom{\|v_1-v_2\|_{X^{0,\frac1{q'}+}_q(T)}\|w\|_{X^{0,\frac1q-}_{q'}} \Big()}+ T^{2-\frac{3-2s}{4\alpha}-} \|u_0\|_{H^{s-\frac12-}} (1 + \|v_0\|_{L^2}) + T^{1-\frac 1{2\alpha}}(1 + \|v_0\|_{L^2})^2\Big).
\end{align*}
Therefore, we obtain \eqref{contractiononditionNN} under the same list of conditions \eqref{TconditionlistNN}, which in particular holds whenever \eqref{Tconditionlwp} holds.  

We now move to estimate the resonant term $\RR$. By \eqref{res12}, \eqref{res13}, \eqref{resrandom}, the symmetry $\RR(u_1,u_2,u_3,u_4) = \RR(u_3,u_2,u_1,u_4)$, \eqref{Xsbpembedding}, \eqref{linest}, and \eqref{timeloc} we obtain that 
\begin{align*}
 \int_0^t & \RR(z(t')+v,z(t')+v,z(t')+v,z(t')+v, \chi_T(t') w(t')) d t' \\
\les &\ \|S(t)u_0\|_{X^{0,\frac1{q'}+,}_{6,q}(T)}^2\|S(t)u_0\|_{X^{0,0}_{6,q}(T)}\|\chi_T w\|_{X^{0,0}_{q'}(T)} \\
&+ \|S(t)u_0\|_{X^{0,\frac1{q'}+}_{\infty,q}(T)}^2\|S(t)v_0 + v\|_{X^{0,0}_q(T)}\|\chi_T w\|_{X^{0,0}_{q'}(T)} \\
&+ \|S(t)u_0\|_{X^{0,\frac1{q'}+}_{\infty, q}(T)}\|S(t)v_0 + v\|_{X^{0,\frac1{q'}+}_q(T)}\|S(t)v_0 + v\|_{X^{0,0}_q(T)}\|\chi_T w\|_{X^{0,0}_{q'}(T)} \\
&+ \|S(t)v_0 + v\|_{X^{0,\frac1{q'}+}_q(T)}^2\|S(t)v_0 + v\|_{X^{0,0}_q(T)}\|\chi_T w\|_{X^{0,0}_{q'}(T)} \\
&\les T^{1-} \| w\|_{X^{0,\frac1q-}_{q'}}\big( \|u_0\|_{\FL^{0,6}}^3 + (1+\|v_0\|_{L^2})^3\big).
\end{align*}
Therefore, we obtain \eqref{imageconditionRRlwp} under the conditions
\begin{align*}
T &\ll \|u_0\|_{\FL^{0,6}}^{-3-}, \\
T &\ll (1+\|v_0\|_{L^2})^{-2-},
\end{align*}
which clearly hold under \eqref{Tconditionlwp}. Proceeding similarly for the difference estimate, we obtain that 
\begin{align*}
 &\int_0^t \RR(z(t')+v_1,z(t')+v_1,z(t')+v_1,z(t')+v_1, \chi_T(t') w(t')) d t' \\
 -&\int_0^t \RR(z(t')+v_2,z(t')+v_2,z(t')+v_2,z(t')+v_2, \chi_T(t') w(t')) d t' \\
 & \hphantom{XXX} \les T^{1-}\|v_1-v_2\|_{X^{0,\frac1{q'}+}_q} \| w(t)\|_{X^{0,\frac1q-}_{q'}}\big(\|u_0\|_{\FL^{0,6}}^2 + (1+\|v_0\|_{L^2})^2\big).
\end{align*}
Therefore, we obtain \eqref{contractiononditionRR} under the conditions 
\begin{align*}
T &\ll \|u_0\|_{\FL^{0,6}}^{-2-}, \\
T &\ll (1+\|v_0\|_{L^2})^{-2-},
\end{align*}
which holds as well under \eqref{Tconditionlwp}.

We now move to the estimate \eqref{tausigmabound}. We will show that, if 
\begin{equation}\label{Tpreservationlwp}
T \ll \Big(1+\|v_0\|_{L^2}^{\frac{4\al}{2\al-1}+\eps} +\|\Xi(u_0)\|_{\mathcal X^{0,q}}^{C_{\eps}}+\|\Xi(u_0)\|_{\mathcal X^{\sigma,q}}^{C_{\eps,\sigma}}\Big)^{-1},
\end{equation}
and 
\begin{equation*}
v \in B_\sigma:= \big\{ \|v\|_{X^{\sigma,\frac1{q'}+}_q(T)} \les 1 + \|v_0\|_{H^{\sigma}} \big\} \cap \big\{ \|v\|_{X^{0,\frac1{q'}+}_q(T)} \les 1 + \|v_0\|_{L^2} \big\}
\end{equation*}
then
\begin{equation*}
\| \Gamma(v) \|_{X^{\sigma,\frac1{q'}+}_q(T)} \ll 1 + \|v_0\|_{H^\sigma}.
\end{equation*}
The fact that on such an interval the solution $v$ will satisfy 
$$\| v \|_{X^{\sigma,\frac1{q'}+}_q(T)} \les 1 + \|v_0\|_{H^\sigma} $$
follows from standard arguments.\footnote{For instance, one can apply Schauder's fixed point theorem to $\Gamma$ on the set $B_\sigma$, after noticing that $X^{\sigma,b}_q(T)$ embeds compactly in $X^{\sigma-,b-}_q(T)$, and use uniqueness of the solution in $L^2$ to conclude that the fixed point given by Schauder's theorem coincides with the solution $v$ built in the first part of the proof. Notice that the $L^2$ estimate needed to apply Schauder's fixed point theorem on $B_\sigma$ follows from \eqref{imageconditionNNlwp} and \eqref{imageconditionRRlwp}.}
By \eqref{duhamelest} and duality, it is enough to check that for $T$ that satisfies \eqref{Tpreservationlwp} and $v \in B_\sigma$,

\begin{align}
\int_0^t \NN(z(t')+v,z(t')+v,z(t')+v,z(t')+v,& \chi_T(t') w(t')) d t'  \notag\\
 &\ll (1 + \|v_0\|_{H^\sigma})\|w\|_{X^{-\sigma,\frac1{q}-}_{q'}} \label{NNpreservation}
\end{align}
and 
\begin{align}
 \int_0^t \RR(z(t')+v,z(t')+v,z(t')+v,z(t')+v, &\chi_T(t') w(t')) d t'  \notag\\
 &\ll (1 + \|v_0\|_{H^\sigma})\|w\|_{X^{-\sigma,\frac1{q}-}_{q'}}. \label{RRpreservation}
\end{align}
We start with estimating the terms depending on $\NN$, proceeding analogously to what we did for \eqref{imageconditionNNlwp}. Notice that if $v \in B_\sigma$, then
$$ \|S(t)v_0 + v \|_{X^{\sigma,\frac1{q'}+}_q(T)} \les 1 + \|v_0\|_{H^\sigma}, \quad \|S(t)v_0 + v \|_{X^{0,\frac1{q'}+}_q(T)} \les 1 + \|v_0\|_{L^2}. $$
Therefore, by the definitions \eqref{X3def}, \eqref{J12def} and \eqref{J13def} of $X^{(3)}$, $J_{12}$ and $J_{13}$ respectively, the symmetry 
$$\NN(u_1,u_2,u_3,u_4) = \NN(u_3,u_2,u_1,u_4)=\NN(u_1,u_4,u_3,u_2) = \NN(u_3,u_4,u_1,u_2),$$
\eqref{quad_det_1}, \eqref{quad_det_2}, \eqref{timeloc}, and \eqref{q'+}, we obtain that 
\begin{align*}
&\int_0^t \NN(z(t')+v,z(t')+v,z(t')+v,z(t')+v, \chi_T(t') w(t')) d t' \\
\les&\  \|X^{(3)}\|_{X^{\sigma, \frac1{q'}+}_q(T)}\|w\|_{X^{-\sigma,\frac1q-}_{q'}} \\
&+ \big(\|J_{12}\|_{b_J+,q,\sigma,-\sigma} + \|J_{13}\|_{b_J+,q,\sigma,-\sigma} \big)\\
&\phantom{+\ } \times \big(\|S(t) v_0 + v\|_{X^{\sigma,0}_q}\|w  \|_{X^{-\sigma,b_J+}_{q'+}} + \|S(t) v_0 + v\|_{X^{\sigma,b_J+}_q}\|\chi_T w  \|_{X^{-\sigma,0}_{q'+}}\big)\\
&+ \Big(\|S(t)u_0\|_{X^{s-\frac12-, \frac{3-2s-}{4\alpha}+}_q(T)} \|S(t) v_0 + v\|_{X^{0,\frac1{q'}+}_q(T)}\|S(t) v_0 + v\|_{X^{\sigma,\frac1{q'}+}_q(T)} \|\chi_T w \|_{X^{-\sigma,0}_{q'}} \\
&\phantom{\Big()}+\|S(t)u_0\|_{X^{s-\frac12-, \frac{1}{q'}+}_q(T)} \|S(t) v_0 + v\|_{X^{0,\frac{3-2s-}{4\alpha}+}_q(T)}\|S(t) v_0 + v\|_{X^{\sigma,\frac1{q'}+}_q(T)} \|\chi_T w  \|_{X^{-\sigma,0}_{q'}}\\
&\phantom{\Big()}+\|S(t)u_0\|_{X^{s-\frac12-, \frac{1}{q'}+}_q(T)} \|S(t) v_0 + v\|_{X^{0,\frac1{q'}+}_q(T)}\|S(t) v_0 + v\|_{X^{\sigma,\frac{3-2s-}{4\alpha}+}_q(T)} \|\chi_T w \|_{X^{-\sigma,0}_{q'}}\\
&\phantom{\Big()} +\|S(t)u_0\|_{X^{s-\frac12-, \frac{1}{q'}+}_q(T)} \|S(t) v_0 + v\|_{X^{0,\frac1{q'}+}_q(T)}\|S(t) v_0 + v\|_{X^{\sigma,0}_q(T)} \|\chi_T w  \|_{X^{-\sigma,\frac{3-2s-}{4\alpha}+}_{q'}}\Big)\\
&+\Big( \|S(t)v_0 + v\|_{X^{0, \frac{1}{2\alpha}}_q(T)} \|S(t) v_0 + v\|_{X^{0,\frac1{q'}+}_q(T)}  \|S(t) v_0 + v\|_{X^{\sigma,\frac1{q'}+}_q(T)} \|\chi_T w  \|_{X_{q'}^{-\sigma,0}} \\
&\phantom{\Big()} + \|S(t) v_0 + v\|_{X^{0,\frac1{q'}+}_q(T)}^2  \|S(t) v_0 + v\|_{X^{\sigma,\frac1{2\alpha}}_q(T)} \|\chi_T w  \|_{X_{q'}^{-\sigma,0}}\\
&\phantom{\Big()} + \|S(t) v_0 + v\|_{X^{0,\frac1{q'}+}_q(T)}^2  \|S(t) v_0 + v\|_{X^{\sigma,0}_q(T)} \|\chi_T w  \|_{X_{q'}^{-\sigma,\frac{1}{2\alpha}}}\Big)\\
\les& \|w\|_{X^{-\sigma,\frac1q-}_{q'}} \Big(T^{\eta-} \|X^{(3)}\|_{X^{\sigma, \frac1{q'}+\eta+}_q(1)} \\
&  \phantom{\|w\|_{X^{-\sigma,\frac12-}} \Big()} + \big(\|J_{12}\|_{b_J+,q,\sigma,-\sigma} + \|J_{13}\|_{b_J+,q,\sigma,-\sigma} \big) T^{1 - b_J-} (1 + \|v_0\|_{H^\sigma}) \\
&\phantom{\|w\|_{X^{-\sigma,\frac12-}} \Big()}+ T^{1-\frac{3-2s}{4\alpha}-} \|u_0\|_{H^{s-\frac12-}} (1 + \|v_0\|_{L^2})(1 + \|v_0\|_{H^\sigma})\\
&\phantom{\|w\|_{X^{-\sigma,\frac12-}} \Big()} + T^{1-\frac 1{2\alpha}}(1 + \|v_0\|_{L^2})^2(1 + \|v_0\|_{H^\sigma})\Big).
\end{align*}
Therefore, \eqref{NNpreservation} holds if $T$ satisfies
\begin{align*}
T &\ll \|X^{(3)}\|_{X^{\sigma, \frac1{q'}+\eta+}_q(1)}^{-\eta^{-1}-} \\
T &\ll (\|J_{12}\|_{b_J+,q,\sigma,-\sigma} + \|J_{13}\|_{b_J+,q,\sigma,-\sigma})^{-\frac1{1-b_J}-},\\
T&\ll  \big[\|u_0\|_{H^{s-\frac12-}} (1 + \|v_0\|_{L^2})\big]^{-\frac1{1-\frac{3-2s}{4\alpha}}-},\\
T&\ll (1 + \|v_0\|_{L^2})^{-\frac{2}{1-\frac 1{2\alpha}}-}.
\end{align*}
Compare to the list \eqref{TconditionlistNN}. The difference is that terms depending on $\|\Xi(u_0)\|_{\mathcal X^{0,q}}$ are replaced by terms depending on $\|\Xi(u_0)\|_{\mathcal X^{\sigma,q}}$. Therefore, one can check that these conditions hold under \eqref{Tpreservationlwp} in the same way as we proved the analogous result for \eqref{TconditionlistNN}. Recall that the extra term depending $\|\Xi(u_0)\|_{\mathcal X^{0,q}}$ in  \eqref{Tpreservationlwp} is there to guarantee that $\| v \|_{X^{0,\frac1{q'}+}_q(T)} \les 1 + \|v_0\|_{L^2}$. We now move to showing \eqref{RRpreservation}. By \eqref{res12}, \eqref{res13}, \eqref{resrandom}, the symmetry 
$$\RR(u_1,u_2,u_3,u_4) = \RR(u_3,u_2,u_1,u_4) = \RR(u_1,u_4,u_3,u_2) = \RR(u_3,u_4,u_1,u_2),$$ 
\eqref{Xsbpembedding}, \eqref{linest}, \eqref{timeloc}, and \eqref{timeloc2}, we obtain that 
\begin{align*}
 &\int_0^t \RR(z(t')+v,z(t')+v,z(t')+v,z(t')+v, \chi_T(t') w(t')) d t' \\
\les &\ \|S(t)u_0\|_{X^{\frac\sigma3,\frac1{q'}+}_{6,q}(T)}^2\|S(t)u_0\|_{X^{\frac\sigma3,0}_{6,q}(T)}\|\chi_T w\|_{X^{0,0}_{q'}(T)} \\
&+ \|S(t)u_0\|_{X^{0,\frac1{q'}+}_{\infty, q}(T)}^2\|S(t)v_0 + v\|_{X^{\sigma,0}_q(T)}\|\chi_T w\|_{X^{0,0}_{q'}(T)} \\
&+ \|S(t)u_0\|_{X^{0,\frac1{q'}+}_{\infty, q}(T)}\|S(t)v_0 + v\|_{X^{0,\frac1{q'}+}_q(T)}\|S(t)v_0 + v\|_{X^{\sigma,0}_q(T)}\|\chi_T w\|_{X^{0,0}_{q'}(T)} \\
&+ \|S(t)v_0 + v\|_{X^{0,\frac1{q'}+}_q(T)}^2\|S(t)v_0 + v\|_{X^{0,0}_q(T)}\|\chi_T w\|_{X^{0,0}_{q'}(T)} \\
&\les T^{1-} \| w(t)\|_{X^{0,\frac1q-}_{q'}}\big( \|u_0\|_{FL^{\frac\sigma3,6}}^3 + (1+\|v_0\|_{L^2})^2)(1+\|v_0\|_{H^\sigma})\big).
\end{align*}
Therefore, \eqref{RRpreservation} follows as long as 
\begin{align*}
T & \ll \|u_0\|_{\FL^{\frac\sigma3,6}}^{-3-}, \\
T & \ll \|u_0\|_{\FL^{0,\infty}}^{-2-}, \\
T& \ll (1+\|v_0\|_{L^2})^{-2-}.
\end{align*}
It is easy to check that all of these conditions hold under \eqref{Tpreservationlwp}. This completes the proof of Proposition~\ref{PROP:LWP}.

\subsection{Local well-posedness for the truncated equation}

We conclude this section by stating an analogous version of Proposition \ref{PROP:LWP} for the truncated system \eqref{4NLStrunc}, together with a convergence result as $N \to \infty$. 

\begin{proposition}\label{PROP:LWPN}
Let $\al>1$ and $\tfrac{1}{4}<s\leq \tfrac{1}{2}$ be such that $\al+s>\tfrac{3}{2}$ and let $ 2 \le q < \tfrac{4\alpha}{3-2s}.$
 Given $u_0$ distributed according to $\mu_s$ and $N\in \mathbb{N}$, we define
$$ \Xi_N(u_0) := \big(u_0, J_{12,N}, J_{13,N}, X_N^{(3)} \big), $$
where $J_{12,N}, J_{13,N}$ are defined in \eqref{J12Ndef} and \eqref{J13Ndef}, and $X^{(3)}_N$ is defined in \eqref{X3Ndef}. 
Define
\begin{align*}
\|\Xi_N(u_0)\|_{\mathcal X^{\sigma,q}}:=& \|u_0\|_{H^{s-\frac 12 -}} + \|u_0\|_{\FL^{\frac{\s}{3}, 6}} + \|X^{(3)}_{N}\|_{X^{\sigma,\frac 1{q'} +}_q(1)}\\
& \hphantom{\|u_0\|_{H^{s-\frac 12 -}}}+ \|J_{12,N}\|_{b_J+, q, \sigma,-\sigma} +  \|J_{13,N}\|_{b_J+,q, \sigma, -\sigma} ,
\end{align*}
where $b_J$ is defined in \eqref{bilinearlwpcondition} and $0\leq \s\leq s$. Let $v_0\in H^{s}(\T)$. Then, the equation \begin{equation}\label{FNLSvN}
\begin{cases}
i\partial_t v_N+(-\dx^2)^\alpha v_N +\NN_N( S(t)(u_0+v_0)+v_N)+\RR_N(S(t)(u_0+v_0)+v_N)=0 ,\\
v_N|_{t=0}=0,
\end{cases}
\end{equation}
is locally-well posed in $L^2(\T)$. Moreover, if $\tau_{0,N}$ is defined as 
\begin{align} \label{tau0Ndef}
\tau_{0,N}(u_0,v_0)=\max\{ t\geq 0\,:\, \|v_N\|_{X^{0,\frac{1}{q'}+}_{q}(T)} \les 1+\|v_0\|_{L^2}\},
\end{align}
then there exists $C_{\eps}>0$ such that 
\begin{align}
\tau_{0,N}(u_0,v_0)^{-1} \les 1 + \|v_0\|_{L^2}^{\frac{4\al}{2\al-1}+\eps} +\|\Xi_N(u_0)\|_{\mathcal X^{0,q}}^{C_{\eps}}, \label{tau0Nbound}
\end{align}
where the implicit constant is independent of $N$.
Moreover, suppose that 
\begin{equation}
\lim_{N \to \infty} \|\Xi_N(u_0) - \Xi(u_0) \|_{\mathcal X^0} \to 0, \label{Xiconvergencelwp}
\end{equation}
and that for some $0 <T \le 1$, and for some $N \ge N_0(u_0,C) \gg 1$, 
\begin{equation}\label{vNboundedness}
\| v_N \|_{L^\infty([0,T], L^2)} \le C < \infty.
\end{equation}
Then, if $v$ is the solution of \eqref{FNLSv}, we have that $v$ can be extended up to time $T$, and 
\begin{equation}\label{vboundedness}
\| v \|_{L^\infty([0,T], L^2)} \le C +1.
\end{equation}
Moreover, in this case,
\begin{equation}\label{vconvergence}
\lim_{N \to \infty} \| v - v_N \|_{L^\infty([0,T], L^2)} =0.
\end{equation}
Furthermore, if $v$ is a solution to \eqref{FNLSv} such that \eqref{Xiconvergencelwp}, \eqref{vboundedness} hold, then there exists $N_1(u_0,C)\gg 1$ such that for $N\geq N_0 (u_0,C)$, $v_N$ can be extended up to time $T$ and  satisfies 
\begin{equation}
\label{vNboundedness2}
\| v_N \|_{L^\infty([0,T], L^2)} \le C+2 < \infty,
\end{equation}
and \eqref{vconvergence} holds.
Finally, for $0 \le \sigma \le s$, define 
\begin{align} \label{tausigmaNdef}
\tau_{\sigma,N}(u_0,v_0)=\max\{ 0 \le t \le \tau_{0}(u_0,v_0)\,:\, \|v_N\|_{X^{\sigma,\frac{1}{q'}+}_{q}(T)} \les 1+\|v_0\|_{H^\sigma}\}.
\end{align}
Then there exists $C_{\eps,\sigma}>0$ such that 
\begin{align}
\tau_{\sigma}(u_0,v_0)^{-1} \les 1+\|v_0\|_{L^2}^{\frac{4\al}{2\al-1}+\eps} +\|\Xi_N(u_0)\|_{\mathcal X^{0,q}}^{C_{\eps}}+\|\Xi_N(u_0)\|_{\mathcal X^{\sigma,q}}^{C_{\eps,\sigma}}, \label{tausigmaNbound}
\end{align}
where the implicit constant is independent of $N$.
\end{proposition}
The proof of the estimates \eqref{tau0Nbound} and \eqref{tausigmaNbound} are completely analogous to the proofs of the respective estimates \eqref{tau0bound} and \eqref{tausigmabound}, and the convergence results \eqref{vboundedness}, \eqref{vconvergence}, and \eqref{vNboundedness2} follow by a standard argument. Hence, we omit the proof.

\section{Uniform exponential tail estimates}

For fixed $u_0$, let $\VV_N(u_0)$ be the solution of the equation
\begin{equation} \label{4NLSshiftedtruncated2}
\begin{cases}
i \partial_t v + (-\dx^2)^{\al} v + \NN_{N}(S(t)u_0 + v) + \Res_{N}(S(t)u_0 + v)= 0,\\
v(0) = 0,
\end{cases}
\end{equation}
where we recall the notation in \eqref{nonlineartruncationcomp}.
Note that this corresponds to \eqref{FNLSvN} with $v_0 = 0$. Moreover, we have $\VV_{N}(u_0)=\VV_{N}(P_N u_0)$.
Let $2 \le q <\frac{4\alpha}{3-4s}$, and define the stopping time $ \tau_N=\tau_{N}(u_0)$ in the following way:
\begin{equation}\label{taudef}
\begin{aligned}
\tau_{N}(u_0) = \max\bigg\{ M^{-\frac{4\alpha}{2\alpha-1}-}: M\ge1 \text{ dyadic},&\ \|\VV_N(u_0)\|_{X^{0,\frac1{q'}+}_q(M^{-\frac{4\alpha}{2\alpha-1}-})} \les M, \\
&\ \|\VV_N(u_0)\|_{X^{s,\frac1{q'}+}_q(M^{-\frac{4\alpha}{2\alpha-1}-})}\les M^\gamma \bigg\}.
\end{aligned}
\end{equation}
We do not keep track of the dependence of $\tau_{N}$ on $q$ as we will make a concrete choice of $q$ in Lemma~\ref{LEM:densityontau}, which will then be fixed for the rest of the paper. 

The aim of this section is to prove that the inverse of $\tau_{N}(u_0)$ is exponentially integrable, uniformly in $N$, see Lemma~\ref{lem:tauintegrability} below. This verifies one of the key assumptions for Bourgain's quasi-invariant measure argument (Theorem~\ref{bima}), which we will use to globalise our local-in-time solutions from Section~\ref{SEC:LWP}.

We begin by verifying that the random object $\Xi(u_0)$ is indeed almost surely finite in $\mathcal{X}^{\s}$ and can be approximated by $\Xi_{N}(u_0)$. It then follows from Proposition~\ref{PROP:LWPN} that $\tau_{N}(u_0)$ is almost surely finite. 

\begin{lemma}\label{LEM:xiconvergence}
For every $1\leq p < \infty$ and for every $0 \le \sigma \le s$, we have that 
 \begin{equation}\label{xiintegrability}
 \sup_{ N\in \mathbb{N}}  \E \big[ \|\Xi_N(u_0)\|_{\mathcal X^\sigma}^p  \big]\les_{\sigma,p} 1, 
 \end{equation}
where the expectation is taken with respect to the measure $\mu_s$.
 Moreover, for $\mu_s$-a.e. $u_0$, we have that 
\begin{equation}\label{xiconvergence}
\lim_{N \to \infty} \|\Xi(u_0) - \Xi_N(u_0) \| = 0.
\end{equation}
\end{lemma}

\begin{proof}
We show \eqref{xiintegrability} and \eqref{xiconvergence} for the terms of $\Xi_N(u_0)$ one-by-one. 

\medskip
\noi
 \underline{\textbf{Term 1:}} $\|u_0\|_{H^{s-\frac12-}}$. For this term the convergence is automatic (since it does not depend on $N$), so we just need to show the integrability \eqref{xiintegrability}. Recall that we can express $u_0$ in terms of the Fourier expansion \eqref{u0}.
Thus, for finite $p\geq 2$,  Minkowski's inequality implies
\begin{align*}
\E\big[ \|u_0\|_{H^{s-\frac12-}}^p \big] & \les \bigg( \sum_{n\in \Z}  \frac{ \jb{n}^{2s-1-}\E[|g_n|^{p}]^{\frac{2}{p}}}{\jb{n}^{2s}} \bigg)^{\frac{p}{2}}  \les_{p} \bigg( \sum_{n\in \Z} \frac{\jb{n}^{2s-1-}}{\jb{n}^{2s}}\bigg)^{\frac{p}{2}}\les_{p} 1.
\end{align*}

\medskip
\noi
 \underline{\textbf{Term 2:}} $\|u_0\|_{\FL^{\frac\sigma3,6}}$.  This term is also independent of $N$, so convergence is automatic. We just need to show the integrability \eqref{xiintegrability}. For $p\geq 6$, Minkowski's inequality again implies \begin{align*}
\E \big[ \|u_0\|_{\FL^{\frac\sigma3,6}}^p \big] &\les \bigg( \sum_{n\in \Z}  \frac{ \jb{n}^{2\s}\E[|g_n|^{p}]^{\frac{6}{p}}}{\jb{n}^{6s}} \bigg)^{\frac{p}{6}}  \les_{p} \bigg( \sum_{n\in \Z} \frac{\jb{n}^{2\s}}{\jb{n}^{6s}}\bigg)^{\frac{p}{6}}\les_{p} \bigg( \sum_{n\in \Z} \frac{1}{\jb{n}^{4s}}\bigg)^{\frac{p}{6}} \les_{p} 1,
\end{align*}
which is finite since $s>\tfrac{1}{4}$.

\medskip
\noi
\underline{\textbf{Terms 3 \& 4}:} $\|J_{12,N}\|_{b_J+,q,\sigma,-\sigma}$ and  $\|J_{13,N}\|_{b_J+,q,\sigma,-\sigma}$. 
The uniform integrability follows from \eqref{Jconvergence} and \eqref{Jintegrability}. Moreover, the a.s.\ convergence \eqref{xiconvergence} of $J_{12,N}$ to $J$ and similarly of $J_{13,N}$ to $J_{13}$ follow from the estimate \eqref{Jconvergence} and the Borel-Cantelli lemma.

\medskip
\noi
\underline{\textbf{Term 5}:} $\|X^{(3)}_N\|_{X^{\sigma,\frac1{q'}+}_q(1)}$. The uniform integrability \eqref{xiintegrability} follows from \eqref{X3integrability} and \eqref{X3convergence}. The a.s.\ convergence \eqref{xiconvergence} follows from \eqref{X3convergence} and the Borel-Cantelli lemma. 

\end{proof}

We now prove the exponential integrability of $\tau_{N}(u_0)^{-1}$, \textit{uniform} in $N$.

\begin{lemma} \label{lem:tauintegrability}
For every $\beta < \frac{2\alpha-1}{2\alpha}\gamma$, there exists a constant $C = C(\alpha,\beta,\gamma,s)$ such that we have 
\begin{equation} \label{tauintegrability}
\sup_{N\in \mathbb{N} }    \int \exp\big(\tau_{N}(u_0)^{-\beta}\big) d\rho_{s,\gamma}(u_0) \le C(\alpha, \beta,\gamma,s) < \infty.
\end{equation}
\end{lemma}
\begin{proof} 
Note that $\tau_{N}(u_0)=\tau_{N}(P_N u_0)$.
By Lemma~\ref{LEM:BD2}, it is enough to show that
\begin{align*}
 \E\Big[\sup_{V_0 \in H^s}  \Big\{  \tau_{N}(Y + V_0)^{-\beta} - \Big|\int \wick{|Y + V_0|^2}\Big|^\gamma - \frac12 \| V_0\|_{H^s}^2 \Big\} \Big] \leq K
\end{align*}
for some constant $K>0$ depending on $\alpha, \beta, \gamma$ and $s$. For readability, we have simply written $P_N Y$ as $Y$ and $P_N V_0$ as $V_0$.

Recalling the definition \eqref{tau0Ndef} of $\tau_{0,N}$, the definition \eqref{tausigmaNdef} of $\tau_{s,N}$, and the definition of $\tau$ given in \eqref{taudef}, we have that 
\begin{equation*}
\begin{aligned}
\{ \tau_{N}(Y+V_0)^{-1} = M^{\frac{4\alpha}{2\alpha-1}+} \} &\subseteq \{ \| V_0 \|_{L^2} \ge M \} \cup \{ \| V_0 \|_{H^s} \ge M^\gamma \} \\
&\phantom{\subseteq\,\,\,}  \cup \{ \tau_{s,N}(Y,V_0)^{-1} \gtrsim M^{\frac{4\alpha}{2\alpha-1}+})\}.
\end{aligned}
\end{equation*}
Therefore, from \eqref{tausigmaNbound}, we obtain that 
\begin{align*}
\tau_{N}(Y+V_0)^{-1} &\les \| V_0 \|_{L^2}^{ \frac{4\alpha}{2\alpha-1}+} + \| V_0 \|_{H^s}^{ \frac{4\alpha}{\gamma(2\alpha-1)}+} + \tau_{s,N}(Y,V_0)^{-1} \\
&\les (1+ \| V_0 \|_{L^2})^{ \frac{4\alpha}{2\alpha-1}+} + \|\Xi_N(u_0)\|_{\mathcal X^{0,q}}^{C_{\eps}} + \|\Xi_N(u_0)\|_{\mathcal X^{s,q}}^{C_{\eps,s}}+ \| V_0 \|_{H^s}^{ \frac{4\alpha}{\gamma(2\alpha-1)}+} 
\end{align*}
Therefore, by \eqref{l2integrability} and \eqref{xiintegrability}, for some constant $C$ that can change from line to line, we have that 
\begin{align*}
& \E\Big[ \sup_{V_0 \in H^s} \Big\{ \tau_{N}(Y + V_0)^{-\beta} - \Big|\int \wick{(Y + V_0)^2}\Big|^\gamma -\frac12 \| V_0\|_{H^s}^2 \Big\} \Big] \\
&\leq  \E\Big[\sup_{V_0 \in H^s} \Big\{ C(1 + \| V_0 \|_{L^2})^{ \beta(\frac{4\alpha}{2\alpha-1}+)} + C\| V_0 \|_{H^s}^{ \beta(\frac{4\alpha}{\gamma(2\alpha-1)}+)} \\
&\phantom{\ge \inf_{V_0 \in H^s} \E\Big[ ]} + C\|\Xi_N(u_0)\|_{\mathcal X_0}^{C_{\eps}} +C\|\Xi_N(u_0)\|_{\mathcal X_0}^{C_{\eps,s}}+ C_\gamma \|Y\|_{H^{s-\frac12-\delta}}^\kappa +C_\gamma \Big| \int \wick{Y^2}\Big|^\gamma \\
&\phantom{\ge \inf_{V_0 \in H^s} \E\Big[ ]} - \tfrac12 \|V_0\|_{L^2}^{2\gamma}- \tfrac14 \| V_0\|_{H^s}^2\Big\} \Big] \\
&\leq   \E\Big[ \sup_{V_0 \in H^s} \Big\{ C(1 + \| V_0 \|_{L^2})^{ \beta(\frac{4\alpha}{2\alpha-1}+)} + C\| V_0 \|_{H^s}^{ \beta(\frac{4\alpha}{\gamma(2\alpha-1)}+)}  - \tfrac12 \|V_0\|_{L^2}^{2\gamma}- \tfrac14 \| V_0\|_{H^s}^2 \Big\} \Big] + C,
\end{align*}
which is finite, as long as
\begin{align*}
\beta\big(\tfrac{4\alpha}{2\alpha-1}+\!\big) < 2\gamma, \quad \text{and} \quad  \beta\big(\tfrac{4\alpha}{\gamma(2\alpha-1)}+\!\big) < 2,
\end{align*}
which is exactly the condition that we imposed on $\beta$. 
\end{proof}

\section{Uniform integrability of the transported density}

The main thrust of this section is to obtain a uniform in $N$ bound on the $L^p(\rho_{s,\g})$-norm of the density of the transported measure $(\Phi_{t}^{N})_{\#} \rho_{s,\g}$, for at least some $1<p<\infty$. In fact, we show it for all finite $p<\infty$.
 From Lemma~\ref{LEM:density} below, we have an explicit formula for the density (Radon-Nikodym derivative) 
\begin{align*}
f_{t}^{N}(u_0) : = \frac{ (\Phi_{t}^{N})_{\#} \rho_{s,\g}}{d\rho_{s,\g}}(u_0),
\end{align*}
for each fixed $N$. This density is trivially in $L^p(\rho_{s,\g})$ for all $p<\infty$ albeit with a bound that grows with $N$; see Remark~\ref{RMK:LpbdN}. The key triumph in this section is, indeed, a uniform in $N$ bound on the $L^p(\rho_{s,\g})$-norm of $f_{t}^{N}$ for any finite $p<\infty$. This is the result of Proposition~\ref{prop:lpbound}, see \eqref{lpdensity}. The proof of Proposition~\ref{prop:lpbound} is built up from the results of Lemma~\ref{LEM:densityontau} and Lemma~\ref{LEM:ftNSM}. One of the major novelties here is the proof of Lemma~\ref{LEM:ftNSM}, which shows that bounds on the density,
at least for as long as the solutions $\Phi_{t}^{N}(u_0)$ can be assumed to be controlled, are essentially reduced to uniform bounds on a single time interval where local well-posedness holds; see \eqref{densitybd1}.
The argument exploits the explicit formula for the density from \eqref{densityformula} and the fact that it is the Radon-Nikodym derivative.
 We prove the remaining uniform bounds in Lemma~\ref{LEM:densityontau}, using the Bou\'e-Dupuis formula, the local well-posedness theory from Section~\ref{SEC:LWP}, and suitable energy estimates from Subsection~\ref{SEC:energyests}.

\subsection{A formula for the transported density}
We define the operator
\begin{equation}
\QQ(u_1,u_2,u_3,u_4) := \text{Re}  \sum_{\substack{n_1 - n_2 + n_3 - n_4 = 0, \\ n_1 \neq n_2,n_4}} i{\Psi_s(\cj{n})} \ft{u_1}(n_1) \cj{\ft{u_2}(n_2)} \ft{u_3}(n_3) \cj{\ft{u_4}(n_4)}, \label{Qenergy}
\end{equation}
where we recall that $\Psi_{s}(\cj{n})$ was defined in \eqref{symderiv}.
With an abuse of notation, we denote
\begin{equation*}
\QQ(u) := \QQ(u,u,u,u).
\end{equation*}

\begin{lemma}\label{LEM:density}
Let $\Phi_t^N(u_0)$ denote the solution of the equation \eqref{4NLStrunc} and let $f_t^N$ denote the density of the measure $(\Phi_t^N)_{\#}\rho_{s,\gamma}$, i.e.\
\begin{equation*}
\int F(\Phi_t^N(u_0)) d \rho_{s,\gamma}(u_0) = \int F(u_0) f^{N}_t(u_0) d \rho_{s,\gamma}(u_0).
\end{equation*}
Then 
\begin{equation} \label{densityformula}
f^{N}_t(u_0) = \exp\bigg( \int_0^t \QQ(P_N\Phi_{-t'}^N(u_0),P_N\Phi_{-t'}^N(u_0),P_N\Phi_{-t'}^N(u_0),P_N\Phi_{-t'}^N(u_0)) dt' \bigg).
\end{equation}
\end{lemma}
\begin{proof}
We observe that if $u$ solves \eqref{4NLStrunc}, then we can write $u = u_1 + u_2$, with $u_1$ solving the finite dimensional system of ODEs
\begin{equation} 
i\partial_t u_1 +(-\dx^2)^{\al} u_1 + \mathcal{N}_{N}(u_1)+\Res_{N}(u_1) = 0, \label{lowfreqNLStrunc}
\end{equation}
where $u_1 = \sum_{|n|\le N} \ft{u_1}(n,t) e^{inx}$, with initial data $u_1(0)=P_{N}u_0$, and $u_2$ solves the linear equation
\begin{equation*}
i\partial_t u_2 + (-\dx^2)^{\al} u_2 = 0,
\end{equation*}
with initial data $u_2(0) = P_{>N} u(0)$, where $P_{>N}:= \text{Id}-P_N$. Let $\Phi^{N,1}_t$ and $\Phi^{N,2}_t$ be the two flows, respectively.
We notice that for $x = \sum_{|n| \le N} x_n e^{inx}$, $f^{N}_t(x)$ solves the Liouville equation:
\begin{equation*}
\partial_t \Big(f^{N}_t(x) \exp\big(-\tfrac12 \|x\|_{H^s}^2\big)  \Big) = - \mathrm{div}\Big( f^{N}_t(x) \exp\big(-\tfrac12 \|x\|_{H^s}^2\big)\cdot(i (-\dd_{x}^{2})^\alpha x + i\NN_N(x,x,x))\Big).
\end{equation*}
Therefore, we have that 
\begin{equation*}
\int F(\Phi_t^{N,1}(x)) d (P_N)_{\#} \mu_s(x) = \int F(x) f^{N}_t(x) d(P_N)_{\#} \mu_s(x).
\end{equation*}
Moreover, it is easy to check that $\Phi^{N,2}_t$ preserves the $H^s$ norm and commutes with the projection $P_N$. Therefore, 
\begin{equation*}
\int F(\Phi_t^{N,2}(u_2)) d (1 - P_N)_{\#} \mu_s(u_2) = \int F(u_2) d(P_{>N})_{\#} \mu_s(u_2).
\end{equation*}
Noticing that $(P_N)_\#\mu_s$ and $(P_{>N})_\#$ are independent, we have
$$\int F(u) d \mu_s(u) = \iint F(x+y) d(P_N)_{\#} \mu_s(x) d (P_{>N})_{\#} \mu_s(y).$$
Therefore, by Fubini, we obtain
\begin{align*}
&\int F(\Phi_t^{N}(u_0)) d \mu_s(u_0) \\
&= \int F(\Phi_t^{N,1}(P_N u_0) + \Phi_t^{N,2}((P_{>N})u_0)) d \mu_s(u_0) \\
& = \iint F(\Phi_t^{N,1}(x) + \Phi_t^{N,2}(y)) d(P_N)_{\#} \mu_s(x) d(P_{>N})_{\#} \mu_s(y) \\
&= \iint F(x+y) f^{N}_t(x) d(P_N)_{\#} \mu_s(x) d(P_{>N})_{\#} \mu_s(y) \\
& = \int F(u_0) f^{N}_t(P_N(u_0))d \mu_s(u_0) \\
& = \int F(u_0) f^{N}_t(u_0)d \mu_s(u_0).
\end{align*}
In order to conclude the proof, we notice that for every $M \ge N$, $ \int |P_Mu_0|^2 dx$
is conserved by the flow $\Phi^N$. Therefore, the quantity 
$$ \int : |u_0|^2 : dx = \lim_{M \to \infty} \bigg( \int |P_Mu_0|^2 dx- \E\Big[ \int |P_M u_0|^2  dx\Big]  \bigg) $$
is conserved as well. In particular, we obtain that 
\begin{align*}
\int F(\Phi_t^{N}(u_0)) d \rho_{s,\gamma}(u_0) &= \int F(\Phi_t^{N}(u_0)) \exp\Big(-\Big|\int : |u_0|^2 : dx\Big|^\gamma\Big)  d \mu_s(u_0) \\
& = \int F(\Phi_t^{N}(u_0)) \exp\Big(-\Big|\int : |\Phi_t^{N}(u_0)|^2 : dx\Big|^\gamma\Big)  d \mu_s(u_0)\\
& = \int F(u_0) \exp\Big(-\Big|\int : |u_0|^2 : dx\Big|^\gamma\Big) f^{N}_t(u_0)d \mu_s(u_0) \\
& = \int F(u_0) f^{N}_t(u_0)d \rho_{s,\gamma}(u_0).  \qedhere
\end{align*}
\end{proof}

\begin{remark}\rm \label{RMK:LpbdN}
It is easy to see (although important later on) that $\int f_{t}^{N}(u_0)^{p}d\rho_{s,\g}(u_0)$ is finite for each fixed $N\geq 1$, $t\geq 0$, and for any $1<p<\infty$.
 Indeed, it follows from Proposition~\ref{4NLStrunc}, \eqref{Qenergy}, writing $P_{N}\Phi^{N}_{t}(u_0)=\Phi_{t}^{N,1}(u_0)$, and the $L^2$-conservation of solutions to \eqref{lowfreqNLStrunc}, that 
\begin{align*}
|\mathcal{Q}(\Phi_{t}^{N,1}(u_0))| &\les\|  \Phi_{t}^{N,1}(u_0)\|_{H^1}^{4} 
 \leq C_{N}  \|  \Phi_{t}^{N,1}(u_0)\|_{L^2}^{4} 
 = C_N\|P_{N} u_0\|_{L^2}^{4}.
\end{align*}
Therefore, 
\begin{align*}
\int f_{t}^{N}(u_0)^{p}d\rho_{s,\g}(u_0) & \les \int \exp\bigg( pC_{N,t} \|P_N u_0\|_{L^2}^{4} - \bigg\vert \int \wick{|P_{N}u_0|^2}dx \bigg\vert^{\g}  \bigg) d(P_N)_{\#}\mu_{s}(u_0),
\end{align*}
which is finite since $2\g>4$.
\end{remark}

\begin{remark}\label{QQdivergence}
In this remark, we expand on the claim made in the introduction that states that $\QQ(u_0)$ is ill defined for $u_0$ distributed according to $\mu_s$. More precisely, we will show that 
for every $0 < s \le 1$, 
\begin{equation*}
\int |\QQ(u_0)|^2 d \mu_s(u_0) = +\infty.
\end{equation*}
By \eqref{u0}, we have that 
\begin{align*}
\int |\QQ(u_0)|^2 d \mu_s(u_0) \sim \sum_{\substack{n_1-n_2+n_3 -n_4 =0, \\ n_1 \neq n_2,n_4}} |\Psi_s(\cj{n})|^2 \prod_{j=1}^4 \jb{n_j}^{-2s}.
\end{align*}
If we restrict this sum to the domain $|n_j| \sim N_j$ with $N_1 \sim N_2 \gg N_3 \gg N_4$, it is easy to check that $|\Psi_s(\cj{n})|^2 \sim N_3^{4s}$. Moreover, we have that 
$$|\{\jb{n_j}\sim N_j: n_1-n_2+n_3 -n_4 =0, n_1 \neq n_2,n_4, N_1 \sim N_2 \gg N_3 \gg N_4 \}| \sim N_2N_3N_4.$$
Therefore, we obtain that 
\begin{align*}
\int |\QQ(u_0)|^2 d \mu_s(u_0) &\ge \sum_{N_1 \sim N_2 \gg N_3 \gg N_4} N_3^{4s} \Big(\prod_{j=1}^{4} N_j^{-2s}\Big) N_2N_3N_4 \\
& \sim \sum_{N_2 \gg N_3 \gg N_4} N_2^{1-4s} N_3^{1+2s} N_4^{1-2s}\\
& \gtrsim \sum_{N_3 \gg N_4} N_3^{2-2s}N_4^{1-2s} = \infty,
\end{align*}
since $s \le 1$.
\end{remark}

\subsection{Energy estimates}\label{SEC:energyests}

In this subsection, we prove estimates on the quantity $\mathcal{Q}$ from \eqref{Qenergy}. We view these as energy estimates since
\begin{align*}
\frac{d}{dt} \bigg(\frac{1}{2} \|\Phi_{-t}^{N,1}(u_0)\|_{H^{s}}^{2}\bigg) =  \mathcal{Q}(\Phi_{-t}^{N,1}(u_0)).
\end{align*}
We will again leverage the expansion $\Phi_{t}^{N,1}(u_0)=S(t)P_N u_0 +\mathcal{V}_{N}(u_0)(t)$,
where $\mathcal{V}_{N}(u_0)(t)$ solves \eqref{4NLSshiftedtruncated2}, and the multi-linearity of $\mathcal{Q}$ to reduce to cases depending upon the number of linear factors $P_{N}S(t)u_0$. Such a decomposition allows us to control certain contributions in a probabilistic manner, which seem to be out of reach with deterministic methods.

We consider the first case when all four factors in $\mathcal{Q}$ are of the form $P_{N}S(t)u_0$. As we will eventually use Lemma~\ref{LEM:BD2}, we will switch notation writing $Y$, rather than $u_0$.
 
\begin{lemma} For $\frac14 < s \le \frac12$, let $Y$ be distributed according to $\mu_s$, and let $0 \le t \le 1$ and $\al>1$. Then,
\begin{equation} \label{QQu0estimate}
\E\Big[ \Big| \int_0^t \QQ(P_NS(t')Y) dt \Big|^2 \Big] \les t,
\end{equation}
where the implicit constant does not depend on $N$.\footnote{The power of $t$ can be improved to $t^{2(1-\frac{3-4s}{4\al}-)}$. However, this worse estimate does not affect the numerology in Theorem~\ref{thm:qiandgwp}.}
\end{lemma}
\begin{proof}
Recall that we can express $Y$ in terms of the Fourier series expansion in \eqref{u0}.
Therefore, we have that 
\begin{align*}
\E \bigg[ \Big|\int_0^t\QQ(P_NS(t')Y)dt'\Big|^{2} \bigg] &\leq  \E \bigg[\Big|\int_0^t\sum_{\substack{n_1 - n_2 + n_3 - n_4 = 0, \\ n_1 \neq n_2,n_4, \\ |n_j| \le N}} {\Psi_s(\cj{n})}\frac{g_{n_1}\cj{g_{n_2}}g_{n_3}\cj{g_{n_4}}}{\prod_{j=1}^4 \jb{n_j}^s}e^{it'\Phi_\alpha(\cj{n})}dt'\Big|^2\bigg]\\
&=  \E\bigg[\Big|\sum_{\substack{n_1 - n_2 + n_3 - n_4 = 0, \\ n_1 \neq n_2,n_4, \\ |n_j| \le N}} {\Psi_s(\cj{n})} \frac{g_{n_1}\cj{g_{n_2}}g_{n_3}\cj{g_{n_4}}}{\prod_{j=1}^4 \jb{n_j}^s}\frac{e^{it\Phi_\alpha(\cj{n})}-1}{i\Phi_\alpha(\cj{n})}\Big|^2 \bigg]\\
&=4 \sum_{\substack{n_1 - n_2 + n_3 - n_4 = 0, \\ n_1 \neq n_2,n_4, \\ |n_j| \le N}} \frac{|\Psi_s(\cj{n})|^2}{\prod_{j=1}^4 \jb{n_j}^{2s}}\frac{|e^{it\Phi_\alpha(\cj{n})}-1|^2}{|\Phi_\alpha(\cj{n})|^2}\\
&\les t \sum_{\substack{n_1 - n_2 + n_3 - n_4 = 0, \\ n_1 \neq n_2,n_4, \\ |n_j| \le N}} \frac{|\Psi_s(\cj{n})|^2}{\jb{\Phi_\alpha(\cj{n})}\prod_{j=1}^4 \jb{n_j}^{2s}}.
\end{align*}
 By \eqref{Phase} and \eqref{psisbd}, we have that 
\begin{align*}
&\sum_{\substack{n_1 - n_2 + n_3 - n_4 = 0, \\ n_1 \neq n_2,n_4, \\ |n_j| \le N}} \frac{|\Psi_s(\cj{n})|^2}{\jb{\Phi_\alpha(\cj{n})}\prod_{j=1}^4 \jb{n_j}^{2s}}  \\
&\les \sum_{\substack{n_1 - n_2 + n_3 - n_4 = 0, \\ n_1 \neq n_2,n_4, }} \frac{1}{\jb{n_1-n_4}^{1-s}\jb{n_1-n_2}^{1-s}n_{\max}^{2\alpha-2}\prod_{j=1}^4 \jb{n_j}^{2s}}\\
&\les \sum_{\substack{N_1,N_2,N_3,N_4 \\ N^1 \sim N^2}} (N_1 N_2 N_3 N_4)^{-2s}N_{\max}^{-2\alpha+2} \sum_{\substack{n_1 - n_2 + n_3 - n_4 = 0,\\ \jb{n_j} \sim N_j}} \frac{1}{\jb{n_1-n_4}^{1-s}\jb{n_1-n_2}^{1-s}}\\
&\les \sum_{\substack{N_1,N_2,N_3,N_4 \\ N^1 \sim N^2}} (N_1 N_2 N_3 N_4)^{-2s}N_{\max}^{-2\alpha+2} (N^{1})^s(N^2)^sN^4 \\
&\les \sum_{\substack{N_1,N_2,N_3,N_4 \\ N^1 \sim N^2}} (N^1)^{-2s-2\alpha + 2}  (N^3)^{-2s} (N^4)^{1-2s} \les 1,
\end{align*}
since $3-6s-2\al<0$.
\end{proof}

The following lemma controls $\mathcal{Q}$ when three of its inputs are of the form $P_{N}S(t)Y$.

\begin{lemma}
Given $0\leq s\leq \tfrac{1}{2}$, let $Y$ be distributed according to $\mu_s$, and let $\al>1$. Define 
\begin{equation*}
 Y^{(3),N}(t) := \sum_{\substack{n_1,n_2,n_3, \\ n_2 \neq n_1,n_3, \\ |n_j|\le N}} \Psi_s(\cj{n}) \ft{Y}(n_1)\cj{\ft{Y}(n_2)}\ft{Y}(n_3)e^{it(|n_1|^{2\alpha}-|n_2|^{2\alpha}+|n_3|^{2\alpha})} e^{i(n_1-n_2+n_3)x}.
\end{equation*}
Let 
\begin{equation}\label{sigma3def}
\sigma_3(\alpha,s):= \begin{cases}
\tfrac 32 - \alpha - s + & \text{ if } \alpha+s \le \frac32,\\
0 & \text{ otherwise.}
\end{cases}
\end{equation}
Then, for every $0 \le T \le 1$, and for every $2 \le q < \infty$, $ 1 \leq  p < \infty$, we have that 
\begin{equation} \label{Y3estimate}
\E \Big[ \| \chi_T(t) Y^{(3),N}(t) \|_{X^{-\sigma_3(\alpha,s),-\frac1{q'}-}_{q'}}^p \Big] \les_{p,q} T,
\end{equation}
where the implicit constant does not depend on $N$.
Moreover, if 
\begin{equation}\label{Y3def}
 Y^{(3)}(t) := \sum_{\substack{n_1,n_2,n_3, \\ n_2 \neq n_1,n_3}} \Psi_s(\cj{n}) \ft{Y}(n_1)\cj{\ft{Y}(n_2)}\ft{Y}(n_3)e^{it(|n_1|^{2\alpha}-|n_2|^{2\alpha}+|n_3|^{2\alpha})} e^{i(n_1-n_2+n_3)x},
\end{equation}
then we have that for $\mu_s$-a.e. $Y$, 
\begin{equation*}
\label{Y3convergence}
\lim_{N \to \infty} \| Y^{(3)} - Y^{(3),N} \|_{X^{-\sigma_3(\alpha,s),-\frac1{q'}-}_{q'} (1)} = 0.
\end{equation*}
\end{lemma}

\begin{proof}
Recalling the definition of the $X^{\sigma,b}_{q}$ norm, and proceeding as in the proof of Proposition \ref{randomtrilwp} and using \eqref{psisbd}, we obtain for finite $p \ge 2$,
\begin{align}
&\E \Big[\| \chi_{T}Y^{(3),N}\|_{X^{-\sigma_3(\alpha,s),-\frac1{q'}-}_{q'}}^p\Big]^\frac 2p   \notag\\
&\les \sum_{\substack{n_1-n_2+n_3-n = 0, \\n_2 \neq n_1,n_3,\\ |n_j|\les N}} \Big|\int_\R \Big[ \frac{|\Psi_s(\cj{n})|^2 |\ft{\chi_T}(\tau-\Phi_\alpha(\cj{n}) - |n|^{2\alpha})|^2}{\jb{n_1}^{2s}\jb{n_2}^{2s}\jb{n_3}^{2s}\jb{n}^{2\sigma_3}}\Big]^{\frac {q'}2} \jb{\tau-|n|^{2\alpha}}^{-1-} d \tau \Big|^{\frac 2{q'}}   \notag \\
&\les  \sum_{\substack{n_1-n_2+n_3-n = 0, \\n_2 \neq n_1,n_3,\\ |n_j|\les N}} \frac{|\Psi_s(\cj{n})|^2 }{\jb{n_1}^{2s}\jb{n_2}^{2s}\jb{n_3}^{2s}\jb{n}^{2\sigma_3}} \Big|\int |\ft{\chi_T}(\tau-\Phi_\alpha(\cj{n}))|^{q'}\jb{\tau}^{-1-} d \tau\Big|^{\frac 2 {q'}}  \notag  \\
&\les  \sum_{\substack{n_1-n_2+n_3-n = 0, \\n_2 \neq n_1,n_3,\\ |n_j|\les N}} \frac{|\Psi_s(\cj{n})|^2 }{\jb{n_1}^{2s}\jb{n_2}^{2s}\jb{n_3}^{2s}\jb{n}^{2\sigma_3}} \Big|\int T^{\frac{q'}2} \jb{\tau-\Phi_\alpha(\cj{n})}^{-\frac{q'}{2}}\jb{\tau}^{-1-} d \tau\Big|^{\frac 2 {q'}}  \notag \\
&\les T \sum_{\substack{n_1-n_2+n_3 - n_4 = 0, \\ n_2 \neq n_1,n_3, \\ |n_j|\le N}}\frac{|\Psi_s(\cj{n})|^2}{\jb{n_1}^{2s}\jb{n_2}^{2s}\jb{n_3}^{2s}\jb{n_4}^{2\sigma_3}\jb{\Phi_\alpha(\cj{n})}}  \notag\\
&\les T \sum_{ \substack{N_1,N_2,N_3,N_4\\ N^1 \sim N^2}} (N_1N_2N_3)^{-2s}N_4^{-2\sigma_3}N_{\max}^{-2\alpha + 2} \sum_{\substack{n_1-n_2+n_3 - n_4 = 0, \\ n_2 \neq n_1,n_3 \\ \jb{n_j}\sim N_j}} \frac{1}{\jb{n_1-n_4}^{1-2s}\jb{n_1-n_2}^{1-2s}}. \label{YN3bd2}
\end{align}
If $N_1\sim N_3 \ges N_4 \gg N_2$, then this contribution to \eqref{YN3bd2} is bounded by 
\begin{align*}
& \les T \sum_{ \substack{N_1,N_2,N_3,N_4\\ N^1 \sim N^2}} (N_1N_2N_3)^{-2s}N_4^{-2\sigma_3}N_{\max}^{-2\alpha + 2}  (N^1)^{-1+4s}N^3 N^4 \\
&  \les T \sum_{ N_1,N_2,N_3,N_4}(N^1)^{-2\al+3-2s-2\s_3} \les T,
\end{align*}
provided that $\al +s+\s_3>\tfrac{3}{2}$, which holds by definition of $\sigma_3$ in \eqref{sigma3def}.
In all remaining cases, we may bound the contribution to \eqref{YN3bd2} by 
\begin{align*}
&\les T \sum_{\substack{N_1,N_2,N_3,N_4\\ N^1 \sim N^2}} (N_1N_2N_3)^{-2s}N_4^{-2\sigma_3}N_{\max}^{-2\alpha + 2 +} (N^1)^{2s}(N^3)^{2s}N^4  \\ 
& \les  T \sum_{\substack{N_1,N_2,N_3,N_4\\ N^1 \sim N^2}}  (N^4)^{1-2s} (N^1)^{-2(\al-1)-2\s} \les T,
\end{align*}
which is  again true provided that \eqref{sigma3def} holds.
\end{proof}

Our final probabilistic energy estimate concerns $\mathcal{Q}$ when exactly two factors are $P_N S(t)Y$.

\begin{lemma}\label{lem:bilineargwp}
Given $\tfrac{1}{4}<s\leq \tfrac{1}{2}$, let $Y$ be distributed according to $\mu_s$, and let $\al>1$. Define the random bilinear forms $J^\QQ_{12,N}(t)$ and $J^\QQ_{13,N}(t)$ by 
\begin{gather*}
J^\QQ_{12,N}(t)[v_1,v_2] := \QQ(P_NS(t)Y,P_NS(t)Y,v_1,v_2)\label{JQQN12def}, \\
J^\QQ_{13;N}(t)[v_1,v_2] := \QQ(P_NS(t)Y,v_1,P_NS(t)Y,v_2),\label{JQQN13def} 
\end{gather*}
and let $\|J\|_{2,b,q,\sigma,\sigma}$ be the best constant $C$ such that 
\begin{equation*}
\int_\R J(t)[v_1(t),v_2(t)] d t \le C\big( \|v_1\|_{X^{\sigma,\frac1{q'}-\frac 1 q +}_q} \|v_2\|_{X^{\sigma,b}_q} + \|v_1\|_{X^{\sigma,b}_q} \|v_2\|_{X^{\sigma,\frac1{q'}-\frac 1 q +}_q}  \big).
\end{equation*}
for every $v_1,v_2$. Then for every $2 \le q < \infty$, $1\leq p< \infty$, $0 < b < 1$,
\begin{equation}\label{JQQintegrability}
\E \big[ \|J^\QQ_{12,N}\|_{2,b,q,\sigma,\sigma}^p +\|J^\QQ_{13,N}\|_{2,b,q,\sigma,\sigma}^p \big] \les_{b,q,\sigma} 1,
\end{equation}
where the implicit constant is independent of $N$, as long as 
\begin{equation}\label{bJQQcondition}
0 < b < \tfrac 12, \quad 0 \le \sigma \le s, \quad 4b\alpha + 4\sigma > 3.
\end{equation}
Moreover, if we define 
\begin{gather}
J^\QQ_{12}(t)[v_1,v_2] := \QQ(S(t)Y,S(t)Y,v_1,v_2)\label{JQQ12def}, \\
J^\QQ_{13}(t)[v_1,v_2] := \QQ(S(t)Y,v_1,S(t)Y,v_2),\label{JQQ13def}
\end{gather}
then for $\mu_s$ a.e.\ $Y$, we have that 
 \begin{equation}\label{JQQconvergence}
\lim_{N\to \infty} \|J^\QQ_{12,N} - J^\QQ_{12}\|_{2,b,q,\sigma,\sigma} = 0, \quad \lim_{N\to \infty} \|J^\QQ_{13,N} - J^\QQ_{13}\|_{2,b,q,\sigma,\sigma} = 0.
\end{equation}
\end{lemma}

\begin{proof}
We only show the estimate for $J_{12,N}^\QQ$, since the proof of the estimate for $J_{13,N}^\QQ$ is identical.
Proceeding as in the proof of \eqref{Jintegrability} (noticing that $\jb{\tau}^{-(\frac 1{q'}-\frac 1q +)} f \in L^{q'}_{\tau}$ whenever $f \in L^q_\tau$), we (essentially) just need to show that under \eqref{bJQQcondition},
\begin{equation*}
\sup_{\wt \tau_1 \in \R}\E \big[ \|K_{\wt\tau_1}\|_{\mathrm{HS}(\l^2,\l^2)}^2 \big] \les 1,
\end{equation*}
independently of $N$, where $K_{\wt\tau_1}$ is defined as 
$$ K_{\wt\tau_1} = \sum_{\substack{n_1-n_2+n_3-n_4=0, \\ n_1 \neq n_2,n_4,\\ |n_1|\le N, |n_2|\le N}}  \ind_{ \{\jb{\wt\tau_1} \gtrsim \jb{\Phi_\alpha(\cj{n})}\}}\frac{\Psi_s(\cj{n}) g_{n_1} \cj{g_{n_2}}}{\jb{n_1}^s\jb{n_2}^s\jb{n_3}^\sigma\jb{n_4}^\sigma\jb{\wt\tau_1}^b}.$$
We have 
\begin{align*}
\E\big[ \|K_{\wt\tau_1}\|_{\mathrm{HS}(\l^2,\l^2)}^2 \big]
&\les \sum_{\substack{n_1-n_2+n_3-n_4=0, \\ n_1 \neq n_2,n_4,}} \frac{|\Psi_s(\cj{n})|^2 }{\jb{n_1}^{2s}\jb{n_2}^{2s}\jb{n_3}^{2\sigma}\jb{n_4}^{2\sigma}\jb{\Phi_\alpha(\cj{n})}^{2b}}\\
&\les \sum_{ \substack{N_1,N_2,N_3,N_4 \\ N^1 \sim N^2}}(N_1N_2)^{-2s}(N_3N_4)^{-2\sigma} \sum_{\substack{n_1-n_2+n_3-n_4=0, \\ n_1 \neq n_2,n_4,\\|n_j|\sim N_j}} \frac{|\Psi_s(\cj{n})|^2 }{\jb{\Phi_\alpha(\cj{n})}^{2b}}
\end{align*}
To estimate the inner sum, we use \eqref{Phase} and \eqref{psisbd} in multiple different ways, obtaining
\begin{align*}
\sum_{\substack{n_1-n_2+n_3-n_4=0, \\ n_1 \neq n_2,n_4,\\ |n_j|\sim N_j}} \frac{|\Psi_s(\cj{n})|^2 }{\jb{\Phi_\alpha(\cj{n})}^{2b}} &\les N_{\max}^{-2b(2\alpha-2)}\sum_{\substack{n_1,n_2,n_4 \\|n_j|\sim N_j}} \frac{\jb{n_1-n_2}^{4s-2b}}{{\jb{n_1-n_4}^{2b}}} \\
&\les N_{\max}^{-2b(2\alpha-2)}N_1N_2N_4(N_1\vee N_4)^{-2b}(N_1\vee N_2)^{4s-2b},
\end{align*}
\begin{align*}
\sum_{\substack{n_1-n_2+n_3-n_4=0, \\ n_1 \neq n_2,n_4,\\|n_j|\sim N_j}} \frac{|\Psi_s(\cj{n})|^2 }{\jb{\Phi_\alpha(\cj{n})}^{2b}} &\les N_{\max}^{-2b(2\alpha-2)}\sum_{\substack{n_1,n_2,n_4 \\|n_j|\sim N_j}} \frac{\jb{n_1-n_4}^{4s-2b}}{{\jb{n_1-n_2}^{2b}}} \\
&\les N_{\max}^{-2b(2\alpha-2)}N_1N_2N_4(N_1\vee N_4)^{4s-2b}(N_1\vee N_2)^{-2b},\\
\end{align*}
from which we get
\begin{equation*}
\sum_{\substack{n_1-n_2+n_3-n_4=0, \\ n_1 \neq n_2,n_4,\\|n_j|\sim N_j}} \frac{|\Psi_s(\cj{n})|^2 }{\jb{\Phi_\alpha(\cj{n})}^{2b}} \les N_{\max}^{-2b(2\alpha-2)}N_1N_2N_4
(N_1\vee N_4)^{2s-2b}(N_1\vee N_2)^{2s-2b},
\end{equation*}
and similarly
\begin{align*}
\sum_{\substack{n_1-n_2+n_3-n_4=0, \\ n_1 \neq n_2,n_4,\\|n_j|\sim N_j}} \frac{|\Psi_s(\cj{n})|^2 }{\jb{\Phi_\alpha(\cj{n})}^{2b}} \les N_{\max}^{-2b(2\alpha-2)}N_1N_2N_3
(N_2\vee N_1)^{2s-2b}(N_2\vee N_3)^{2s-2b},\\
\sum_{\substack{n_1-n_2+n_3-n_4=0, \\ n_1 \neq n_2,n_4,\\|n_j|\sim N_j}} \frac{|\Psi_s(\cj{n})|^2 }{\jb{\Phi_\alpha(\cj{n})}^{2b}} \les N_{\max}^{-2b(2\alpha-2)}N_1N_3N_4
(N_1\vee N_4)^{2s-2b}(N_3\vee N_4)^{2s-2b},\\
\sum_{\substack{n_1-n_2+n_3-n_4=0, \\ n_1 \neq n_2,n_4,\\|n_j|\sim N_j}} \frac{|\Psi_s(\cj{n})|^2 }{\jb{\Phi_\alpha(\cj{n})}^{2b}} \les N_{\max}^{-2b(2\alpha-2)}N_2N_3N_4
(N_3\vee N_4)^{2s-2b}(N_3\vee N_2)^{2s-2b}.
\end{align*}
Let $N^1 \ge N^2 \ge N^3 \ge N^4$ be a reordering of $N_1,N_2,N_3,N_4$. By the previous estimates, and recalling that $\sigma \le s$, we obtain 
\begin{align*}
&\sum_{N_1,N_2,N_3,N_4}(N_1N_2)^{-2s}(N_3N_4)^{-2\sigma} \sum_{\substack{n_1-n_2+n_3-n_4=0, \\ n_1 \neq n_2,n_4,\\|n_j|\sim N_j}} \frac{|\Psi_s(\cj{n})|^2 }{\jb{\Phi_\alpha(\cj{n})}^{2b}}\\
&\les \sum_{N_1,N_2,N_3,N_4}(N_1N_2)^{-2s}(N_3N_4)^{-2\sigma}N_{\max}^{-2b(2\alpha-2)}N^2N^3N^4 (N^2)^{2s-2b}\big((N^2)^{2s-2b}\vee (N^3)^{2s-2b}\big)\\
&\les \sum_{N_1,N_2,N_3,N_4}(N^1N^2)^{-2\sigma}(N^3N^4)^{-2s}N_{1}^{-2b(2\alpha-2)}N^2N^3N^4 (N^2)^{2s-2b}\big((N^2)^{2s-2b}\vee (N^3)^{2s-2b}\big)
\end{align*}
Recalling that $N^1 \sim N^2$, this is summable provided that $4b\al +4\s >3$ and 
\begin{align}
\begin{cases}
4b\al+4\s > 1+4s, &\quad \text{if} \quad b\leq s, \\
4b\al+4\s> 1+2s+2b, & \quad \text{if} \quad b> s.
\end{cases}
\label{JQQnumerology}
\end{align}
Under the hypothesis $\frac 14 < s \le \frac 12$, and $b < \tfrac{1}{2}$, the conditions \eqref{JQQnumerology} are implied under the condition $4b\al +4\s >3$.
This shows \eqref{JQQintegrability}.

 We obtain \eqref{JQQconvergence} by first showing the analogous of \eqref{Jconvergence} and then arguing exactly as in Lemma \ref{xiconvergence}.
\end{proof}

Our final energy estimate is proved deterministically and handles the remaining cases when there is exactly one input of $P_{N}S(t)Y$ or no such inputs.

\begin{lemma}\label{LEM:detenergy}
Let $\al>1$, $0 \le \sigma \le s$, $ \sigma' \geq 0$,
and suppose that $0 < b < 1$ satisfies 
\begin{equation}
 2 \alpha b + 3 \sigma - \sigma' > 1 + 2s. \label{QQquadboundcondition}
\end{equation}
Then, we have that 
\begin{equation}\label{QQquadbound}
\begin{aligned}
&\int_\R \QQ(u_1(t'),u_2(t'),u_3(t'),u_4(t')) dt' \\
&\les \big( \|u_1\|_{X^{\sigma,\frac1{q'}+}_q}\|u_2\|_{X^{\sigma,\frac1{q'}+}_q}\|u_3\|_{X^{\sigma,\frac{1}{q'}-\frac{1}{q}+}_q}\|u_4\|_{X^{-\sigma',b}_{q}} \\
&\phantom{\les\big()} +\|u_1\|_{X^{\sigma,\frac1{q'}+}_q}\|u_2\|_{X^{\sigma,\frac1{q'}+}_q}\|u_3\|_{X^{\sigma,b}_q}\|u_4\|_{X^{-\sigma',\frac{1}{q'}-\frac{1}{q}+}_{q'}}\\
&\phantom{\les\big()} +\|u_1\|_{X^{\sigma,\frac1{q'}+}_q}\|u_2\|_{X^{\sigma,b}_q}\|u_3\|_{X^{\sigma,\frac1{q'}+}_q}\|u_4\|_{X^{-\sigma',\frac{1}{q'}-\frac{1}{q}+}_{q'}} \\
&\phantom{\les\big()} +\|u_1\|_{X^{\sigma,b}_q}\|u_2\|_{X^{\sigma,\frac1{q'}+}_q}\|u_3\|_{X^{\sigma,\frac1{q'}+}_q}\|u_4\|_{X^{-\sigma',\frac{1}{q'}-\frac{1}{q}+}_{q'}}\big),
\end{aligned}
\end{equation}
\end{lemma}
\begin{proof}
Proceeding as in the proof of \eqref{quad_det_1} and \eqref{quad_det_2}, it is enough to show that 
\begin{equation*}
\sum_{\substack{n_1-n_2+n_3-n_4=0, \\ n_2 \neq n_1,n_3}} \frac{|\Psi_s(\cj{n})|\jb{n_4}^{\sigma'}}{\jb{\Phi_\alpha(\cj{n})}^{b}\jb{n_1}^{\sigma}\jb{n_2}^{\sigma}\jb{n_3}^{\sigma}}\prod_{j=1}^4 w_j(n_j) \les \prod_{j=1}^4\|w_j\|_{\l^2}.
\end{equation*}
By \eqref{Phase}, we obtain that 
\begin{align}
&\sum_{\substack{n_1-n_2+n_3-n_4=0, \\ n_2 \neq n_1,n_3}} \frac{|\Psi_s(\cj{n})|\jb{n_4}^{\sigma'}}{\jb{\Phi_\alpha(\cj{n})}^{b}\jb{n_1}^{\sigma}\jb{n_2}^{\sigma}\jb{n_3}^{\sigma}}\prod_{j=1}^4 w_j(n_j) \notag \\
&\les \sum_{\substack{n_1-n_2+n_3-n_4=0, \\ n_2 \neq n_1,n_3}} \frac{|\Psi_s(\cj{n})|\jb{n_4}^{\sigma'}\prod_{j=1}^4 |w_j(n_j)|}{\jb{n_1-n_2}^{b} \jb{n_1-n_4}^{b}\jb{n_1}^{\sigma}\jb{n_2}^{\sigma}\jb{n_3}^{\sigma}\jb{n_{\max}}^{b(2\alpha-2)}}\notag \\
&\les \sum_{N_1,N_2,N_3,N_4} N_4^{\sigma'}(N_1N_2N_3)^{-\sigma}N_{\max}^{- (2\alpha -2)b}  \sum_{\substack{n_1-n_2+n_3-n_4=0, \\ n_2 \neq n_1,n_3,\\ \jb{n_j} \sim N_j}} \frac{|\Psi_s(\cj{n})|\prod_{j=1}^4 |w_j(n_j)|}{\jb{n_1-n_2}^{b} \jb{n_1-n_4}^{b}}. \label{energyest1}
\end{align}
We split the sum over $N_1,N_2,N_3,N_4$ in \eqref{energyest1} into two cases. Notice that by symmetry between $n_1$ and $n_3$, we can assume that $N_1 \ge N_3$. Moreover, by symmetry between $n_2$ and $n_4$ in the inner sum, and since $N_4^{\sigma'}N_2^{-\sigma} \le (N_4 \vee N_2)^{\sigma'} (N_4\wedge N_2)^{-\sigma}$, we can assume that $N_4 \ge N_2$. Finally, since $n_1-n_2+n_3-n_4 = 0$, we have that $\max(N_1,N_2,N_3) \sim N_{\max}$.

\medskip
\noi
\textbf{Case 1}: $N_1 \sim \max(N_1,N_2,N_3)$. 

\medskip
\noi
\textbf{Case 1.1}: $N_2 \ll N_1$.
 
\noi
By Lemma \ref{LEM:psi} and H\"older, this contribution to \eqref{energyest1} is bounded by
\begin{align*}
& \sum_{\substack{N_1,N_2,N_3,N_4,\\ N_1 \sim N_{\max} \gg N_2}} \hspace{-0pt}N_4^{\sigma'}(N_1N_2N_3)^{-\sigma}(N_2\vee N_3)^{2s-1}N_{\max}^{- (2\alpha -2)b}\hspace{-15pt} \sum_{\substack{n_1-n_2+n_3-n_4=0, \\ n_2 \neq n_1,n_3,\\ |n_j|\sim N_j}}\hspace{-15pt} \frac{\jb{n_2-n_3}^{1-b}\prod_{j=1}^4 |w_j|(n_j)}{\jb{n_1-n_2}^{b}}\\
&\les \sum_{\substack{N_1,N_2,N_3,N_4,\\ N_1 \sim N_{\max}\gg N_2}} \hspace{-0pt}N_4^{\sigma'}(N_1N_2N_3)^{-\sigma}(N_2\vee N_3)^{2s-1}N_{\max}^{- (2\alpha -2)b} \\
&\phantom{\les \sum} \times \sup_{|n_1| \sim N_1}\Big( \sum_{|n_2|\sim N_2 \ll N_1,|n_3|\sim N_3} \frac{\jb{n_2-n_3}^{2-2b}}{\jb{n_1-n_2}^{2b}}\Big)^\frac12 \prod_{j=1}^4 \|w_j\|_{\l^2}\\
&\les \sum_{\substack{N_1,N_2,N_3,N_4,\\ N_1 \sim N_{\max}}} \hspace{-0pt}N_4^{\sigma'}(N_1N_2N_3)^{-\sigma}(N_2\vee N_3)^{2s-1}N_{\max}^{- (2\alpha -2)b}   \\
&\phantom{\les \sum\les \sum\les \sum} \times \big(N_3(N_2\vee N_3)^{2-2b}N_2N_1^{-2b}\big)^{\frac12} \prod_{j=1}^4 \|w_j\|_{\l^2}.
\end{align*}
As $\sigma \le \frac 12$, this is summable if 
\begin{equation}\label{QQquadboundcondition2}
(2\alpha - 1)b + \sigma > \sigma', \quad \text{and} \quad \quad 2\alpha b + 3 \sigma > 1 + 2s + \sigma'.
\end{equation}
Since $\sigma \le s$, $b < 1$, under condition \eqref{QQquadboundcondition}, we have that  
\begin{align*}
 (2\alpha - 1)b + \sigma - \sigma' &= 2\alpha b + 3\sigma - \sigma' - b - 2 \sigma  > 1 + 2s - b - 2\sigma\ge 1- b > 0.
\end{align*}
Therefore, \eqref{QQquadboundcondition2} holds, and we obtain \eqref{QQquadbound} in this case.

\medskip
\noi
\textbf{Case 1.2}: $N_2 \sim N_4 \sim N_1$. 

\noi
In this case, we cover the annulus $\{|n_2| \sim N_2\}$ with intervals $J_1,\dotsc, J_k$ of size $N_3$, with $k \sim N_2 N_3^{-1}$.
By Lemma \ref{LEM:psi} and $0<b<1$, we have 
\begin{align}
|\Psi_{s}(\cj{n})| \les N_{\max}^{-1+2s}\jb{n_1-n_2}^{b}\jb{n_1-n_4}^{1-b}. \label{Psisbd}
\end{align}
 Then, by \eqref{Psisbd}, H\"older, and Schur's inequality, this part of \eqref{energyest1} is bounded by
\begin{align*}
& \sum_{\substack{N_1,N_2,N_3,N_4,\\ N_1, N_2, N_4 \sim N_{\max}}} \hspace{-0pt}N_{\max}^{\sigma' - 2 \sigma - (2\alpha-2)b}N_3^{-\sigma} \sum_{\substack{n_1-n_2+n_3-n_4=0, \\ n_2 \neq n_1,n_3,\\ |n_j|\sim N_j}} \frac{|\Psi_s(\cj{n})|\prod_{j=1}^4 |w_j(n_j)|}{\jb{n_1-n_2}^{b} \jb{n_1-n_4}^{b}}\\
&\les \sum_{\substack{N_1,N_2,N_3,N_4,\\ N_1 \sim N_{\max}}} \hspace{-0pt}N_{\max}^{\sigma' - 2 \sigma - (2\alpha-2)b + 2s - 1 }N_3^{-\sigma}\hspace{-15pt} \sum_{\substack{n_1-n_2+n_3-n_4=0, \\ n_2 \neq n_1,n_3}}\hspace{-15pt} {\jb{n_2-n_3}^{1-2b}\prod_{j=1}^4 |w_j(n_j)|}\\
&\les \sum_{\substack{N_1,N_2,N_3,N_4,\\ N_1 \sim N_{\max}}} \hspace{-0pt}N_{\max}^{\sigma' - 2 \sigma - (2\alpha-2)b + 2s - 1 }N_3^{-\sigma}\\
&\phantom{\sum_{\substack{N_1,N_2,N_3,N_4,\\ N_1 \sim N_{\max}}}} \times  \sum_{k} \Big( \sup_{n_2 \in J_k} \sum_{|n_3|\sim N_3} {\jb{n_2-n_3}^{1-2b}} + \sup_{|n_3| \sim N_3} \sum_{n_2\in J_k} {\jb{n_2-n_3}^{1-2b}}\Big) \\
&\phantom{\sum_{\substack{N_1,N_2,N_3,N_4,\\ N_1 \sim N_{\max}}} \times  \sum_{k}} \times \|w_1\|_{\l^2}\|w_2\|_{\l^2(J_k)}\|w_3\|_{\l^2}\|w_4\|_{\l^2}\\
&\les \sum_{\substack{N_1,N_2,N_3,N_4,\\ N_1 \sim N_{\max}}} \hspace{-0pt}N_{\max}^{\sigma' - 2 \sigma - (2\alpha-2)b + 2s - 1 + 1 -2b }N_3^{1-\sigma}\sum_{k} \|w_1\|_{\l^2}\|w_2\|_{\l^2(J_k)}\|w_3\|_{\l^2}\|w_4\|_{\l^2}\\
&\les \sum_{\substack{N_1,N_2,N_3,N_4,\\ N_1 \sim N_{\max}}} \hspace{-0pt}N_{\max}^{\sigma' - 2 \sigma - (2\alpha-2)b + 2s  -2b}N_3^{1-\sigma}(N_2N_3^{-1})^\frac12\prod_{j=1}^4 \|w_j\|_{\l^2}.
\end{align*}
Since $\sigma \le s \le \frac 12$, it is easy to check that this is summable if $2\alpha b + 3 \sigma > 1 + 2s + \sigma', $
so in particular whenever \eqref{QQquadboundcondition} holds.

\medskip
\noi
\textbf{Case 2}: $N_2 \sim \max(N_1,N_2,N_3)$. 

\noi
In view of Case 1.2 above, we may further assume that $N_2\sim N_4\gg N_1 \geq N_3$. By Lemma \ref{LEM:psi} and H\"older, we bound this portion of \eqref{energyest1} by
\begin{align*}
&\sum_{\substack{N_1,N_2,N_3,N_4,\\ N_2 \sim N_{\max}}} N_4^{\sigma'}(N_1N_2N_3)^{-\sigma} N_{\max}^{- (2\alpha -2)b}N_{\max}^{2s-2b}\hspace{-15pt} \sum_{\substack{n_1-n_2+n_3-n_4=0, \\ n_2 \neq n_1,n_3}}  \prod_{j=1}^4 |w_j(n_j)|\\
&\les \sum_{\substack{N_1,N_2,N_3,N_4,\\ N_1 \sim N_{\max}}} \hspace{-0pt}  N_{\max}^{\s'-\s-(2\al-2)b+2s-2b} N_{1}^{\frac{1}{2}-\s} N_{3}^{\frac{1}{2}-\s} \prod_{j=1}^{4} \|w_j\|_{\l^2}.
\end{align*}
Since $\sigma \leq \frac 12$, this is again summable  if \eqref{QQquadboundcondition2} holds, and we already saw that this is satisfied under \eqref{QQquadboundcondition}.
\end{proof}

\subsection{On the uniform $L^p$ estimates for the transported density}\label{SEC:Growth} 

In this subsection, we prove the uniform $L^p$ integrability of the densities $f_{t}^{N}$. The first result below is a uniform bound on the density for at least a short time.

\begin{lemma}\label{LEM:densityontau}
For $\alpha > 1$, $\tfrac 14 < s \le \tfrac 12$, such that $\al+s>\tfrac{3}{2}$, define 
\begin{equation}\label{ta*def}
\ta_*(\alpha,s) = \begin{cases}
\frac{4\alpha + 2s - 3}{4\alpha}+ &\text{ if }\quad \alpha \le \frac{5}{4} + \frac{s}{2},\vspace{4pt} \\
\frac{1+2s}{2\alpha}+ &\text{ if }\quad \alpha > \frac{5}{4} + \frac{s}{2},
\end{cases}
\end{equation}
and
\begin{equation}\label{casdef}
c(\alpha,s,\gamma) = \begin{cases}
4 + \frac{(\gamma-1)(5+2s-4\alpha)}{2s}+ &\text{ if }\quad \alpha \le \frac{5}{4} + \frac{s}{2}, \\
4 &\text{ if }\quad \alpha > \frac{5}{4} + \frac{s}{2}.
\end{cases}
\end{equation}
Let $M \ge 1$ be a dyadic number, $0<\tau_* \le M^{-\frac{4\alpha}{2\alpha-1}-}$, $\lambda > 0$, and $\frac{2}{1-2s} > 2\gamma > c(\alpha,s,\gamma)$. Then, we have that
\begin{align}
\begin{split}
&\log\Big( \int \exp\Big(\lambda \int_0^{\tau_*} \QQ(P_N \Phi_{t'}^N(u_0)) d t'  \Big) \ind_{\{\tau_{N}(u_0)^{-1} \le M^{\frac{4\alpha}{2\alpha-1}-}\}} d \rho_{s,\gamma}(u_0)\Big) \\
&\hphantom{XXXXXX} \les  \lambda + (\lambda\tau_*^{1-\ta_*})^{\frac{2\gamma}{2\gamma- c(\alpha,s,\gamma)}} + \lambda\tau_*^{1-\ta_*(\alpha,s)} M^{c(\alpha,s,\gamma)}.
\end{split}\label{QQintegrability}
\end{align}
\end{lemma}
\begin{proof}
By the Bou\'e-Dupuis formula \eqref{P4},\footnote{See Remark \ref{RMK:LpbdN} for the integrability hypothesis on $\int \QQ$.} we have that 
\begin{align*}
&~\log\Big( \int \exp\Big(\lambda \int_0^{\tau_*} \QQ(P_N \Phi_{t'}^N(u_0)) d t'  \Big) \ind_{\{\tau_{N}(u_0)^{-1} \le M^{\frac{4\alpha}{2\alpha-1}-}\}} d \rho_{s,\gamma}(u_0)\Big) \\\
\le&~ \E\Big[ \sup_{V_0 \in H^s} \Big\{ \ind_{\{\tau_{N}(Y+V_0)^{-1} \le M^{-\frac{4\alpha}{2\alpha-1}-}\}}  \lambda \int_0^{\tau_*} \QQ(P_N \Phi_{t'}^N(Y + V_0)) d t'  \\
& \hphantom{XXXXXXXXXXXXXXXXXXXX}-  \Big|\int \wick{|Y + V_0|^2}\Big|^\gamma - \frac 12 \| V_0\|_{H^s}^2 \Big\}\Big]. \stepcounter{equation} \tag{\theequation} \label{bd3}
\end{align*}
We can write $\Phi^{N}_{t'}(Y + V_0) = S(t') Y + S(t') V_0 + \VV_N(Y+V_0)(t')$, where $\VV_N(Y+V_0)(t')$ solves \eqref{4NLSshiftedtruncated2} with initial data given by $Y + V_0$. By definition of $\tau_{N}$ (see \eqref{taudef}), we have that on the set $\{ \tau_{N}(Y+V_0)^{-1} \le M^{-\frac{4\alpha}{2\alpha-1}-}\}$, for $q = \frac{4\alpha}{3 -2s} - $,
$$\| \VV_N(Y+V_0) \|_{X^{0,\frac1{q'}+}_q(M^{-\frac{4\alpha}{2\alpha-1}-})} \les M, \quad \text{and} \quad \quad \| \VV_N(Y+V_0) \|_{X^{s,\frac1{q'}+}_q(M^{-\frac{4\alpha}{2\alpha-1}-})} \les M^\gamma. $$
By interpolation, we obtain that for $0 \le \sigma \le s$, 
\begin{equation}\label{vboundontau}
\| \VV_N(Y+V_0) \|_{X^{\sigma,\frac1{q'}+}_q(M^{-\frac{4\alpha}{2\alpha-1}-})} \les M^{ 1 +\frac{(\gamma-1)\sigma}{s}}.
\end{equation}
Let $M_{V_0} := \|V_0\|_{L^2} + \|V_0\|_{H^s}^{\frac1\gamma}$, so that we obtain the analogous of \eqref{vboundontau},
\begin{equation}\label{V0boundontau}
\| S(t)V_0 \|_{X^{\sigma,\frac1{q'}+}_q(M^{-\frac{4\alpha}{2\alpha-1}-})}\les M_{V_0}^{ 1 +\frac{(\gamma-1)\sigma}{s}}.
\end{equation}
In this way, we obtain that 
\begin{equation*}
\| S(t)V_0 + \VV_N(Y+V_0) \|_{X^{\sigma,\frac1{q'}+}_q(M^{-\frac{4\alpha}{2\alpha-1}-})} \les (M+M_{V_0})^{ 1 +\frac{(\gamma-1)\sigma}{s}}
\end{equation*}
We define the following regularity parameters:
\begin{align*}
\sigma_2 &= 
\begin{cases}
\frac{3-2\alpha}{4}+ & \text{if } \alpha \le \frac32, \\
0 & \text{if } \alpha > \frac32,
\end{cases} &
b_2 &= 
\begin{cases}
\frac12- & \text{if } \alpha \le \frac32, \\
\frac{3}{4\alpha}+ & \text{if } \alpha > \frac32,
\end{cases} 
\\
\sigma_1 &= 
\begin{cases}
\frac{3 - 2\alpha}{3}+& \text{if } \alpha \le \frac32, \\
0 & \text{if } \alpha > \frac32,
\end{cases}
&
b_1 &= 
\begin{cases}
\frac{4\alpha + 2s -3}{4\alpha}-  & \text{if } \alpha \le \frac32, \\
\frac{3+2s}{4\alpha} + & \text{if } \alpha > \frac32,
\end{cases} 
\\
\sigma_0 &= 
\begin{cases}
\frac{5 + 2s - 4\alpha}{6}+& \text{if } \alpha \le \frac54 +\frac s2, \\
0 & \text{if } \alpha > \frac54 +\frac s2,
\end{cases}
&
b_0 &= 
\begin{cases}
\frac{4\alpha + 2s -3}{4\alpha}- & \text{if } \alpha \le \frac54 +\frac s2, \\
\frac{1+2s}{2\alpha} + & \text{if } \alpha > \frac54 +\frac s2.
\end{cases}
\end{align*}
By \eqref{Y3def}, \eqref{sigma3def}, \eqref{JQQN12def}, \eqref{JQQN13def}, and \eqref{QQquadbound} with \eqref{QQquadboundcondition}, \eqref{vboundontau} and \eqref{V0boundontau}, using the symmetry 
$$\QQ(u_1,u_2,u_3,u_4) = \QQ(u_3,u_2,u_1,u_4)=\QQ(u_1,u_4,u_3,u_2) = \QQ(u_3,u_4,u_1,u_2),$$
and by \eqref{timeloc}, we obtain that 
\begin{align*}
&\ind_{\{\tau_{N}(Y+V_0)^{-1} \le M^{-\frac{4\alpha}{2\alpha-1}-}\}} \int_0^{\tau_*} \QQ(P_N \Phi_{t'}^N(Y + V_0)) d t' \\
&\les \ind_{\{\tau_{N}(Y+V_0)^{-1} \le M^{-\frac{4\alpha}{2\alpha-1}-}\}} \\
&\phantom{\les\ } \times \Big(\int_0^{\tau_*} \QQ(P_N S(t')Y ) d t' \\
&\phantom{\les\  \times \Big()}+ \|Y^{(3),N}(t)\|_{X^{-\sigma_3,-\frac1{q'}-}_{q'}(\tau_*)} \| S(t)V_0 + \VV_N(Y+V_0) \|_{X^{\sigma_3,\frac1{q'}+}_q(\tau_*)}\\
&\phantom{\les\  \times \Big()}+ (\|J_{12}^\QQ\|_{2,b_2,q,\sigma_2,\sigma_2} + \|J_{13}^\QQ\|_{2,b_2,q,\sigma_2,\sigma_2} )\\
&\phantom{\les\  \times \Big()+\ }\times \| S(t)V_0 + \VV_N(Y+V_0) \|_{X^{\sigma_2,b_2}_q(\tau_*)}\| \chi_{\tau_*}(S(t)V_0 + \VV_N(Y+V_0)) \|_{X^{\sigma_2,\frac{1}{q'}-\frac1 q +}_q(\tau_*)}\\
&\phantom{\les\  \times \Big()}+ \|\chi_{\tau_*} S(t)Y\|_{X^{s-\frac{1}{2} -,b_1}_q}\| S(t)V_0 + \VV_N(Y+V_0) \|_{X^{\sigma_1,\frac1{q'}+}_q(\tau_*)}^2\\
& \hphantom{\les\  \times \Big()XX} \times  \| S(t)V_0 + \VV_N(Y+V_0) \|_{X^{\sigma_1,\frac1{q'}-\frac1q+}_{q'}(\tau_*)}\\
&\phantom{\les\  \times \Big()}+ \|\chi_{\tau_*} S(t)Y\|_{X^{s-\frac12 -,\frac1{q'}-\frac1q+}_{q'}}\| S(t)V_0 + \VV_N(Y+V_0) \|_{X^{\sigma_1,\frac1{q'}+}_q(\tau_*)}^2 \\
&\hphantom{\les\  \times \Big()XX} \times \| S(t)V_0 + \VV_N(Y+V_0) \|_{X^{\sigma_1,b_1}_q(\tau_*)}\\
&\phantom{\les\  \times \Big()}+ \|\chi_{\tau_*} (S(t)V_0 + \VV_N(Y+V_0))\|_{X^{0,b_0}_q}\|\chi_{\tau_*} (S(t)V_0 + \VV_N(Y+V_0))\|_{X^{\sigma_0,\frac1{q'}+}_q(\tau_*)}^2\\
&\phantom{\les\  \times \Big()+}\ \times \|\chi_{\tau_*} (S(t)V_0 + \VV_N(Y+V_0))\|_{X^{\sigma_0,\frac1{q'}-\frac1q+}_{q'}(\tau_*)}\\
&\phantom{\les\  \times \Big()}+ \|\chi_{\tau_*} (S(t)V_0 + \VV_N(Y+V_0))\|_{X^{0,\frac{1}{q'}-\frac1q+}_q}\|\chi_{\tau_*} (S(t)V_0 + \VV_N(Y+V_0))\|_{X^{\sigma_0,b_0}_q(\tau_*)}^2\\
&\phantom{\les\  \times \Big()+}\ \times \|\chi_{\tau_*} (S(t)V_0 + \VV_N(Y+V_0))\|_{X^{\sigma_0,\frac1{q'}-\frac1q+}_q(\tau_*)}\\
&\les \ind_{\{\tau_{N}(Y+V_0)^{-1} \le M^{-\frac{4\alpha}{2\alpha-1}-}\}} \\
&\phantom{\les\ } \times \Big(\int_0^{\tau_*} \QQ(P_N S(t')Y ) d t' \\
&\phantom{\les\  \times \Big()}+ (\tau_*^{-\frac12}\|Y^{(3),N}(t)\|_{X^{-\sigma_3,-\frac1{q'}-}_{q'}(\tau_*)}) \tau_*^\frac12(M+M_{V_0})^{ 1 +\frac{(\gamma-1)\sigma_3}{s}}\\
&\phantom{\les\  \times \Big()}+ (\|J_{12}^\QQ\|_{2,b_2,q,\sigma_2,\sigma_2} + \|J_{13}^\QQ\|_{2,b_2,q,\sigma_2,\sigma_2} )\tau_*^{1-b_2-}(M+M_{V_0})^{ 2(1 +\frac{(\gamma-1)\sigma_2}{s})} \\
&\phantom{\les\  \times \Big()}+ \tau_*^{1-b_1-} \|Y\|_{H^{s-\frac12-}} (M+M_{V_0})^{ 3(1 +\frac{(\gamma-1)\sigma_1}{s})}\\
&\phantom{\les\  \times \Big()}+ \tau_*^{1-b_0-}  (M+M_{V_0})^{ 1 + 3(1 +\frac{(\gamma-1)\sigma_0}{s})}\Big).
\end{align*}
Notice that $ 1 + 3(1 +\frac{(\gamma-1)\sigma_0}{s}) = c(\alpha,s,\gamma)$. Moreover, recalling that  $\gamma < \frac{1}{1-2s}$, it is straightforward to check that 
\begin{align*}
\tfrac12 c(\alpha,s,\gamma) &>  (1-b_0) (1 +\tfrac{(\gamma-1)\sigma_3}{s}), \\
 c(\alpha,s,\gamma) & >2(1 +\tfrac{(\gamma-1)\sigma_2}{s}), \\
(1 - b_1)  c(\alpha,s,\gamma) & >(1-b_0)3(1 +\tfrac{(\gamma-1)\sigma_1}{s}).
\end{align*}
Note that in the limit $\gamma = \frac{1}{1-2s}$, and the case $\alpha < \frac 54 + \frac s2$, the last inequality becomes an equality.
As $\tau_* \le 1$, and noticing that $b_0 + = \ta_*$, by Young's inequality we obtain for some constant $C = C(\alpha,s,\gamma)$, 
\begin{align*}
&\ind_{\{\tau_{N}(Y+V_0)^{-1} \le M^{-\frac{4\alpha}{2\alpha-1}-}\}} \int_0^{\tau_*} \QQ(P_N \Phi_{t'}^N(Y + V_0)) d t' \\
&\les \ind_{\{\tau_{N}(Y+V_0)^{-1} \le M^{-\frac{4\alpha}{2\alpha-1}-}\}} \\
&\phantom{\les\ } \times \Big(\int_0^{\tau_*} \QQ(P_N S(t')Y ) d t' + \big(1 + (\tau_*^{-\frac12}\|Y^{(3),N}(t)\|_{X^{-\sigma_3,-\frac1{q'}-}_{q'}(\tau_{\ast})}) + (\|J_{12}^\QQ\|_{b_2,\sigma_2} + \|J_{13}^\QQ\|_{b_2,\sigma_2})\big)^C \\
&\phantom{\les\  \times \Big()} +\|Y\|_{H^{s-\frac{1}{2}-}}^{C}+ \tau_*^{1-\ta_*}  (M+M_{V_0})^{c(\alpha,s,\gamma) }\Big).
\end{align*}
Therefore, by inserting this estimate into \eqref{bd3}, H\"older's inequality,\eqref{l2integrability}, \eqref{xiintegrability}, \eqref{QQu0estimate}, \eqref{Y3estimate} and \eqref{JQQintegrability}, we obtain for some constant $\eps_0 > 0$,
\begin{align*}
&\log\Big( \int \exp\Big(\lambda \int_0^{\tau_*} \QQ(P_N \Phi_{t'}^N(u_0)) d t'  \Big) \ind_{\{\tau(u_0)^{-1} \le M^{-\frac{4\alpha}{2\alpha-1}-}\}} d \rho_{s,\gamma}(u_0)\Big) \\
&\les \lambda \tau_*^\frac12 + \lambda + \lambda \tau_*^{1-\ta_*} M^{c(\alpha,s,\gamma)} \\
&\phantom{\les\ } + \E \big[\sup_{V_0 \in H^s} \big\{ \lambda \tau_*^{1-\ta_*} M_{V_0}^{c(\alpha,s,\gamma)} - \eps_0 M_{V_0}^{2\gamma} \big\}\big]\\
&\les \lambda + \lambda \tau_*^{1-\ta_*} M^{c(\alpha,s,\gamma)} + (\lambda \tau_*^{1-\ta_*})^{\frac{2\gamma}{2\gamma - c(\alpha,s,\gamma)}},
\end{align*}
which completes the proof of Lemma~\ref{LEM:densityontau}.
\end{proof}

\begin{remark}\rm \label{RMK:reasonforq}
It is in the proof of Lemma~\ref{LEM:densityontau} that we see the benefit of allowing $q>2$ in the local-well-posedness result (Proposition~\ref{PROP:LWP}). Specifically, in obtaining \eqref{bd3}, we applied Lemma~\ref{LEM:detenergy} where one of the terms here has twisted temporal regularity $b<1$. In order to use \eqref{vboundontau}, we need $b<\tfrac{1}{q'}$. As $q>2$, we have $\tfrac{1}{2}<\tfrac{1}{q'}$.
This extra room in choosing $b$ allows us to keep $\mathcal{V}_{N}(Y+V_0)$ in $L^2_{x}$ for smaller values of $\alpha$ compared to if we had fixed $q=2$.
\end{remark}

\begin{lemma}\label{LEM:ftNSM}
Let $M\geq 1$ be a dyadic number, and for $t \ge 0$, consider the sets 
$$ S_t^M := \Big\{ u_0 : \sup_{0 \le t' \le t} \big[ \tau_{N}(\Phi^{N}_{-t'}(u_0))^{-1} \big] \le M^{\frac{4\alpha}{2\alpha-1}+}\Big\},  $$
where $\tau_{N}$ is the stopping time defined in \eqref{taudef}.
For every $1 \le p < \infty$, and for every $t \ge 0$, we have that $f^{N}_t$ defined in \eqref{densityformula} satisfies $f^{N}_t\ind_{S_t^M} \in L^p(\rho_{s,\gamma})$. More precisely, we have that
\begin{equation}\label{densitybound1}
\begin{aligned}
\log \| f^{N}_t \ind_{S_t^M} \|_{L^p(\rho_{s,\gamma})}\les p^{\frac{c(\alpha,s,\gamma)}{2\gamma - c(\alpha,s,\gamma)}}\jb{t}^{\frac{2\gamma}{2\gamma - c(\alpha,s,\gamma)}} M^{\frac{4\alpha\ta_*}{2\alpha-1}\cdot\frac{2\gamma}{2\gamma - c(\alpha,s,\gamma)}+} + \jb{t} M^{\frac{4\alpha\ta_*}{2\alpha-1}+ c(\alpha,s,\gamma)+}.
\end{aligned}
\end{equation}
where the implicit constant is independent of $N$.
\end{lemma}
\begin{proof}
By remark~\ref{RMK:LpbdN}, we know that $\|f_t^{N}\|_{L^p(\rho_{s,\g})}$ is finite for each fixed $N$ and $t\geq 0$.

We first note that for $0 \le t' \le t$, we have the following inclusion:
\begin{align}
\begin{split}
\big\{\Phi_t^N(u_0) \in S_t^M\big\} &= \Big\{\sup_{0 \le \wt{t} \le t} \tau_{N}(\Phi^{N}_{\wt{t}}(u_0))^{-1} \le M^{\frac{4\alpha}{2\alpha-1}+}\Big\} \\
&\subseteq\Big\{\sup_{0 \le \wt{t} \le t'} \tau_{N}(\Phi^{N}_{\wt{t}}(u_0))^{-1} \le M^{\frac{4\alpha}{2\alpha-1}+} \Big\} = \big\{\Phi_{t'}^N(u_0) \in S_{t'}^M\big\}.
\end{split} \label{SMinclusions}
\end{align}
Let $k = k(t) \in \N$ be the smallest positive integer such that $ \frac {t}{k} \le M^{-\frac{4\alpha}{2\alpha-1}-}$. 
By \eqref{densityformula}, Jensen's inequality and \eqref{SMinclusions}, we have that 
\begin{align*}
&\int f_t^{N}(u_0)^{p} \ind_{S_t^M}(u_0) d \rho_{s,\gamma}(u_0) \\
&= \int f_t^{N}(u_0)^{p-1} \ind_{S_t^M}(u_0) f^{N}_t(u_0) d \rho_{s,\gamma}(u_0) \\
&= \int f_t^{N}(\Phi^{N}_t(u_0))^{p-1} \ind_{S_t^M}(\Phi_t^N(u_0)) d \rho_{s,\gamma}(u_0) \\
&= \int \exp\bigg((p-1) \int_0^t \QQ(P_N\Phi_{t'}^N(u_0)) dt' \bigg)\ind_{S_t^M}(\Phi_t(u_0)) d \rho_{s,\gamma}(u_0) \\
&= \int \exp\bigg((p-1) \sum_{h=0}^{k-1} \int_{\frac{h}{k} t}^{\frac{h+1}{k} t} \QQ(P_N\Phi_{t'}^N(u_0)) dt' \bigg)\ind_{S_t^M}(\Phi_t^N(u_0)) d \rho_{s,\gamma}(u_0) \\
&\le \int \frac{1}{k}\sum_{h=0}^{k-1} \exp\bigg((p-1) k \int_{\frac{h}{k} t}^{\frac{h+1}{k} t}\QQ(P_N\Phi_{t'}^N(u_0)) dt' \bigg)\ind_{S_t^M}(\Phi_t^N(u_0)) d \rho_{s,\gamma}(u_0) \\
&= \int \frac{1}{k}\sum_{h=0}^{k-1} \exp\bigg((p-1) k \int_{0}^{\frac{t}{k}}\QQ(P_N\Phi_{\frac{h}{k} t + t'}^N(u_0)) dt' \bigg)\ind_{S_t^M}(\Phi_t^N(u_0)) d \rho_{s,\gamma}(u_0) \\
&\le \frac{1}{k}\sum_{h=0}^{k-1} \int  \exp\bigg((p-1) k \int_{0}^{\frac{t}{k}}\QQ(P_N\Phi_{\frac{h}{k} t + t'}^N(u_0)) dt' \bigg)\ind_{S_{\frac hk t}^M}(\Phi_{\frac hk t}^N(u_0)) d \rho_{s,\gamma}(u_0) \\
&= \frac{1}{k}\sum_{h=0}^{k-1} \int  \exp\bigg((p-1) k \int_{0}^{\frac{t}{k}}\QQ(P_N\Phi_{t'}^N(u_0)) dt' \bigg)\ind_{S_{\frac hk t}^M}(u_0) f^{N}_{\frac hk t}(u_0) d \rho_{s,\gamma}(u_0) \\
&\le \frac{1}{k}\sum_{h=0}^{k-1}\bigg[ \int  \exp\Big(p'(p-1) k \int_{0}^{\frac{t}{k}}\QQ(P_N\Phi_{t'}^N(u_0)) dt' \Big)\ind_{S_{\frac hk t}^M}(u_0) d \rho_{s,\gamma}(u_0) \bigg]^\frac{1}{p'} \\
& \hphantom{\frac{1}{k}\sum_{h=0}^{k-1}\bigg[ \int  \exp\Big(p'(p-1) k }  \times\bigg[\int f^{N}_{\frac hk t}(u_0)^p \ind_{S_{\frac hk t}^M}(u_0) d \rho_{s,\gamma}(u_0) \bigg]^\frac1p\\
&\le \frac{1}{k}\sum_{h=0}^{k-1}\bigg[ \int  \exp\bigg(p'(p-1) k \int_{0}^{\frac{t}{k}}\QQ(P_N\Phi_{t'}^N(u_0)) dt' \bigg)\ind_{S_{0}^M}(u_0) d \rho_{s,\gamma}(u_0) \bigg]^\frac{1}{p'} \\
& \hphantom{\frac{1}{k}\sum_{h=0}^{k-1}\bigg[ \int  \exp\Big(p'(p-1) k }  \times\bigg[\int f^{N}_{\frac hk t}(u_0)^p \ind_{S_{\frac hk t}^M}(u_0) d \rho_{s,\gamma}(u_0) \bigg]^\frac1p\\
&\le \bigg[ \int  \exp\Big(p k \int_{0}^{\frac{t}{k}}\QQ(P_N\Phi_{t'}^N(u_0)) dt' \Big)\ind_{S_{0}^M}(u_0) d \rho_{s,\gamma}(u_0) \bigg]^\frac{1}{p'} \\
&\hphantom{\frac{1}{k}\sum_{h=0}^{k-1}\bigg[ \int  \exp\Big(p'(p-1) k } \times \bigg[\sup_{0 \le t' \le t} \int f^{N}_{t'}(u_0)^p \ind_{S^{M}_{t'}}(u_0) d \rho_{s,\gamma}(u_0)\bigg]^\frac{1}{p}, \stepcounter{equation} \tag{\theequation} \label{togronw}
\end{align*}

Therefore, by taking the supremum in \eqref{togronw} for $0\leq t' \le t$, \eqref{togronw} implies
\begin{align}
\begin{split}
&~\sup_{0\leq t'\leq t}\int f^{N}_{t'}(u_0)^p \ind_{S^{M}_{t'}}(u_0) d \rho_{s,\gamma}(u_0) \\
&\le \sup_{0\le t' \le t} \int \exp\bigg(p k(t') \int_{0}^{\frac{t'}{k(t')}}\QQ(P_N\Phi_{\wt{t}}^N(u_0)) d\wt{t} \bigg)\ind_{S_{0}^M}(u_0) d \rho_{s,\gamma}(u_0).
\end{split} \label{densitybd1}
\end{align}
Therefore, by \eqref{QQintegrability}, we obtain that 
\begin{align*}
&~\log \| f^{N}_t \ind_{S_t^M} \|_{L^p} \\
&= p^{-1} \log \| f^{N}_t \ind_{S_t^M} \|_{L^p}^p \\
&\les p^{-1} \sup_{0\le t' \le t}\Big[ pk(s) +\Big(pk(t') \big(\tfrac{t'}{k(t')}\big)^{1-\ta_*}\Big)^{\frac{2\gamma}{2\gamma - c(\alpha,s,\gamma)}} + pk(t') \big(\tfrac{t'}{k(t')}\big)^{1-\ta_*} M^{c(\alpha,s,\gamma)}\Big]\\
&\les \jb{tM^{\frac{4\alpha}{2\alpha-1}+}} + p^{\frac{c(\alpha,s,\gamma)}{2\gamma - c(\alpha,s,\gamma)}} \big(t^{1-\ta_*} \jb{tM^{\frac{4\alpha}{2\alpha-1}+}}^{\ta_*}\big)^{\frac{2\gamma}{2\gamma - c(\alpha,s,\gamma)}}
 + t^{1-\ta_*} \jb{tM^{\frac{4\alpha}{2\alpha-1}+}}^{\ta_*}M^{c(\alpha,s,\gamma)} \\
&\les p^{\frac{c(\alpha,s,\gamma)}{2\gamma - c(\alpha,s,\gamma)}}\jb{t}^{\frac{2\gamma}{2\gamma - c(\alpha,s,\gamma)}} M^{\frac{4\alpha\ta_*}{2\alpha-1}\cdot\frac{2\gamma}{2\gamma - c(\alpha,s,\gamma)}+} + \jb{t} M^{\frac{4\alpha\ta_*}{2\alpha-1}+ c(\alpha,s,\gamma)+}. \qedhere
\end{align*}
\end{proof}

Finally, we remove the localizations to the sets $S_{t}^{M}$ in \eqref{densitybound1}.

\begin{proposition}\label{prop:lpbound}
Let $\al>1$, $\tfrac{1}{4}<s\leq \tfrac{1}{2}$ with $\alpha + s > \tfrac 32$, and $\gamma > 2$ be such that 
\begin{equation} \label{condition}
\gamma < \frac{1}{1-2s},  \quad \text{and}\quad 2\gamma > \frac{4\alpha\ta_*}{2\alpha-1}+c(\alpha,s,\gamma).
\end{equation}
Then, the density $f^{N}_t$, defined in \eqref{densityformula}, belongs to $L^p(\rho_{s,\gamma})$, for any $1\leq p<\infty$, uniformly in $N$. More precisely, 
\begin{equation}\label{lpdensity}
\log \| f^{N}_t \|_{L^p(\rho_{s,\g})} \les_p\jb{t}^{\frac{2\gamma}{2\gamma - c(\alpha,s,\gamma) - \frac{4\alpha\ta_*}{2\alpha-1}}+},
\end{equation}
where the implicit constant is independent of $N$.
\end{proposition}

\begin{proof}
Fix $t \ge 0$, and for $M\geq 1$ dyadic, consider the sets 
$$E_M(t) := \bigcap_{h = 0}^{\lfloor tM^{\frac{4\alpha}{2\alpha-1}+}\rfloor} \Big\{ \tau_{N}(\Phi_{hM^{-\frac{4\alpha}{2\alpha-1}-}}^N(u_0))^{-1} \le \eps_0 M^{\frac{4\alpha}{2\alpha-1}+} \Big\},$$
where $0<\eps_0 \ll 1$.
From the relationship
\begin{equation}
\VV_N(u_0)(t_1 + t_2) = \VV_N\big(S(t_1) u_0 + \VV_N(u_0)(t_1)\big)(t_2) +S(t_2)\big(\VV_N(u_0)(t_1)\big), \label{VViteration}
\end{equation}
it is easy to check that for $\eps_0$ small enough, we have that $E_M(t) \subseteq \{\Phi^{N}_t(u_0) \in S_t^{M}\}$. Moreover, by Proposition \ref{4NLStruncGWP}, together with the estimate \eqref{tau0Nbound}, we have that for fixed $N$, $\rho_{s,\gamma}(E_M(t)) \to \|\rho_{s,\gamma}\|_{\text{TV}}$ as $M \to \infty$. Note that we did not normalize $\rho_{s,\g}$ to be a probability measure, so $\|\rho_{s,\g}\|_{\text{TV}}=\int 1 d\rho_{s,\g}(u_0)$. Therefore, for $1<p<\infty$, by \eqref{densitybound1}, we obtain
\begin{align*}
\| f_t^{N} \|_{L^p}^p &= \| f_t^{N} \circ \Phi_t \|_{L^{p-1}}^{p-1}\\
&\leq \| (f_t^{N} \circ \Phi_{t}^{N}) \ind_{E_1(t)}\|_{L^{p-1}}^{p-1} + \sum_{M \ge 1} \| (f_t^{N} \circ \Phi_{t}^{N})\ind_{E_{2M}(t) \setminus E_M(t)} \|_{L^{p-1}}^{p-1} \\
&\le \| (f_t^{N} \circ \Phi_{t}^{N}) (\ind_{S_t^1}\circ \Phi_{t}^{N}) \|_{L^{p-1}}^{p-1}+ \sum_{M \ge 1} \| (f_t^{N} \circ \Phi_{t}^{N}) (\ind_{S_t^{2M}}\circ \Phi_{t}^{N})\ind_{ E_M(t)^c} \|_{L^{p-1}}^{p-1}\\
&\le \| f_t^{N} \ind_{S_t^1}\|_{L^{p}}^{p} + \sum_{M \ge 1} \| (f_t^{N} \circ \Phi_{t}^{N}) (\ind_{S_t^M}\circ \Phi_{t}^{N})\|_{L^{2p-1}}^{p-1} \rho_{s,\gamma}(\{\Phi_{t}^{N} \in S_t^{2M}\} \cap E_M(t)^c)^{\frac p {2p-1}}\\
&\le \| f_t^{N} \ind_{S_t^1}\|_{L^{p}}^{p} + \sum_{M \ge 1} \| f_t^{N} \ind_{S_t^{2M}}\|_{L^{2p}}^{p - \frac{p}{2p-1}} \rho_{s,\gamma}(\{\Phi_{t}^{N} \in S_t^{2M}\} \cap E_M(t)^c)^{\frac p {2p-1}}\stepcounter{equation} \tag{\theequation} \label{ftlp1}
\end{align*}
By \eqref{SMinclusions}, \eqref{densitybound1} and \eqref{tauintegrability}, we have that for some constants $C_1, C_2, C_1' > 0$,
\begin{align*}
&~\rho_{s,\gamma}(\{\Phi_{t}^{N}(u_0) \in S_t^{2M}\} \cap E_M(t)^c) \\
& \le \sum_{h=0}^{\lfloor tM^{\frac{4\alpha}{2\alpha-1}+}\rfloor} \hspace{-0.1cm} \rho_{s,\gamma}\Big(\Big\{\Phi_{t}^{N}(u_0) \in S_t^{2M}\Big\} \cap  \Big\{ \tau_{N}\big(\Phi_{hM^{-\frac{4\alpha}{2\alpha-1}-}}^N(u_0)\big)^{-1} > \eps_0 M^{\frac{4\alpha}{2\alpha-1}+} \Big\}\Big) \\
&\le \sum_{h=0}^{\lfloor tM^{\frac{4\alpha}{2\alpha-1}+}\rfloor}  \hspace{-0.1cm} \rho_{s,\gamma}\Big( \Big\{ \Phi_{hM^{-\frac{4\alpha}{2\alpha-1}-}}^N(u_0) \in S_{hM^{-\frac{4\alpha}{2\alpha-1}-}}^{2M}\Big\} \cap \Big\{ \tau_{N}(\Phi_{hM^{-\frac{4\alpha}{2\alpha-1}-}}^N(u_0))^{-1} > \eps_0 M^{\frac{4\alpha}{2\alpha-1}+} \Big\}\Big) \\
&=  \sum_{h=0}^{\lfloor tM^{\frac{4\alpha}{2\alpha-1}+}\rfloor} \hspace{-0.1cm} \int \ind_{S_{hM^{-\frac{4\alpha}{2\alpha-1}-}}^{2M}}(\Phi_{hM^{-\frac{4\alpha}{2\alpha-1}-}}^N(u_0)) \ind_{\{\tau_{N}(u_0)^{-1} > \eps_0M^{-\frac{4\alpha}{2\alpha-1} - \eps}\}}(\Phi_{hM^{-\frac{4\alpha}{2\alpha-1}-}}^N(u_0)) d\rho_{s,\gamma}(u_0) \\
&= \sum_{h=0}^{\lfloor tM^{\frac{4\alpha}{2\alpha-1}+}\rfloor}\hspace{-0.1cm} \int f^{N}_{{hM^{-\frac{4\alpha}{2\alpha-1}-}}} \ind_{S_{hM^{-\frac{4\alpha}{2\alpha-1}-}}^{2M}} \ind_{\{\tau_{N}(u_0)^{-1} > \eps_0M^{-\frac{4\alpha}{2\alpha-1} - \eps}\}} d\rho_{s,\gamma}(u_0)\\
&\le \sum_{h=0}^{\lfloor tM^{\frac{4\alpha}{2\alpha-1}+}\rfloor} \hspace{-0.1cm}  \Big\| f^{N}_{{hM^{-\frac{4\alpha}{2\alpha-1}-}}} \ind_{S_{hM^{-\frac{4\alpha}{2\alpha-1}-}}^{2M}} \Big\|_{L^2(\rho_{s,\gamma})} \rho_{s,\gamma}(\{\tau_{N}(u_0)^{-1} > \eps_0M^{\frac{4\alpha}{2\alpha-1}+}\})^\frac12\\
&\les \sum_{h=0}^{\lfloor tM^{\frac{4\alpha}{2\alpha-1}+}\rfloor} \exp\Big(C_1\big(\jb{t}^{\frac{2\gamma}{2\gamma - c(\alpha,s,\gamma)}} M^{\frac{4\alpha\ta_*}{2\alpha-1}\cdot\frac{2\gamma}{2\gamma - c(\alpha,s,\gamma)}+} + \jb{t} M^{\frac{4\alpha\ta_*}{2\alpha-1}+ c(\alpha,s,\gamma)+}\big) - C_2 M^{2\gamma-} \Big)\\
&\les \exp\Big(C_1'\big(\jb{t}^{\frac{2\gamma}{2\gamma - c(\alpha,s,\gamma)}} M^{\frac{4\alpha\ta_*}{2\alpha-1}\cdot\frac{2\gamma}{2\gamma - c(\alpha,s,\gamma)}+} + \jb{t} M^{\frac{4\alpha\ta_*}{2\alpha-1}+ c(\alpha,s,\gamma)+}\big) - C_2 M^{2\gamma-} \Big).
\end{align*}
Therefore, by \eqref{ftlp1} and \eqref{densitybound1}, we obtain that for some $C_p, C_p', C_p'' > 0$, 
\begin{align*}
&\| f^{N}_t \|_{L^p}^p \\
&\les \sum_{M \ge 1}\exp\Big(C_p\big(\jb{t}^{\frac{2\gamma}{2\gamma - c(\alpha,s,\gamma)}} M^{\frac{4\alpha\ta_*}{2\alpha-1}\cdot\frac{2\gamma}{2\gamma - c(\alpha,s,\gamma)}+} + \jb{t} M^{\frac{4\alpha\ta_*}{2\alpha-1}+ c(\alpha,s,\gamma)+}\big) - C_p' M^{2\gamma-} \Big)\\
&\les \exp\Big(C_p'' \jb{t}^{\frac{2\gamma}{2\gamma - c(\alpha,s,\gamma) - \frac{4\alpha\ta_*}{2\alpha-1}}+}\Big),
\end{align*}
which yields \eqref{lpdensity}.
\end{proof}

\begin{remark}\label{rk:numerology}
We can now explain the numerology in Theorem~\ref{thm:qiandgwp}.
Recalling the definitions of $\ta_{\ast}$ in \eqref{ta*def} and $c(\al,s,\g)$ from \eqref{casdef}, the condition \eqref{condition} becomes
\begin{equation}\label{condition1}
\tfrac{2}{1-2s} > 2 \gamma > 4 + \tfrac{\gamma-1}{2s}(5+2s-4\alpha) + \tfrac{4\alpha+2s-3}{2\alpha-1}
\end{equation}
if $1<\alpha \le \frac54+\frac s2$, and 
\begin{equation}\label{condition2}
\tfrac{2}{1-2s} > 2 \gamma > 4 + \tfrac{2+4s}{2\alpha-1}
\end{equation}
if $\alpha > \frac54 + \frac s2$. Recall that in both cases, we have the extra condition $\alpha + s > \frac 32$. Fixing $\gamma = \frac 1{1-2s}-$, we obtain that \eqref{condition1} is equivalent to 
\begin{equation*}
s > \frac14 \Big(\sqrt{68\alpha^2-52\alpha + 9} - 10 \alpha + 7 \Big),
\end{equation*}
while \eqref{condition2} is equivalent to 
\begin{equation*}
s> \frac12 \Big(\sqrt{4\alpha^2-2\alpha+1} - 2\alpha + 1\Big).
\end{equation*}
It it then straightforward that in both cases the condition $\alpha + s > \frac 32$ is satisfied.
By imposing the extra conditions $1<\alpha \le \frac54+\frac s2$ and $\alpha > \frac54+\frac s2$ respectively, we find
\begin{equation*}
s > \begin{cases}
\frac14 \Big(\sqrt{68\alpha^2-52\alpha + 9} - 10 \alpha + 7 \Big) & \text{ if }\quad  1< \alpha \le \frac{1}{32} \big(35 + \sqrt{105}\big), \\
\frac12 \Big(\sqrt{4\alpha^2-2\alpha+1} - 2\alpha + 1\Big) & \text{ if } \quad \alpha >  \frac{1}{32} \big(35 + \sqrt{105}\big),
\end{cases}
\end{equation*}
which corresponds exactly to \eqref{conditionalphas}.
\end{remark}

\section{Convergence of the truncated flow and quasi-invariance}

In this section, we complete the proof of Theorem~\ref{thm:qiandgwp}. Under the uniform estimate \eqref{lpdensity} for the $L^p$ norm of the density, we can apply Bourgain's (quasi-)invariant measure argument (Theorem~\ref{bima}) to show that the limiting flow is globally well-posed. A weak convergence argument then implies the quasi-invariance of the Gaussian measures $\mu_s$.

\begin{theorem}[Bourgain's quasi-invariant measure argument] \label{bima}
Let $X$ be a Polish space, and for every $t \in \R$, let 
$ \Phi_t : X \to X $
be a flow map, i.e.\ $\Phi_0(u_0) = u_0$, $\Phi_{t+s} = \Phi_t \circ \Phi_s$.
Suppose moreover that the flow map $(t,u_0) \mapsto \Phi_{t}(u_0)$ is jointly measurable in $t$ and $u_0$. Let $\rho$ be a $\sigma$-finite measure on the Borel sets of $X$, and suppose that $(\Phi_t)_\# \rho= f_t \rho$, for some $f_t\in L^1(X)$. Finally, suppose that the following properties hold:
\begin{enumerate}[\rm{(}i\rm{)}]
\item There exists a quantity $ r: X \to \R$, a non-increasing function $\tau: \R^+ \to \R^+$, and a number $K > 1$, such that 
$$ r(\Phi_s(u_0)) \le K r(u_0) \quad \text{for every } 0 \le s \le \tau(r(u_0)), $$ \label{bimai}
and $\tau$ decays sub-polynomially, i.e.\ $\tau(x) \gtrsim \jb{x}^{-\kappa}$ for some $\kappa \ge 0$.
\item There exists $\gamma > 0$ such that 
$$ \int \exp\Big(r(u_0)^\gamma\Big) d \rho(u_0) < + \infty. $$ \label{bimaii}
\item For some $p > 1$, and some nondecreasing function $c: \R^+ \to \R$,  
\begin{equation*}
\sup_{0 \le s \le t} \log \| f_s \|_{L^p(\rho)} \les c(t).
\end{equation*} \label{bimaiii}
\end{enumerate}
Then, for $\rho_0$-a.e.\ $u_0$ we have the estimate 
\begin{equation}
 r(\Phi_t(u_0))\les_{u_0} \log\big(2 + t \big)^\frac 1 \gamma + c(t)^{\frac 1 {\gamma}}. \label{gwthestimate}
\end{equation}
More precisely, for every $M \gg 1$,
\begin{equation}\label{gwthestimatestrong}
- \log \rho\Big(\Big\{r(\Phi_t(u_0))> M( \log\big(2 + t \big)^\frac 1 \gamma + c(t)^{\frac 1 {\gamma}}) \textup{ for some }t>0\Big\} \Big) \gtrsim M^\gamma.
\end{equation}
\end{theorem}
\begin{proof}
Let $g(t) =\log(2 + t) ^\frac 1 \gamma + c(t)^{\frac 1 {\gamma}}$, and let $M > 0$. Define recursively the sequence of times $\tau_n$ by $\tau_0 = 0$, 
$$ \tau_{n+1} = \tau_n + \tau(M g(\tau_n)).$$
We have, for $x \ge (\log2)^\frac1\gamma $, for some constant $c> 0$, noticing that the function $g$ is nondecreasing, we have that 
\begin{align*}
|\{n : g(\tau_n) \le x\}| &= |\{n : g(\tau_n) \le x, \tau_n \le \exp(x^\gamma) \}|\le \tfrac{\exp(x^\gamma)}{\tau(Mx)}\les M^{\kappa} \exp(c x^\gamma),
\end{align*}
and trivially, $g(t) \ge (\log2)^\frac1\gamma$ for every $t \ge 0$.
In particular, this implies that $g(\tau_n) \to \infty$ as $n \to \infty$, so we must also have $\tau_n \to \infty$ as $n \to \infty$. Notice that, by \eqref{bimai}, we have that 
\begin{align*}
\Big\{\sup_{s \ge 0} \frac{r(\Phi_s(u_0))}{K M g(s)} > 1\Big\} & \subseteq \bigcup_{n} \Big\{ r(\Phi_{\tau_n}(u_0)) > M g(\tau_n) \Big\}.
\end{align*}
Therefore, for $M \gg 1$, by H\"older's inequality, \eqref{bimaii} and \eqref{bimaiii}, for some constant $C_p > 0$, we have that
\begin{align*}
\phantom{\le\ }\rho\Big(\Big\{\sup_{s \ge 0} \frac{r(\Phi_s(u_0))}{K M g(s)} > 1\Big\} \Big) 
&\le \sum_{n=1}^{\infty} (\Phi_{\tau_n})_{\#}\rho\Big( \{ r(u_0) > M g(\tau_n) \} \Big) \\
&\le \sum_{n=1}^{\infty}\rho\Big( \{ r(u_0) > M g(\tau_n) \} \Big)^\frac 1 {p'} \| f_{\tau_n} \|_{L^p(\rho)} \\
&\le  \sum_{n=1}^\infty \exp\big( -2C_{p} M^\gamma g(\tau_n)^\gamma + c(\tau_n)  \big)\\
&\le \sum_{n=1}^\infty \exp\big( -C_{p} M^\gamma g(\tau_n)^\gamma \big)\\
&= \sum_{n=1}^\infty \int_{-\infty}^{+\infty} \Big|\Big\{n:  -C_p M^\gamma g(\tau_n)^\gamma \ge \lambda \Big\}\Big| \exp(\lambda) d\lambda \\
&= \int_{C_pM^\gamma\log 2}^{+\infty} \Big|\Big\{n: g(\tau_n) \le \Big(\frac{\lambda}{C_pM^\gamma}\Big)^\frac 1 \gamma\Big\}\Big| \exp(-\lambda) d\lambda\\
&\les  \int_{C_pM^\gamma\log 2}^{+\infty} M^\kappa \exp\Big(-\lambda \Big(1 - \frac{c}{C_pM^\gamma}\Big)\Big) d \lambda\\
&\les  \exp(-C'_{p,\kappa,\gamma,\dl} M^\gamma),
\end{align*}
which is \eqref{gwthestimatestrong}, up to replacing $M$ with $K^{-1}M$. In particular, \eqref{gwthestimate} follows.
\end{proof}

As a direct consequence, we have the following. 

\begin{corollary}
Let $\tau_{N}$ be the stopping time defined in \eqref{taudef}, and assume that $\alpha, s ,\gamma$ satisfy \eqref{condition}. Then for some exponents $A = A(\alpha,s,\gamma)$, $\beta = \beta(\alpha,s,\gamma) > 0$, for some constant $c=c(\alpha,s,\gamma)>0$, we have that
\begin{equation} \label{taugrowthestimate}
\rho_{s,\gamma} \Big(\Big\{\tau_{N}(\Phi_t^N(u_0))^{-1} \le M \jb{t}^{A} \textup{ for some }t>0\Big\} \Big) \ge \|\rho_{s,\gamma}\|_{\textup{TV}}- \exp(-c M^\beta).
\end{equation}
\end{corollary}

\begin{proof}
We apply Theorem \ref{bima} to the flow $\Phi_t^N$, measure $\rho_{s,\gamma}$, and $r = \tau_{N}^{-1}$. We just need to check that the conditions \eqref{bimai}, \eqref{bimaii} and \eqref{bimaiii} hold.
\begin{enumerate}
\item 
From \eqref{VViteration}, we have that 
\begin{align*}
\|\VV_N(\Phi^{N}_t(u_0))\|_{X^{\sigma,\frac 1q+}_{q}(\tfrac\tau2)} \le 2 \|\VV_N(u_0)\|_{X^{\sigma,\frac 1q+}_{q}(\tau)}
\end{align*}
for any $0\leq \s\leq s$.
In particular, recalling that $\gamma > 1$ and $\frac{4\alpha}{2\alpha-1} > 2$, by definition \eqref{taudef} of $\tau$, this implies that 
$$ \sup_{0 \le t' \le \tfrac {\tau(u_0)}2} \tau_{N}(\Phi_{t'}^N(u_0))^{-1} \le 2^{\frac{4\alpha}{2\alpha-1}+} \tau_{N}(u_0)^{-1},$$
which is \eqref{bimai} with $\kappa = 1$.
\item \eqref{bimaii} is a direct consequence of \eqref{tauintegrability}.
\item \eqref{bimaiii} is \eqref{lpdensity}.
\end{enumerate}
Therefore, \eqref{taugrowthestimate} is exactly \eqref{gwthestimatestrong} in this setting, with 
\begin{gather*}
A = \frac{2\gamma}{2\gamma - c(\alpha,s,\gamma) - \frac{4\alpha\ta_*}{2\alpha-1}}+,  \quad \text{and} \quad \beta = \frac{2\alpha-1}{2\alpha}\gamma-.
\end{gather*}
\end{proof}

This allows us to prove the following.

\begin{proposition}
Suppose that $\alpha, s ,\gamma$ satisfy \eqref{conditionalphas}. Let $\VV_N(u_0)$ be the solution of \eqref{4NLSshiftedtruncated2}. Then, for $\mu_s$ a.e.\ $u_0$, for every $t \ge 0$, we have that 
$$ \VV_N(\Phi_t(u_0)) \in L^{\infty}_t([0,1]; L^2_x).$$
Moreover, for some exponents $A' = A'(\alpha,s,\gamma), \beta = \beta(\alpha,s,\gamma)$, there exists a constant $c > 0$ independent of $N$, such that for $M \gg 1$
\begin{equation}\label{VVgrowth}
\rho_{s,\gamma}\Big(\|\VV_N(\Phi_t(u_0))\|_{L^{\infty}_t([0,1]; L^2_x)} \le \exp(M\jb{t}^{A'})\Big) \ge \|\rho_{s,\gamma}\|_{\textup{TV}}- \exp(-c M^\beta).
\end{equation}
\end{proposition}
\begin{proof}
By Remark \ref{rk:numerology}, we obtain that if $\alpha, s ,\gamma$ satisfy \eqref{conditionalphas}, they satisfy \eqref{condition} as well.
From the \eqref{VViteration},
we obtain that, for every $k \in \N$, 
$$ \sup_{0 \le t' \le 1} \|\VV_N(u_0)\|_{L^2} \le 2^{k} \sup_{0 \le h \le k-1} \|\VV_N\big(\Phi_{\tfrac{h}{k}}^N(u_0)\big)\|_{L_{t}^{\infty}([0,k^{-1}]; L_{x}^2)}. $$
Moreover, if $\tau_{N}(\Phi_{\frac{h}{k}}^{N}(u_0)) \ge k^{-1}$, then by definition \eqref{taudef} of $\tau_{N}$ and \eqref{ctsembed}, 
$$ \|\VV_N\big(\Phi_{\tfrac{h}{k}}^N(u_0)\big)\|_{L_{t}^{\infty}([0,k^{-1}]; L_{x}^2)} \les \tau_{N}(\Phi_{\frac{h}{k}}^{N}(u_0))^{-\frac{2\alpha-1}{4\alpha}-}.$$
Therefore, for fixed $t \ge 0$, by picking $k \sim \sup_{0 \le t' \le t+1} \tau_{N}(\Phi_{t'}(u_0))^{-1}$, we obtain that 
\begin{align*}
\|\VV_N(\Phi_t(u_0))\|_{L^{\infty}_t([0,1]; L^2_x)} &\les \exp\big(c \sup_{0 \le t' \le t+1} \tau_{N}(\Phi_{t'}(u_0))^{-1}\big) \sup_{0 \le t' \le t+1} \tau_{N}(\Phi_{t'}(u_0))^{-\frac{2\alpha-1}{4\alpha}-} \\
&\les \exp\big(c' \sup_{0 \le t' \le t+1} \tau_{N}(\Phi_{t'}(u_0))^{-1}\big)
\end{align*}
Therefore, \eqref{VVgrowth} follows from \eqref{taugrowthestimate}.
\end{proof}

\begin{proof}[Proof of Theorem \ref{thm:qiandgwp}]
We start by showing global well-posedness and quasi-invariance up to time $1$. Let $N_0(u_0,C)$ be as in \eqref{vNboundedness}. Fix $M \gg1$. For $N$ dyadic, let  
$$ F_N:= \{u_0 \, :\, N_0(u_0, \exp(M)) \le N \}. $$
By definition of $N_0$, we have that $\lim_{N \to \infty} \rho_{s,\gamma}(F_N) = \|\rho_{s,\gamma}\|_{TV}.$ Moreover, by \eqref{vboundedness} and \eqref{VVgrowth}, we have that 
\begin{align*}
&\rho_{s,\gamma}( \{ \|\VV(u_0)\|_{L_{t}^\infty([0,1];L_{x}^2)} \le \exp(M)+1 \}) \\
&\ge \rho_{s,\gamma}( \{ \|\VV^N(u_0)\|_{L_{t}^\infty([0,1];L_{x}^2)} \le \exp(M)+1 \} \cap F_N) \\
&\ge \rho_{s,\gamma}( \{ \|\VV^N(u_0)\|_{L_{t}^\infty([0,1];L_{x}^2)} \le \exp(M)+1 \})  - \rho_{s,\gamma}(F_N^c) \\
&\ge \|\rho_{s,\gamma}\|_{\text{TV}}- \exp(-c (M+1)^\beta) - \rho_{s,\gamma}(F_N^c).
\end{align*}
By taking limits as $N \to \infty$, we obtain that
\begin{equation}
\rho_{s,\gamma}( \{ \|\VV(u_0)\|_{L_{t}^\infty([0,1];L_{x}^2)} \le \exp(M)+1 \}) \ge \|\rho_{s,\gamma}\|_{\text{TV}}- \exp(-c (M+1)^\beta). 
\end{equation}
In particular, by sending $M \to \infty$, we obtain that $\|\VV(u_0)\|_{L_{t}^\infty([0,1];L_{x}^2)} < \infty$ $\mu_s$-a.s. 

We now move to showing quasi-invariance of $\mu_s$ (up to time 1). Fix $0 \le t \le 1$. Let $f_t^N$ be the density of $(\Phi_t^N)_\#\rho_{s,\gamma}$, i.e.\ the function defined in \eqref{densityformula}. By \eqref{lpdensity}, we have that $f_t^N$ is uniformly bounded in $L^p(\rho_{s,\gamma})$ for any $ p < \infty$. Therefore, up to subsequences, we have that $f_t^N$ has a weak limit $f_t$ in $L^2(\rho_{s,\gamma})$ as $N \to \infty$. Let $F: H^{-1} \to \R$ be bounded and continuous. By \eqref{vconvergence}, and dominated convergence, we obtain that 
\begin{equation} \label{compatibility}
\begin{aligned}
\int F(u_0) d (\Phi_t)_\#\rho_{s,\gamma}(u_0) &= \int F(\Phi_t(u_0)) d\rho_{s,\gamma}(u_0) \\
&= \lim_{N \to \infty}  \int F(\Phi_t^N(u_0)) d\rho_{s,\gamma}(u_0) \\
&= \lim_{N \to \infty}  \int F(u_0) f_t^N(u_0) d\rho_{s,\gamma}(u_0) \\
&= \int F(u_0) f_t(u_0) d\rho_{s,\gamma}(u_0).
\end{aligned}
\end{equation}
In particular, we obtain that 
$$ (\Phi_t)_\#\rho_{s,\gamma} = f_t d\rho_{s,\gamma},$$
so the measure $\rho_{s,\gamma}$ is quasi-invariant up to $t=1$ (and consequently, also the measure $\mu_s$). 

In order to obtain quasi-invariance for every positive time $t$, we simply iterate this process. First of all, since $(\Phi_1)_\# \rho_{s,\gamma} \ll \rho_{s,\gamma}$, we have that 
\begin{align}
\|\VV(u_0)\|_{L_{t}^\infty([0,1];L_{x}^2)} < \infty\quad (\Phi_1)_\#\rho_{s,\gamma}\text{-a.s.} \Rightarrow \|\VV(\Phi_1(u_0))\|_{L_{t}^\infty([0,1];L_{x}^2)} < \infty\quad \rho_{s,\gamma}\text{-a.s.},  \label{bd01}
\end{align}
and in the same way, 
$$ \lim_{N \to \infty} \|\VV(\Phi_1(u_0)) - \VV_N(\Phi_1(u_0)) \|_{L_{t}^\infty([0,1];L_{x}^2)} = 0 \quad\rho_{s,\gamma}\text{-a.s.} $$
For $1<t\leq 2$, we then define
\begin{align*}
\VV(u_0)(t) = \VV(\Phi_{1}(u_0))(t-1)+S(t-1)\VV(\Phi_1(u_0)),
\end{align*}
and so, by \eqref{bd01}, we have 
\begin{align*}
\| \VV(u_0) \|_{L^{\infty}([0,2];L^2_{x})} <\infty \quad \rho_{s,\gamma}\text{-a.s.}
\end{align*} Then, by Proposition~\ref{PROP:LWPN}, for $N \ge N_3(u_0) \gg 1$, we can extend $\VV_{N}(u_0)$ to times $0\leq t\leq 2$, and by \eqref{VViteration} and the analogous property in the interval $[0,1]$, we obtain that 
$$ \lim_{N \to \infty} \|\VV(u_0) - \VV_N(u_0) \|_{L_{t}^\infty([0,2];L_{x}^2)} = 0 \quad\rho_{s,\gamma}\text{-a.s.} $$ 
By taking a weak limit of $f_t^N$ for $0 \le t \le N$, and repeating the chain of equalities in \eqref{compatibility}, this shows that 
$$ (\Phi_{t})_\#\rho_{s,\gamma} \ll \rho_{s,\gamma} \quad \text{ for } 0 \le t \le 2.$$
Proceeding inductively, we obtain both global well-posedness of \eqref{alphaNLS} and quasi-invariance of $\mu_s$ for every time $t > 0$. By the property 
$$ \Phi_{-t}(u_0) = \cj{\Phi_t(\cj{u_0})}, $$
and the fact that both $\mu_s$ and $\rho_{s,\gamma}$ are invariant for complex conjugation, we derive the same result holds for every $t < 0$.

Finally, we move to showing the growth estimate \eqref{sobolevgrowth}. We want to apply the result of Theorem \ref{bima} to the flow $\Phi_t$, the measure $\rho_{s,\g}$, and 
$$ r(u) := \inf_{0\le t_0 < 1, u_0, v} \{ (1-t_0)^{-1} + \|\Xi(u_0)\|_{\mathcal X^{0,q}} + \|v\|_{L^2}: u = S(t_0)u_0 + v \},$$
where $\|\Xi(u_0)\|_{\mathcal X^{0,q}}$ is defined in \eqref{normxidef}. We clearly have that $r(u) \ge \|u\|_{H^{s-\frac12-}}$.
Notice that the definition of $\|\Xi(u_0)\|_{\mathcal X^{0,q}}$ is invariant with respect to replacing $u_0$ with $S(t)u_0$, 
with the exception of the term depending on $X^{(3)}$. Denoting $X^{(3)} = X^{(3)}[u_0]$ to make the dependency on $u_0$ explicit, by \eqref{X3def} we have that 
$$X^{(3)}[S(t_0)u_0](t) = X^{(3)}[u_0](t+t_0) -  S(t_0)X^{(3)}[u_0](t_0).$$
Therefore, by \eqref{tau0def} and \eqref{tau0bound}, for some constant $\eps_0 > 0$ we have that
\begin{equation*}
\| \VV(S(t_0)u_0 + v) (t) \|_{L^2} \les \|v\|_{L^2} \quad \text{ if } t \le \eps_0\big(1 + \|v\|_{L^2}^{\frac{4\alpha}{2\alpha-1}} + \|\Xi(u_0)\|_{\mathcal X^{0,q}}^{C_\eps}\big)^{-1}\wedge (1-t_0).
\end{equation*}
Therefore, if $u = S(t_0)u_0 + v$, on the interval 
$$t \in \Big[0, \eps_0\big(1 + \|v\|_{L^2}^{\frac{4\alpha}{2\alpha-1}} + \|\Xi(u_0)\|_{\mathcal X^{0,q}}^{C_\eps}\big)^{-1}\wedge \tfrac{(1-t_0)}2\Big],$$ 
we have that 
$$ \Phi_{t}(u) = S(t+t_0)u_0 + S(t)v + \VV(S(t_0)u_0 + v) (t),  $$
and so 
$$ r(\Phi_t(u)) \les (1-t_0)^{-1} + \|\Xi(u_0)\|_{\mathcal X^{0,q}} + C\|v\|_{L^2}.$$
In particular, noticing that 
$$ \eps_0\big(1 + \|v\|_{L^2}^{\frac{4\alpha}{2\alpha-1}} + \|\Xi(u_0)\|_{\mathcal X^{0,q}}^{C_\eps}\big)^{-1}\wedge \tfrac{(1-t_0)}2 \gtrsim \big((1-t_0)^{-1} + \|\Xi(u_0)\|_{\mathcal X^{0,q}} + C\|v\|_{L^2}\big)^{-\kappa}$$
for $\kappa$ big enough, we have that \eqref{bimai} is satisfied. Therefore, recalling that $r(u) \ge \|u\|_{H^{s-\frac12-}}$, if we check that \eqref{bimaii}, \eqref{bimaiii} hold with $c(t) \les \jb{t}^A$, we obtain \eqref{sobolevgrowth} as a direct consequence of \eqref{gwthestimate}. The fact that \eqref{bimaiii} holds with $c(t) \les \jb{t}^A$ follows from \eqref{lpdensity}, together with the weak convergence in $L^p$ of the density $f_t^N$ to $f_t$. Therefore, we just need to check exponential integrability of the quantity $r(u)$. By the Bou\'e-Dupuis formula \eqref{P4},\footnote{We are applying the formula without the Fourier truncation $P_N$, so in order to be rigorous, we should apply the formula to $r(P_Nu)$ and then take a limit as $N \to \infty$. We omit the details, which are straightforward.} \eqref{l2integrability}, \eqref{xiintegrability} and the definition of $r$, for some constant $C \in \R$, we have that 
\begin{align*}
&-\log \int \exp\Big(r(u)\Big) d\rho_{s,\gamma}(u)\\
&\le \E\Big[ \sup_{V_0 \in H^s} \Big(r(Y + V_0) - \Big| \int :(Y+V_0)^2: \Big|^\gamma - \frac 12 \|V_0\|_{H^s}^2\Big)\Big]\\
&\le \E\Big[ \sup_{V_0 \in H^s} \Big(1 + \|\Xi(Y)\|_{\mathcal X^{0,q}} + \|V_0\|_{L^2} + C_\gamma \Big| \int :Y^2:\Big|^\gamma + C_\gamma \|Y\|_{H^{s-\frac12-}}^\kappa\\
&\phantom{\le \sup_{V_0 \in H^s} \E\Big[]} - \frac 12 \|V_0\|_{L^2}^{2\gamma} - \frac 14 \|V_0\|_{H^s}^2\Big)\Big] \\
&\le  \E\Big[\sup_{V_0 \in H^s} - \frac 12 \|V_0\|_{L^2}^{2\gamma} + \|V_0\|_{L^2}\Big] + C \\
&< \infty.
\end{align*}
Therefore, we have \eqref{bimaii} as well, and \eqref{gwthestimate} gives us that 
\begin{equation*}
\| \Phi_t(u)\|_{H^{s-\frac12-}} \le r(\Phi_t(u)) \les_u \jb{t}^A
\end{equation*}
for $\mu_s$ a.e.\ $u$.
\end{proof}

\end{document}